\documentclass[11pt,reqno]{amsart}
\usepackage{fullpage}
\thanks{The first named author's work was supported by the German Research Foundation (DFG) grant no 465189426;
it was carried out as a postdoctoral researcher at Saarland University, during the tenure of an ERCIM “Alain
Bensoussan” Fellowship Programme at NTNU Trondheim, and as a guest researcher at Saarland University in the scope
of the SFB-TRR 195.
}
\thanks{
This research was made possible through financial support from the OeAD-GmbH (Austria’s Agency for Education and Internationalisation) under the Ernst-Mach Grant ASEA-UNINET, which funded the PhD studies of the second named author.}

\newcommand{\orcidlogo}{{\includegraphics[width=\fontcharht\font`l]{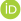}}}

\newcommand{\Addresses}{% additional braces for segregating \footnotesize
\setlength{\parindent}{0pt}{
\bigskip
\footnotesize

    Malte Gerhold \par\nopagebreak
    \textsc{Department of Mathematics, Saarland University, D-66123 Saarbrücken, Germany}\par\nopagebreak
    \textit{E-mail address}: \texttt{gerhold@math.uni-sb.de}\par\nopagebreak
    \href{https://orcid.org/0000-0003-4029-1108}{ 
   \orcidlogo\,0000-0003-4029-1108}
   
}

{
\bigskip
\footnotesize

    Arvin Lamando\par\nopagebreak

     \textsc{Acoustics Research Institute, Austrian Academy of Sciences, Dominikanerbastei 16,
    1010 Vienna, Austria} \par\nopagebreak
     
    \textsc{Faculty of Mathematics, University of Vienna, Oskar-Morgenstern-Platz 1, 1090 Vienna, Austria}\par\nopagebreak

     \textsc{Institute of Mathematics, University of the Philippines, Diliman, 1101 Quezon City, Philippines}\par\nopagebreak
    \textit{E-mail address}: 
    \texttt{alamando@math.upd.edu.ph; arvin.lamando@univie.ac.at}
    \par\nopagebreak
    \href{https://orcid.org/0009-0001-2969-4857}{ 
   \orcidlogo\,0009-0001-2969-4857}
}

{
\bigskip
\footnotesize

    Franz Luef \par\nopagebreak
    \textsc{Department of Mathematical Sciences, NTNU Trondheim, Alfred Getz' vei 1, 7491 Trondheim, Norway}\par\nopagebreak
    \textit{E-mail address}: \texttt{franz.luef@ntnu.no}
    \par\nopagebreak
    \href{https://orcid.org/0000-0001-7413-8350}{ 
   \orcidlogo\,0000-0001-7413-8350}
}}
\title{Linear Deformations of Heisenberg modules and Gabor frames}
\author{Malte Gerhold, Arvin Lamando, and Franz Luef}
\date{}
\newlength{\bibitemsep}
\newlength{\bibparskip}
\let\oldthebibliography\thebibliography
\renewcommand\thebibliography[1]{%
  \oldthebibliography{#1}%
  \setlength{\parskip}{\bibitemsep}%
  \setlength{\itemsep}{\bibparskip}%
}
\usepackage[T1]{fontenc}
\usepackage{hyphenat}
\usepackage{listings}
\usepackage{enumitem}
\usepackage{setspace}
\usepackage[dvipsnames]{xcolor}
\usepackage[hypertexnames=false, colorlinks, citecolor=BrickRed,linkcolor=MidnightBlue, urlcolor=BrickRed]{hyperref}
\usepackage{acro}
\usepackage{amsthm}
\usepackage{graphicx}
\usepackage{subcaption}
\usepackage{pifont}
\usepackage{cancel}
\usepackage{leftindex}
\usepackage[normalem]{ulem}
\usepackage{tikz}
\usepackage{amsmath,amssymb, mathtools}
\usepackage{mathrsfs}
\usepackage{multicol}
\usepackage{tikz}
\usetikzlibrary{matrix,arrows,decorations.pathmorphing}
\usepackage{tikz-cd}
\tikzset{commutative diagrams/.cd}
\usepackage{tikz}
\usetikzlibrary{matrix,arrows,decorations.pathmorphing}
\usepackage{tikz-cd}
\tikzset{commutative diagrams/.cd}
\usepackage{tikz}
\usetikzlibrary{
  calc,
  decorations.pathmorphing,
  shapes,
  arrows.meta
}

\newtheorem{exmp}{Example}[section]

\newtheoremstyle{mystyle}
{\topsep}% measure of space to leave above the theorem. E.g.: 3pt
{\topsep}% measure of space to leave below the theorem. E.g.: 3pt
{\itshape}% name of font to use in the body of the theorem
{0pt}% measure of space to indent
{\bfseries\scshape}% name of head font
{.}% punctuation between head and body
{ }% space after theorem head; " " = normal interword space
{\thmname{#1}\thmnumber{ #2}\thmnote{ (#3)}}

\newtheorem{proposition}[exmp]{Proposition}

\newtheorem{observation}[exmp]{Observation}

\newtheorem{lemma}[exmp]{Lemma}

\newtheorem{theorem}[exmp]{Theorem}

\newtheorem{corollary}[exmp]{Corollary}

\theoremstyle{definition}
\newtheorem{definition}[exmp]{Definition}

\newtheorem{remark}[exmp]{Remark}

\newtheorem{example}[exmp]{Example}

\theoremstyle{remark}
\newtheorem*{notation}{Notation}

\theoremstyle{definition}

\def\hmath$#1${\texorpdfstring{{\rmfamily\textit{#1}}}{#1}}

\newcommand{\ie}{i.e.\ }
\newcommand{\eg}{e.g.\ }
\newcommand{\mycomment}[1]{}
\newcommand{\op}[1]{\operatorname{#1}}

\newcommand{\lin}[3]{\leftindex_{#1}{\left \langle #2,#3 \right\rangle}}
\newcommand{\rin}[3]{\left\langle #2,#3 \right\rangle_{#1}}

\newcommand{\comment}[1]{}

\renewcommand{\d}[1]{\ensuremath{\operatorname{d}\!{#1}}}

\begin{document}
\begin{abstract}
Heisenberg modules over noncommutative tori may also be viewed as Gabor frames. Building on this fact, we relate to deformations of noncommutative tori a bundle of Banach spaces induced by Heisenberg modules. The construction of this bundle of Banach spaces rests on deformation results of Gabor frames with windows in Feichtinger's algebra due to Feichtinger and Kaiblinger. We extend some of these results to Heisenberg modules, \eg we establish an analog of the results by Feichtinger-Kaiblinger and a Balian-Low theorem. Finally, we extend our results to several generators on the bundle of Heisenberg modules and show that they provide a generalized Fell's condition on the bundle of noncommutative tori.
\end{abstract}

\maketitle
\section{Introduction}
Heisenberg modules have first been introduced by Connes in \cite{co80} in his work on differential geometry for noncommutative $C^*$-algebras and since then they have found applications in number theory, physics and engineering, more concretely in time-frequency analysis. The link between time-frequency analysis and Heisenberg modules was reported for the first time in \cite{Lu07}, which was based on Rieffel's description of Heisenberg modules in \cite{Ri88}. In this paper we discuss the basic question if "near-by" noncommutative tori may be equipped with "near-by" Heisenberg modules. We address this by using the notion of Banach bundles and the fact that Heisenberg modules over noncommutative tori are just Gabor frames for lattices in phase space and the seminal result by Feichtinger-Kaiblinger on linear deformations of Gabor frames \cite{FeKa04}. In addition, we establish that the Feichtinger-Kaiblinger approach also works for functions in Heisenberg modules, which is based on the work of Austad and Enstad in \cite{AuEn20}, where they show that the Janssen representation holds for functions in Heisenberg modules. Another central part of our proof of the main result is a dilation theorem for noncommutative tori \cite{GPSS21}, which allows us to show the continuity property of our bundle of Heisenberg modules. 

Gabor frames provide representations of functions in terms of a family of functions, where the latter is generated by a fixed function, often called the Gabor atom, and countably many time-frequency shifts of the Gabor atom. Dennis Gabor studied systems of this type in \cite{Ga47} with the Gaussian as Gabor atom. 

One views Gabor systems from the perspective of frame theory and it has been realized that the well-behaved Gabor systems are those which are frames for $L^2(\mathbb{R}^d)$: Given $g\in L^2(\mathbb{R}^d)$ and a set of points $\{(x_j,\omega_j)\}_{j\in J}$ in $\mathbb{R}^{2d}$. Then we consider the family $\{\pi(x_j,\omega_j)g\}_{j\in J}$, where $\pi(x,\omega)g(t)=e^{2\pi it\omega}g(t-x)$ denotes a time-frequency shift by $(x,\omega)$. We call $\mathcal{G}(g,\{x_j,\omega_j\}_{j\in J})=\{\pi(x_j,\omega_j)g\}_{j\in J}$ a Gabor frame for $L^2(\mathbb{R}^d)$ if there exists constants $A,B>0$ such that
\begin{align}\label{the-frame-q}
 A\|f\|_2^2\le\sum_{j\in J}|\langle f,\pi(x_j,\omega_j)g\rangle|^2 \le B\|f\|^2_2\qquad\text{for all}~~f\in L^2(\mathbb{R}^d).   
\end{align}
 
 Note that $\pi(x,\omega)=M_\omega T_x$, where $M_\omega g(y)=e^{2\pi iy\cdot\omega}g(t)$ and $T_xg(y)=g(y-x)$ denote the modulation and translation operator, respectively. One calls a Gabor system $\mathcal{G}(g,\{x_j,\omega_j\}_{j\in J})$ a Bessel family if in the frame definition just the upper bound holds, we refer to $B$ as a Bessel bound for $\mathcal{G}(g,\{x_j,\omega_j\}_{j\in J})$. 

In case that the point set $\{x_j,\omega_j\}_{j\in J}$ has some additional structure, then the associated Gabor system $\mathcal{G}(g,\{x_j,\omega_j\}_{j\in J})$ inherits some extra features. In the following we are going to study Gabor systems, where the point set is a lattice $\Lambda$ in $\mathbb{R}^{2d}$. Note that a lattice in $\mathbb{R}^{2d}$ is of the form $L\mathbb{Z}^{2d}$ for some invertible 2d$\times$2d-matrix $L$.   

There is a class of ``nice" Gabor systems $\mathcal{G}(g,\Lambda)$, where the Gabor atom is an element of the Feichtinger algebra $M^1(\mathbb{R}^d)$, a Banach space of continuous functions, which is densely and continuously embedded into  $L^2(\mathbb{R}^d)$. A Gabor system $\mathcal{G}(g,\Lambda)$ with $g\in M^1(\mathbb{R}^d)$ is \emph{always} a Bessel family for \emph{any} lattice $\Lambda$ \cite[Theorem 3.3.1]{FeZi98}. A myriad of other properties of $M^1(\mathbb{R}^d)$ make it an ideal function space for time-frequency analysis, but we shall focus on Feichtinger and Kaiblinger's stability result \cite[Thm. 3.8]{FeKa04} that allows deformations of the Gabor atom and of the lattice:
\begin{theorem}\label{feichkaib}
    Let $g_0\in M^1(\mathbb{R}^{d})$ and $L_0\in \op{GL}_{2d}(\mathbb{R})$ such that $\mathcal{G}(g_0, L_0\mathbb{Z}^{2d})$ is a Gabor frame, then there exist open sets $U_0\ni g_0$ in $M^1(\mathbb{R}^d)$ and $V_0\ni L_0$ in $\op{GL}_{2d}(\mathbb{R})$ such that $\mathcal{G}(g,L\mathbb{Z}^{2d})$ is a Gabor frame for all $(g,L)\in U_0\times V_0.$
\end{theorem}
\noindent This type of deformation is termed \emph{linear} because the lattice $\Lambda_0=L_0\mathbb Z^{2d}$ (or its generating matrix $L_0$) is perturbed by a linear transformation, namely multiplication with $LL_0^{-1}\in\op{GL}_{2d}$. Note that Gabor stability has been considered in terms of generalized deformations. In \cite{GrOrJo15}, a generalized notion of `deformation` of non-uniform sampling sets (\eg $\Lambda \subseteq \mathbb{R}^{2d}$ possibly not a lattice), which encompasses linear deformations on lattices as a special case, was studied. On the other hand, linear deformations for lattices in the abstract time-frequency plane $G\times\widehat{G}$ for $G$ a locally compact abelian group were studied in \cite{EnJaLuOm22} by considering the Braconnier topology on the automorphism group of $G\times \widehat{G}$.

In this paper, our main concern is to extend Theorem \ref{feichkaib} to the so-called Heisenberg modules \cite{co80}, where we follow the construction by Rieffel in \cite{Ri88}, mainly to study finitely generated projective modules for noncommutative tori. Originally, the Heisenberg modules were constructed by completing the Schwartz space $\mathcal{S}(\mathbb{R}^d)$ with respect to a suitable norm. The observation that we can replace $\mathcal{S}(\mathbb{R}^d)$ by Feichtinger's algebra $M^1(\mathbb{R}^d)$ in their construction, and that they are connected to Gabor analysis was first made by Luef in \cite{Lu07}. Further probing in this direction can be found in \cite{Lu09,JaLu21}. Recently, it was shown by Austad and Enstad in \cite{AuEn20} that the Heisenberg modules can be seen as function spaces, in that they can be constructed by completing Feichtinger's algebra $M^1(\mathbb{R}^d)$ in the following norm:
\begin{align*}
\|f\|_{\mathcal{E}_L(\mathbb{R}^d)}:=\op{inf}\{C^{\frac{1}{2}}:C\text{ is a Bessel bound for $\mathcal{G}(f,L\mathbb{Z}^{2d})$}\},
\end{align*}
for each $f\in M^1(\mathbb{R}^d)$ and $L\in \op{GL}_{2d}(\mathbb{R}).$ In this description, it is now apparent that the Heisenberg module, which we denote by $\mathcal{E}_L(\mathbb{R}^d)$, depends on the generating matrix $L\in \operatorname{GL}_{2d}(\mathbb{R})$.
Unfortunately, we are now also met with an immediate ambiguity trying to implement our idea of extending Theorem \ref{feichkaib} to the Heisenberg modules. 
Suppose $g_{L_0} \in \mathcal{E}_{L_0}(\mathbb{R}^d)$ such that $\mathcal{G}(g_{L_0},L_0\mathbb{Z}^{2d})$ is a frame, it is not clear what we mean for values $g$ such that $g$ is in some Heisenberg module `near' $g_{L_0}$, where we have that $\mathcal{G}(g,L)$ are all frames in $L^2(\mathbb{R}^d)$ whenever $L$ is close to $L_0$. 
In general, because the construction of a Heisenberg module depends on the generating matrix $L$, we should expect the `nearby' $g$'s to vary continuously with the matrices $L\in \op{GL}_{2d}(\mathbb{R})$ as well. 
With this in mind, what we would like to do is to look at families of the Heisenberg modules $\{\mathcal{E}_{L}(\mathbb{R}^d)\}_{L}$ for $L$ `near' $L_0$ such that there also are families of elements $g_{L}\in \mathcal{E}_L(\mathbb{R}^d)$ that are `sufficiently near' $g_{L_0}\in \mathcal{E}_{L_0}(\mathbb{R}^d)$. 
What we mean by `nearness' between different Banach spaces is not clear at this point, but it can be made rigorous by the language of bundles. To make it precise, what we want is to be able to define a Banach bundle 
\[\kappa\colon \mathcal{E} \to \op{GL}_{2d}(\mathbb{R}), \quad \displaystyle\mathcal{E} = \bigsqcup_{L\in \op{GL}_{2d}(\mathbb{R})}\mathcal{E}_{L}(\mathbb{R}^d)\] 
such that, for each $L\in \op{GL}_{2d}(\mathbb{R})$, we have the fiber $\kappa^{-1}(L) = \mathcal{E}_{L}(\mathbb{R}^d)$ and such that the (constant) sections associated with $f\in M^1(\mathbb R^d)\subset \mathcal{E}_L(\mathbb R^d)$ are continuous. The construction of such a Banach bundle is lifted directly from corresponding deformation results on higher dimensional noncommutative tori (in this case, the noncommutative tori generated by the unitaries $T_x,M_\omega$ with $(x,\omega)\in\Lambda)$. 
Once $\kappa\colon \mathcal{E}\to \op{GL}_{2d}(\mathbb{R})$ has been constructed, we now have access to its continuous sections $\Gamma(\kappa)$, which are `continuous' maps $\Upsilon\colon \op{GL}_{2d}(\mathbb{R}^d)\to \mathcal{E}$ with the following property: $\Upsilon(L)\in \mathcal{E}_L(\mathbb{R}^d)$ for all $L\in \op{GL}_{2d}(\mathbb{R})$. 
The Gabor-stability theorem for the Heisenberg modules can now be stated in terms of these continuous sections, which is one of our main results (see Theorem \ref{stability-heis}). Moreover, we shall study the algebraic properties of the space $\Gamma_0(\kappa)$, the space of continuous section of $\kappa$ that vanish at infinity. We shall show that we can equip it with a structure which shows that it is an imprimitivity bimodule over certain $C^*$-algebras.

Let us mention related results concerning the convergence of Heisenberg modules in different settings. Latremoliere has shown that Heisenberg modules over noncommutative 2-tori convergence with respect to the modular Gromov-Hausdorff propinquity \cite{La20}, which provides stability results of a different flavor but quite close in spirit and it might be worth clarifying the exact relation between the results in \cite{La20} and our main statements.  

The paper is structured as follows. 
Section \ref{section: prelims} provides a detailed overview of the preliminary knowledge necessary for this paper, covering topics that may seem unrelated at first glance. 
Section \ref{subsec:framegab} touches upon the basics of Gabor analysis using the framework of frame theory. 
Section \ref{subsec:hilbcmodule} is a rundown on the basic notions of Hilbert $C^*$-modules and imprimitivity bimodules. 
This is also where we shall discuss the Heisenberg modules, which we present concretely as a subspace of $L^2(\mathbb{R}^d)$. 
We end this section with a discussion of the extension of the well-known Janssen's representation and Wexler-Raz relations to the Heisenberg modules. 
In Section \ref{subsec: bbunldles}, we finally give a formal definition of a Banach bundle, including some of its basic properties. 
The most relevant result of this section is the `canonical' Banach bundle construction theorem of Fell (Theorem \ref{bbundleconstruction}). 
Lastly, in Section \ref{subsec:nctori} we give a detailed discussion of the noncommutative tori $\mathcal A_\Theta$ (for skew-symmetric $n{\times}n$-matrices $\Theta$), the universal $C^*$-algebras generated by $n$ unitaries $U_i(\Theta)$ ($i=1,\ldots,n$) such that $U_i(\Theta),U_j(\Theta)$ commute up to the phase factor $e^{2\pi i \Theta_{ij}}$; here we also incorporated computations of Gjertsen \cite{Gj23}, which yield an explicit description of noncommutative tori as twisted group $C^*$-algebras, in particular, we get that the map $\mathbb Z^n \to \mathcal A_\Theta$, via $k\mapsto U_1(\Theta)^{k_1}\cdots U_n(\Theta)^{k_n}$ extends to a $C^*$-isomorphism $\pi_{\Theta}\colon C^*(\mathbb{Z}^n,\zeta_{\Theta^{\op{low}}})\to \mathcal{A}_{\Theta}$ for a certain cocycle $\zeta_{\Theta^{\op{low}}}$.

Section \ref{section:results} contains all of our main results. In Section \ref{subsec:general-nctori-deform}, we review various deformation results for noncommutative tori in the literature. We end with a proof of our basic result involving the $\frac{1}{2}$-H\"older continuity of the map $\Theta\mapsto \|\pi_{\Theta}(\mathbf a)\|$ for $\mathbf a$ from the weighted $\ell^1$-space $\ell^1_\nu$, which may be identified with a common dense subspace of the different noncommutative tori; this will be the basis of all our subsequent constructions. In Section \ref{subsec: nctori by lattices}, we recognize that each lattice $L\mathbb{Z}^{2d}$ has an associated noncommutative torus $\mathcal{A}_{\Theta_L}\cong C^*(\{\pi(Le_i)\}_{i=1}^{2d})$, and we show in Lemma \ref{lemma: tf-cocycle-cohom-canoncocycle}, that the $2$-cocycle that naturally appears in the commutation-relations of the time-frequency shifts is cohomologous to the $2$-cocycle $\zeta_{\Theta_L^{\op{low}}}$ of the twisted $C^*$-algebra in the domain of the representation $\pi_{\Theta_L}\colon C^*(\mathbb{Z}^n,\zeta_{\Theta_L^{\op{low}}})\to \mathcal{A}_{\Theta_L}$. This connection will allow us to finally prove a deformation theorem, showing that the coefficient $C^*$-algebras of the Heisenberg modules generate a $C^*$-bundle (Theorem \ref{fundamental-c-bundle}) with the lattice generators $\op{GL}_{2d}(\mathbb{R})$ as the base space. Then in Section \ref{algbundle} we lift the main construction of Section \ref{subsec: nctori by lattices} to a proof of the existence of a Banach bundle of Heisenberg modules over $\op{GL}_{2d}(\mathbb{R})$ (Theorem \ref{main-construction}) such that every constant section associated with some $f\in M^1(\mathbb R^d)$ is continuous (Proposition \ref{for-refinement-of-sections}). We also show various refinement results concerning the continuous sections of the said bundle. In Corollary \ref{corollary: imp-bim-from-bundles}, we prove that the associated Banach space of continuous sections of our Banach bundle of Heisenberg modules forms a natural imprimitivity bimodule. Finally in Section \ref{subsec: gabanalysis app}, we now capitalize in our efforts and use the constructed Banach bundle to prove Gabor-stability of the Heisenberg modules in Theorem \ref{stability-heis}. We also include a Balian-low type result as corollary. 
In Section \ref{subsec: modulus}, we show that constant sections associated with $f$ in some weighted versions of Feichtinger's algebra (that is if we add some regularity condition) are not only continuous, but actually $\frac{1}{2}$-Hölder continuous. Finally, in Section \ref{subsec: continuous-trace}, we extend the Balian-Low theorem to several generators on the bundle of Heisenberg modules. As a corollary, we obtain in Theorem \ref{theorem: local-projections} a non-trivial generalization of Fell's condition on our bundle of noncommutative tori. We end with a discussion of a possible generalization of continuous-trace $C^*$-algebras to general $C^*$-bundles via Hilbert $C^*$-modules. 
\section{Preliminaries}\label{section: prelims}
In this section, we shall collect all the material relevant for the discussion of our main results.
We assume that the reader is familiar with the basic theory of $C^*$-algebras, including the notion of universal $C^*$-algebras (see \cite{Bl85} for an introduction).
Throughout the article, for any $n\in \mathbb{N}_{>0}$, we implicitly identify elements of $\mathbb{R}^n$ as either a column vector or a row vector. The standard inner product of $x,y\in\mathbb R^d$ is denoted by $x\cdot y$.  In formulae involving matrices, the standard interpretation of vectors in $\mathbb R^d$ is as a column vector, row vectors are indicated by a transpose, \eg $x^T J y=x\cdot (J y)$ for $x,y\in\mathbb R^d$ and $J$ a $d{\times}d$-matrix.
Hilbert space inner products are linear in the first argument and are denoted by angle brackets.
For a Hilbert space $H$, we denote the space of all bounded and linear operators by $\mathcal{L}(H)$.
We shall denote by $\|\cdot\|$ the operator norm on $\mathcal{L}(H)$, and we extend the use of this notation to matrices.
Hence for any matrix $L$, we choose $\|L\|$ to mean the operator norm.
Furthermore, sequences are denoted by $\mathbf{a}=\{a_i\}_{i\in I}$, $\mathbf{b}=\{b_i\}_{i\in I}$, etc. Finally, if $S$ is a family of elements of a $C^*$-algebra $A$, we denote the $C^*$-algebra \textbf{generated} by $S$ with $C^*(S)$, and we have by definition that $C^*(S)\subseteq A$.

\subsection{Frame Theory and Gabor Analysis}\label{subsec:framegab}

We shall recall a few notions from frame theory and apply them to Gabor analysis, see \cite{Gr13, Ch03} for a more in-depth discussion.

\begin{definition}
    Let $\Psi:=\{f_i\}_{i\in I}$ be a countable family of vectors in a Hilbert space $H$ for which there exist constants $A,B>0$ such that for all $f\in H:$
    \begin{align}\label{frame-ineq}
        A \|f\|_2^2 \leq \sum_{i\in I}|\rin{H}{f}{f_i}|^2 \leq B\|f\|_2^2.
    \end{align}
    In this case, $\Psi$ is called a \textbf{frame} for $H$. If only the rightmost inequality is true in \eqref{frame-ineq}, then we say that $\Psi$ is a \textbf{Bessel family} in $H$, with \textbf{Bessel bound} $B$.
\end{definition}

Given a Bessel family $\Psi = \{f_i\}_{i\in I}$, the associated \textbf{analysis operator}, defined via $C_{\Psi}(f) = \{\rin{H}{f}{f_i}\}_{i\in I}$ for each $f\in H$, is a bounded linear operator $C_{\Psi}\colon H\to \ell^2(I)$. 
The adjoint map $C_{\Psi}^*\colon \ell^2(I)\to H$ is called the \textbf{synthesis operator}, and one has that $C_{\Psi}^*(\mathbf{x}) = \sum_{i\in I}x_if_i $ for all $\mathbf{x}=\{x_i\}_{i\in I}\in \ell^2(I)$. 
Finally, the associated \textbf{frame operator}, denoted by $S_{\Psi}$, is given by the composition $S_{\Psi}:= C_{\Psi}^*C_{\Psi}\colon H\to H$. An elementary computation gives that $S_{\Psi}f:=\sum_{i\in I}\rin{H}{f}{f_i}f_i$ for all $f\in H$. 
By the definition of $S_{\Psi}$, the frame operator is a self-adjoint and (semi-)positive operator in $\mathcal{L}(H)$. 

Let us state a well-known characterization of frames $\Psi$ in terms of $S_{\Psi}$.
\begin{observation}\label{frame-invertible}
    $\Psi$ is a frame if and only if $S_{\Psi}$ is an invertible operator in $\mathcal{L}(H)$. 
\end{observation}
This observation above has an important consequence: If $\Psi$ is a frame, then we have the following reconstruction formula:
\begin{align}\label{reconstruction-formula}
f = \sum_{i\in I}\rin{H}{f}{S^{-1}_{\Psi}f_i}f_i = \sum_{i\in I}\rin{H}{f}{f_i}S^{-1}_{\Psi}f_i\qquad\text{for all}~~f\in H.
\end{align}
We are also sometimes interested in analyzing and synthesizing using possibly different families of elements in $H$. 
Suppose we have another Bessel family $\Phi = \{h_i\}_{i\in I}\subseteq H$, then we have the \textbf{mixed-frame operator} $S_{\Psi,\Phi}=C_{\Phi}^*C_{\Psi}\colon H \to H$.
It can also be shown that $S_{\Psi,\Phi}f = \sum_{i\in I}\rin{H}{f}{f_i}h_i$ for all $f\in H$. 
If $\Psi$ and $\Phi$ are both frames, then we say that $\Psi$ and $\Phi$ are \textbf{dual frames} whenever $S_{\Psi,\Phi}=\op{Id}_{H}$. 
Note that it follows from the reconstruction formula \eqref{reconstruction-formula} that every frame $\Psi = \{f_i\}_{i\in I}$ has a dual frame given by $\Phi = \{S_{\Psi}^{-1}f_i\}_{i\in I}$, called the \textbf{canonical dual frame}.

Let us review the basics of frame theory relevant for Gabor analysis. 
We consider $\mathbb{R}^d\times \mathbb{R}^d = \mathbb{R}^{2d}$ as the \textbf{time-frequency plane}, aka in physics as phase space. 
For a point $z=(x,\omega)$ in the time-frequency plane $\mathbb{R}^{2d}$, we consider two fundamental unitary operators on $L^2(\mathbb{R}^d)$: The translation operator $T_x$, and the modulation operator $M_{\omega}$, which are defined by 
\begin{align*}
    (T_xf)(y) &= f(y-x);\\
    (M_{\omega}f)(y) &= e^{2\pi i\, y\cdot \omega}f(y).
\end{align*}
Their compositions are called the \textbf{time-frequency shifts}, given by:
\begin{align}\label{TFS}
\pi(x,\omega)=M_{\omega}T_x.
\end{align}
Note that 
\begin{align}\label{eq:mod-trans-commutation}
    (T_x M_\omega f) (y)= e^{2\pi i\, (y-x)\cdot \omega} f(y-x) = e^{-2\pi i\, x\cdot \omega} (M_\omega T_x f)(y),
\end{align} 
which implies 
\begin{align}\label{eq:TFS-commutation}
\pi(z_1)\pi(z_2)=e^{-2\pi i\, (x_1\cdot \omega_2-\omega_1\cdot x_2)} \pi(z_2)\pi(z_1)
=e^{-2\pi i\, z_1^T Jz_2} \pi(z_2)\pi(z_1),
\end{align}
for $z_1=(x_1,\omega_1),\,z_2=(x_2,\omega_2)$ where $J$ denotes the $2d\times 2d$-matrix 
\(J=\begin{pmatrix}
    0& I_d\\-I_d &0
\end{pmatrix}\). 
The matrix $J$ implements the standard symplectic form on $\mathbb{R}^{2d}$.

The \textbf{short-time Fourier transform} (STFT) of $f\in L^2(\mathbb{R}^d)$ with respect to a \textbf{window} $g\in L^2(\mathbb{R}^d)$,  denoted by $\mathcal{V}_g f \in L^2(\mathbb{R}^{2d})$, is given by
\begin{align*}
    \mathcal{V}_g f(x,\omega) := \int_{\mathbb{R}^d}f(t)\overline{g(t-x)}e^{-2\pi i\, t\cdot \omega}\d t = \rin{2}{f}{\pi(x,\omega)g}.
\end{align*}
Gabor analysis aims to construct reconstruction formulas for $f$ in terms of the time-frequency representation given by the STFT, as we shall see later. 

A \textbf{lattice} in $\mathbb{R}^{2d}$ is a discrete subset of $\mathbb{R}^{2d}$ of the form $\Lambda=L\mathbb{Z}^{2d}$ for $L\in \op{GL}_{2d}(\mathbb{R})$, we call $L$ a \textbf{matrix generator} for $\Lambda$. 
For a fixed $g\in L^2(\mathbb{R}^d)$, we define the so-called \textbf{Gabor system} $\mathcal{G}(g,L)$ via:
\begin{align*}
    \mathcal{G}(g,L):=\left\{\pi(Lk)g: k\in \mathbb{Z}^{2d} \right\}\subseteq L^2(\mathbb{R}^d). 
\end{align*}
In other words, $\mathcal{G}(g,L)$ is the family of time-frequency shifts of $g$ coming from points of the lattice $L\mathbb{Z}^{2d}$.

Suppose that $\mathcal{G}(g,L)$ is a Bessel family in $L^2(\mathbb{R}^{d})$. Then the associated analysis operator is denoted by  $C_{g,L}\colon L^2(\mathbb{R}^d)\to \ell^2(\mathbb{Z}^{2d})$, while the associated frame operator is denoted by $S_{g,L}$, and called a \textbf{Gabor frame operator}. 
In general, if there is another Bessel family $\mathcal{G}(h,L)$ for $h\in L^2(\mathbb{R}^d)$ (but the same $L$), we have the \textbf{mixed (Gabor) frame operator} given by $S_{g,h,L} = C^*_{h,L}C_{g,L}$. Note that using our notation: $S_{g,g,L}= S_{g,L}$. In the case that $\mathcal{G}(g,L)$ is a frame for $L^2(\mathbb{R}^d)$, we say that $\mathcal{G}(g,L)$ forms a \textbf{Gabor frame}, equivalently the Gabor frame operator $S_{g,L}$ must be invertible according to Observation \ref{frame-invertible}. If we denote $h:= S^{-1}_{g,L}g\in L^2(\mathbb{R}^d)$, then the reconstruction formula \eqref{reconstruction-formula} reads, for all $f\in L^2(\mathbb{R}^d)$,
\begin{align*}
    f &= \sum_{k\in \mathbb{Z}^{2d}}\rin{2}{f}{S^{-1}_{g,L}\pi(Lk)g}\pi(Lk)g = \sum_{k\in \mathbb{Z}^{2d}}\mathcal{V}_hf(Lk)\pi(Lk)g \\
    &=\sum_{k\in \mathbb{Z}^{2d}}\rin{2}{f}{\pi(Lk)g}S^{-1}_{g,L}\pi(Lk)g = \sum_{k\in \mathbb{Z}^{2d}}\mathcal{V}_g f(Lk)\pi(Lk)h.
\end{align*}
We see that the existence of a Gabor frame $\mathcal{G}(g,L)$ amounts to a representation of $f$ as above, with coefficients given by STFTs with respect to the \textbf{Gabor atom} $g$, or its \textbf{canonical dual atom} $h=S^{-1}_{g,L}g$. 
In general, we call any $h\in L^2(\mathbb{R}^d)$ a \textbf{dual atom} for $\mathcal{G}(g,L)$ if $S_{g,h,L}=\op{Id}_{L^2(\mathbb{R}^d)}$. 

The frame theory of Gabor systems is quite intricate and there are several conditions on $g$ and $L$ that are implied by the frame property of a Gabor system $\mathcal{G}(g,L)$. We start with the fundamental density theorem, see \eg \cite[Theorem 5.6]{JaLe16}: 
\begin{theorem}\label{basic-density}
    If $g\in L^2(\mathbb{R}^d)$ and $L\in \op{GL}_{2d}(\mathbb{R})$ generate a Gabor frame $\mathcal{G}(g,L)$ for $L^2(\mathbb{R}^d)$, then $|\det L| \leq 1$.
\end{theorem}
\noindent The theorem above is a fundamental limit to the sparseness of the lattice if we want a single atom $g$  to span the whole of $L^2(\mathbb{R})$ via time-frequency shifts. 
For a subtler discussion on the existence of possibly finitely many Gabor atoms generating a frame and the sufficient density requirements, see \cite{JaLu21} and \cite{AuJaMaLu20}.

Further additional structure is related to the existence of an ``adjoint'' Gabor system  generated by the so-called \textbf{adjoint lattice}. 
For a matrix generator $L$ of the lattice $\Lambda=L\mathbb{Z}^{2d}$, we consider $L^{\circ}=J(L^{-1})^TJ^T$, the lattice generator of the adjoint lattice
\begin{align*}
L^{\circ}\mathbb{Z}^{2d}=\Lambda^{\circ}:=\left\{(x,\omega)\in \mathbb{R}^{2d}: \pi(x,\omega)\pi(Lk)=\pi(Lk)\pi(x,\omega),\ \text{for all $k\in \mathbb{Z}^{2d}$} \right\},
\end{align*}
which follows from the commutation relation \eqref{eq:TFS-commutation}.
We shall denote by $\delta_m\in \ell^1(\mathbb{Z}^n)$ for $m\in \mathbb{Z}^n$ the sequence defined via:
\begin{align*}
        \delta_m(k) = \begin{cases}
            1, \qquad \text{if $k=m$},\\
            0, \qquad \text{otherwise}.
        \end{cases}
\end{align*}

\begin{theorem}[Wexler-Raz Relations]\label{wexler-raz}
    Let $g,h\in L^2(\mathbb{R}^d)$ and $L\in \op{GL}_{2d}(\mathbb{R})$. The Gabor systems $\mathcal{G}(g,L)$ and $\mathcal{G}(h,L)$ are dual frames if and only if for all $k\in \mathbb{Z}^{2d}$:
    \begin{align*}
        \rin{2}{h}{\pi(L^{\circ}k)g}=|\det L | \cdot  \delta_0(k).
    \end{align*}
\end{theorem}
\noindent The Wexler-Raz relation is a fundamental incarnation of \emph{duality principles} in Gabor analysis, \ie where a frame theoretic statement in $\Lambda$ amounts to some equivalent statement involving the adjoint lattice $\Lambda^{\circ}$.

We now introduce Feichtinger's algebra, denoted by $M^1(\mathbb{R}^d)$. For a detailed account, one may consult the original paper by Feichtinger \cite{Fe81}, or the recent survey by Jakobsen \cite{Ja18}.
\begin{definition}
    Consider the Gaussian $\phi(t) = e^{-\pi \|t\|^2}$ for $t\in \mathbb{R}^d$. We set
    \begin{align*}
        M^1(\mathbb{R}^d):=\left\{f\in L^2(\mathbb{R}^d): \mathcal{V}_{\phi}f\in L^1(\mathbb{R}^{2d}) \right\}.
    \end{align*}
\end{definition}
\noindent Note that $\phi\in M^1(\mathbb{R}^d)$ and, furthermore, $\mathcal{V}_{g}f\in M^1(\mathbb{R}^{2d})$ for any $f,g\in M^1(\mathbb{R}^{d})$ \cite[Proposition 4.5]{Ja18}. Feichtinger's algebra is a Banach space when equipped with the norm $\|f\|_{M^1,\phi}:= \|\mathcal{V}_{\phi}f\|_1$.  
In fact, we can replace $\phi$ with any other $g\in M^1(\mathbb{R}^d)$ and then the norm $\|f \|_{M^1,g}=\|\mathcal{V}_g f\|_1$ defines an equivalent Banach space norm on $M^1(\mathbb{R}^d)$, \ie there exist constants $A,B$ such that $A\|f \|_{M^1,\phi}\leq\|f \|_{M^1,g}\leq B\|f \|_{M^1,\phi}$.
In light of this, unless the window function $\phi$ is explicitly needed, we shall ease our notation for the norm on $M^1(\mathbb{R}^d)$, and denote it by $\|\cdot\|_{M^1}$. 
There are plenty of properties of $M^1(\mathbb{R}^d)$ which are quite useful for time-frequency analysis, however the most relevant for us will be the fact that, for any $L\in \op{GL}_{2d}(\mathbb{R})$, $\mathcal{G}(g,L)$ is always a Bessel family for any $g\in M^1(\mathbb{R}^d)$ \cite[Theorem 3.3.1]{FeZi98}. 
Consequently, this gives for Gabor systems $\mathcal{G}(g,L)$ for $g\in M^1(\mathbb{R}^d)$ that the analysis operator $C_{g,L}$ is bounded, hence the Gabor frame operator $S_{g,L}$ is bounded, too. Furthermore, another instance of the duality theory of Gabor frames can be described for Gabor systems with atoms in $M^1(\mathbb{R}^d)$.
\begin{theorem}
    For all $g,h\in M^1(\mathbb{R}^d)$, Janssen's representation of the mixed Gabor frame operator holds, \ie
    \begin{align}\label{jan-rep}
        S_{g,h,L} = \frac{1}{|\det L|} \sum_{k\in \mathbb{Z}^{2d}}\rin{2}{\pi(L^{\circ}k)h}{g}\pi(L^{\circ}k).
    \end{align}
\end{theorem}
\noindent Janssen's representation is perhaps one of the most important identities in Gabor analysis, for a myriad of reasons. Foremost, it connects the Heisenberg modules of Rieffel \cite{Ri88} to Gabor analysis \cite{Lu09}, and it is a crucial ingredient in Feichtinger and Kaiblinger's proof of the stability of Gabor frames with atoms in $M^1(\mathbb{R}^d)$ under perturbations of lattices and Gabor atoms \cite{FeKa04}. In the following, we shall combine these two facts into stability results for Heisenberg modules.

Before we proceed, we shall discuss a few more properties of Feichtinger's algebra of the time-frequency plane $\mathbb{R}^{2d}$, \ie $M^1(\mathbb{R}^{2d})$. We introduce the following maps, denoted by $\mathcal{F}_2\colon M^1(\mathbb{R}^{2d})\to M^1(\mathbb{R}^{2d})$ and $\tau_{a}\colon M^1(\mathbb{R}^{2d})\to M^1(\mathbb{R}^{2d})$, called the \textbf{partial Fourier transform} and the \textbf{asymmetric coordinate transform} respectively, defined via:
\begin{align*}
    (\mathcal{F}_2F)(x,\omega) &=\int_{\mathbb{R}^d}F(x,t)e^{-2\pi i t\cdot \omega}\d t, \qquad \forall (x,\omega)\in \mathbb{R}^{2d},\\
    (\tau_aF)(x,t) &= F(t,t-x), \qquad \forall (x,t)\in \mathbb{R}^{2d},
\end{align*}
for $F\in M^1(\mathbb{R}^{2d})$. 
That these maps, along with the complex-conjugation map, are well-defined Banach space isomorphisms on $M^1(\mathbb{R}^{2d})$, follow from \cite[Example 5.2 (ii), (iv), and (viii)]{Ja18}. 
The \textbf{tensor product} of two functions $f,g\colon \mathbb{R}^d\to \mathbb{C}$ is denoted by $f\otimes g:\mathbb{R}^{2d}\to \mathbb{C}$ to be $(f\otimes g) (x,t) = f(x)g(t)$ for $(x,t)\in \mathbb{R}^{2d}.$
\begin{lemma}\label{tensor-conseq}
    For any $f,g\in M^1(\mathbb{R}^{d})$, we have $f\otimes g\in M^1(\mathbb{R}^{2d})$ and $\mathcal{V}_g f = (\mathcal{F}_2\circ \tau_a)(f\otimes \overline{g})\in M^1(\mathbb{R}^{2d})$.
\end{lemma}
\begin{proof}
    We have from \cite[Theorem 5.3]{Ja18} that $f\otimes g\in M^1(\mathbb{R}^{2d})$. The rest follows from an elementary calculation and from the fact that $\mathcal{F}_2$, $\tau_{a}$, and the complex-conjugation map are Banach space isomorphisms.
\end{proof}
\noindent It was shown in \cite[Theorem 7.4]{Ja18}, as a refinement of the tensor product result of Lemma \ref{tensor-conseq}, that we can represent $M^1(\mathbb{R}^{2d})$, as the \textbf{projective tensor product} (see \eg \cite[Chapter~2]{Ry02}): $M^1(\mathbb{R}^{2d})= M^1(\mathbb{R}^d)\widehat{\otimes}M^1(\mathbb{R}^d)$, i.e.
\begin{align*}
    M^1(\mathbb{R}^{d})\widehat{\otimes} M^1(\mathbb{R}^d):= \left\{F=\sum_{i\in \mathbb{N}}f_i\otimes g_i\in M^1(\mathbb{R}^{2d}): \{f_i\}_{i\in \mathbb{N}}, \{g_i\}_{i\in \mathbb{N}}\ \text{are admissible} \right\},
\end{align*}
where we say that $\{f_i\}_{i\in \mathbb{N}}, \{g_i\}_{i\in \mathbb{N}}$ are \textbf{admissible} if $\sum_{i\in \mathbb{N}}\|f_i\|_{M^1}\|g_i\|_{M^1}<\infty.$ The result says that for $F\in M^1(\mathbb{R}^{2d})$, the following is an equivalent norm on $M^1(\mathbb{R}^{2d}):$ $$\op{inf}\left\{\sum_{i\in \mathbb{N}}\|f_i\|_{M^1} \|g_i\|_{M^1}\right\}$$ with the infimum taken over all admissible $\{f_i\}_{i\in \mathbb{N}},\{g_i\}_{i\in \mathbb{N}}\subseteq M^1(\mathbb{R}^d)$ such that $F= \sum_{i\in \mathbb{N}}f_i\otimes g_i$. We finally have the following result:
\begin{corollary}\label{for-fullness}
    For any $F\in M^1(\mathbb{R}^{2d})$, there exist admissible $\{f_i\}_{i\in \mathbb{N}}$, $\{g_i\}_{i\in \mathbb{N}}\subseteq M^1(\mathbb{R}^{d})$ such that $F = \sum_{i\in \mathbb{N}}\mathcal{V}_{g_i}f_i$ converges absolutely in the $M^1(\mathbb{R}^{2d})$-norm.
\end{corollary}
\begin{proof}
    Let $F$ be $M^1(\mathbb{R}^{2d})$. Since $\mathcal{F}_2$ and $\tau_a$ are Banach space isomorphisms, there exists a $G\in M^1(\mathbb{R}^{2d})$ such that $F=(\mathcal{F}_2\circ \tau_a)(G)$, it follows from the fact that complex conjugation is a Banach space isomorphism in $M^1(\mathbb{R}^d)$ and the tensor product decomposition of $M^1(\mathbb{R}^{2d})$ that we can find admissible $\{f_i\}_{i\in \mathbb{N}},\{g_i\}_{i\in \mathbb{N}}\subseteq M^1(\mathbb{R}^d)$ such that $G = \sum_{i\in \mathbb{N}}f_i\otimes \overline{g_i}.$ It follows from Lemma \ref{tensor-conseq} that $F=\sum_{i\in \mathbb{N}}\mathcal{V}_{g_i}f_i$. Absolute convergence of the sum follows from the norm-relationship \cite[Theorem 5.3 (ii)]{Ja18} and admissibility of $\{f_i\}_{i\in \mathbb{N}},\{g_i\}_{i\in \mathbb{N}}\subseteq M^1(\mathbb{R}^d)$.
\end{proof}

We shall, for technical reasons, also consider weighted variants of Feichtinger's algebra. We denote the $\ell^1$-norm of $x=(x_1,\ldots, x_n)\in \mathbb{R}^n$ by $|x|$, and define the \textbf{weight} $\nu\colon \mathbb{R}^n\to \mathbb{R}$ by
\begin{align}\label{weight}
    \nu(x) = 1+ |x|, \qquad x\in \mathbb{R}^n.
\end{align}
In particular for $n=2d$, we now obtain the \textbf{weighted Feichtinger's algebra} $M_{\nu}^1(\mathbb{R}^d)$:
\begin{align}
    M_{\nu}^1(\mathbb{R}^d):= \left\{f\in L^2(\mathbb{R}^d): \mathcal{V}_{\phi}f \cdot \nu \in L^1(\mathbb{R}^{2d}) \right\}.
\end{align}
Some of the properties of the unweighted case are inherited by  $M^1_{\nu}(\mathbb{R}^d)$, in particular we have the following \cite{AuLu21}:
\begin{proposition}\label{prop:cont-embed-weight}
    $M^1_{\nu}(\mathbb{R}^d)$ is a Banach space with the norm $\|f\|_{M^1_{\nu}}:= \|\mathcal{V}_{\phi}f \cdot \nu\|_1$. If $\phi$ is replaced with any $g\in M^1_{\nu}(\mathbb{R}^d)$, then we obtain an equivalent norm. Furthermore, $M^1_{\nu}(\mathbb{R}^d)$ is continuously and densely embedded in $M^1(\mathbb{R}^d)$ satisfying
    \begin{align}\label{eq:cont-embed-weight}
        \|f\|_{M^1}\leq \|f\|_{M^1_{\nu}}, \qquad \text{for all}~~ f\in M^1_{\nu}(\mathbb{R}^d).
    \end{align}
\end{proposition}
\subsection{Hilbert \texorpdfstring{$C^*$}{C*}-Modules and the Heisenberg Modules}\label{subsec:hilbcmodule}
We start with a discussion of the basic theory of Hilbert $C^*$-modules, imprimitivity bimodules, and Morita equivalence. One may consult \cite{RaWi98} and \cite{La95} for reference. The latter half of the section is dedicated to a particular imprimitivity bimodule called the Heisenberg module, which we present as a dense subspace of $L^2(\mathbb{R}^d)$ following \cite{AuEn20}.
\begin{definition}\label{in-prod}
Let $A$ be a $C^*$-algebra and $Z$ a complex vector space. $Z$ is said to be an \textbf{inner product left} $\mathbf{A}$\textbf{-module} if it is a left $A$-module (with the module action denoted by juxtaposition) compatible with complex scalar multiplication in the sense that:
        \begin{align*}
            \lambda (a z) = (\lambda a) z = a (\lambda z), \qquad \forall (\lambda, a, z)\in \mathbb{C} \times A \times Z
        \end{align*}
        together with map $\lin{A}{\cdot}{\cdot}\colon  Z\times Z\to A$ (called \textbf{inner product}) such that for all $z,w,y\in Z$, $a\in A$, and scalars $\lambda,\mu \in \mathbb{C}$:
		\begin{enumerate}
			\item $\lin{A}{\lambda z+ \mu y}{w}=\lambda \lin{A}{z}{w} + \mu\lin{A}{y}{w} $ 
			\item $\lin{A}{a z}{w}= a \lin{A}{z}{w}$
			\item $\lin{A}{z}{w}^* = \lin{A}{w}{z}$
			\item $\lin{A}{z}{z}\geq 0$ 
                \item $\lin{A}{z}{z}=0 \implies z=0$.
		\end{enumerate}
\end{definition}
\noindent We see that $Z$ can be seen as a generalization of an inner-product space where we replace each instance of $\mathbb{C}$ with a general $C^*$-algebra $A$. In light of this, we can also induce a norm on $Z$. If we denote by $\|\cdot\|$ the norm of $A$, then we define, using the same notation, the induced norm on $Z$ via $\|z\|:= \|{\lin{A}{z}{z}}\|^{1/2}$. We say that $Z$ is a \textbf{Hilbert left $A$-module} or \textbf{Hilbert left $C^*$-module over $A$} if $Z$ is complete with the just defined induced norm. In this case, we also say that $Z$ is \textbf{full} if the ideal $\op{span}\{\lin{A}{z}{w}:z,w\in Z\}$ is dense in $A$. We also have the following basic estimates, showing the continuity of the operations involved.
\begin{proposition}\label{cauchy-sch}
    If $Z$ is an inner-product $A$-module, then for any $a\in A$ and $z,w\in Z:$
    \begin{enumerate}
        \item $\|{\lin{A}{z}{w}}\|\leq \|z\| \|w\|;$
        \item $\|az\|\leq \|a\| \|z\|$.
    \end{enumerate}
\end{proposition}
We note here that there are properties of Hilbert spaces that no longer hold in this generalization, for example we have the existence of non-adjointable bounded linear maps between Hilbert $C^*$-modules, and the existence of (orthogonally) non-complemented subspaces \cite{La95}. In general the existence of adjoint is supposed to be assumed for the proper maps between Hilbert $C^*$-modules. We denote the space of all \textbf{adjointable operators} for a Hilbert $A$-module $Z$ via $\mathcal{L}_A(Z):= \{T:Z\to Z \ | \ \text{$T$ has an adjoint} \}$. We can also talk about \textbf{compact operator}s on $Z$, a subalgebra of $\mathcal{L}_A(Z),$ whose definition is completely analogous to the Hilbert space case. We denote the algebra of compact operators on $Z$ via $\mathcal{K}_A(Z).$
\begin{remark}
    It is straightforward to give an analogous definition for $Z$ when it is being acted upon by a $C^*$-algebra from the \emph{right} so that we obtain a \textbf{Hilbert right $B$-module}. Everything that we have just discussed for Hilbert left modules carry over, except that the associated $B$-valued inner-product $\rin{B}{\cdot}{\cdot}\colon Z\times Z\to B$ must be linear in the second argument instead. 
\end{remark}
\begin{definition}\label{imprimitivity-bimodule}
        Let $A,B$ be $C^*$-algebras. If $Z$ is a complex vector space satisfying the following:
        \begin{enumerate}
            \item $Z$ is a full Hilbert left $A$-module and a full Hilbert right $B$-module;
            \item For all $a\in A$, $b\in B$, $z,w\in Z$:
            \begin{align*}
                \rin{B}{a z}{w} = \rin{B}{z}{a^*w},\ \text{and } \lin{A}{zb}{w} = \lin{A}{z}{wb^*};
            \end{align*}
            \item For all $z,y,w\in Z$:
            \begin{align}\label{for-janssen}
            \lin{A}{z}{y}w = z \rin{B}{y}{w},
            \end{align}
        \end{enumerate}
        then $Z$ is called an {\boldmath\textbf{$A$-$B$-imprimitivity bimodule}}. If there exists an $A$-$B$-imprimitivity bimodule between two $C^*$-algebras, then they are said to be \textbf{Morita equivalent}. 
\end{definition}
\noindent Note that, in this case, there is no ambiguity on the Banach space structure of $Z$ since $\|z\| = \|{\lin{A}{z}{z}}\|^{1/2}=\|{\rin{B}{z}{z}}\|^{1/2}$ for all $z\in Z$. That is, the norm induced by $A$ and the norm induced by $B$ on $Z$ coincide.\footnote{Actually, one can avoid Hilbert left modules altogether and define Morita equivalence via correspondences, see \cite[Section 5]{Skeide-Morita-equivalence}. While this has some advantages from a theoretical perspective, for the concrete bimodules we deal with, the Heisenberg modules, the left inner product appears quite naturally, so we prefer to work with Rieffel's original definition.} 

Introduced by Rieffel \cite{Rie74}, $A$-$B$-imprimitivity bimodules turned out to be quite an important tool for studying $C^*$-algebras, as they reveal certain equivalence of structures between $A$ and $B$, namely their ideal structures or representation theories. We shall study a particular imprimitivity bimodule called the \textbf{Heisenberg module}. 
To describe it, we introduce the \textbf{Heisenberg cocycle} $c\colon  \mathbb{R}^{2d}\times \mathbb{R}^{2d}\to \mathbb{T}$: for $z_1=(x_1,\omega_1), z_2=(x_2,\omega_2)\in\mathbb{R}^{2d}$, we define
\begin{align}\label{heis-cos}
    c(z_1,z_2) = e^{-2\pi i\, x_1\cdot \omega_2}=e^{-2\pi i\, z_1^T Kz_2},
\end{align}
with $K=\begin{pmatrix}
    0& I_d\\0 &0
\end{pmatrix}$. Clearly, the Heisenberg cocycle $c$ is a continuous map. It further satisfies the following (for all $z_1,z_2,z_3\in \mathbb{R}^{2d}$):
\begin{align}
c(z_1,z_2)c(z_1+z_2,z_3)&=c(z_1,z_2+z_3)c(z_2,z_3)\label{cocycle 1} ;\\
c(z_1,0)=c(0,0)&=c(0,z_2)=1; \label{cocycle 2}\\
\overline{c(z_1,z_2)} = c(-z_1,z_2)&=c(z_1,-z_2).\nonumber
\end{align}
Note that Equations \eqref{cocycle 1} and \eqref{cocycle 2} above are what defines a general \textbf{normalized $2$-cocycle} on any locally compact abelian group, like $\mathbb{R}^{2d}$. The \textbf{symplectic cocycle} associated with $c$ is denoted by $c_s\colon  \mathbb{R}^{2d}\times \mathbb{R}^{2d}\to \mathbb{T}$ and given by
\begin{align}\label{symp-cos}
    c_s(z_1,z_2) = c(z_1,z_2)\overline{c(z_2,z_1)}=e^{2\pi i\, (x_1\cdot \omega_2-x_2\cdot\omega_1) }=e^{-2\pi i\, z_1^T Jz_2}.
\end{align}
The relevance of the Heisenberg $2$-cocycle to Gabor analysis can be traced back to the commutation\hyp{}relations of the translation and modulation operators; indeed, \eqref{eq:TFS-commutation} shows that
\begin{align}
    \pi(z_1)\pi(z_2) &= c(z_1,z_2) \pi(z_1+z_2); \label{cov-rep} \\
    \pi(z_1)\pi(z_2)&=c_s(z_1,z_2)\pi(z_2)\pi(z_1). \label{symp-comm-rel}
\end{align}
The adjoint of the time-frequency shift is related to its reflection via:
\begin{align}
    \pi^*(z_1)=c(z_1,z_1)\pi(-z_1). \label{adj-cocycle}
\end{align}
This concludes our review of the relationship between the $2$-cocycle and the time-frequency shifts, these will be relevant in our discussion of noncommutative tori later.

For a fixed $L\in \op{GL}_{2d}(\mathbb{R}^d)$, we now introduce two $C^*$-algebras. We define  $c_L\colon \mathbb{Z}^{2d}\times \mathbb{Z}^{2d}\to \mathbb{T}$ by $c_L(k,m) = c(Lk,Lm)$ for all $(k,m)\in \mathbb{Z}^{2d}\times \mathbb{Z}^{2d}$. 
That is, $c_L$ is the restriction of the Heisenberg $2$-cocycle to the lattice $L\mathbb{Z}^{2d}$. 
To describe the relevant $C^*$-algebras in this paper, we introduce the \textbf{weighted} $\mathbf{\ell^1}$\textbf{-space}:
\begin{align*}
    \ell^1_{\nu}(\mathbb{Z}^{2d}):= \biggl\{\mathbf{a}=\{a(k)\}_{k\in \mathbb{Z}^{2d}}:\|\mathbf{a}\|_{\ell^1_{\nu}}:=\sum_{k\in \mathbb{Z}^{2d}}|a(k)|\nu(k)<\infty \biggr\}
\end{align*}
(recall that $\nu(k)=1+|k|$).
We have a result analogous to Proposition \ref{prop:cont-embed-weight} for the sequence spaces.

\begin{proposition}\label{ellnu-ell1}
    The weighted space $\ell^1_{\nu}(\mathbb{Z}^{2d})$ is a Banach space with norm $\|\cdot\|_{\ell^1_{\nu}}$, and is densely continuously embedded in $\ell^1(\mathbb{Z}^{2d})$ satisfying:
    \begin{align}\label{cont-embed-ell1}
        \|\mathbf{a}\|_{\ell^1}\leq \|\mathbf{a}\|_{\ell^1_{\nu}}, \qquad \forall \mathbf{a}\in \ell^1_{\nu}(\mathbb{Z}^{2d}).
    \end{align}
\end{proposition}

Now we equip the unweighted sequence space $\ell^1(\mathbb{Z}^{2d})$ with the following \emph{twisted} convolution and involution, respectively:
\begin{equation}\label{twisted-structure-a}
    \begin{split}
       (a_1 *_{c_L}a_2)(n) &:= \sum_{k\in \mathbb{Z}^{2d}}a_1(k)a_2(n-k)c_L(k,n-k);\\
a_1^{*_{c_L}}(n) &:= \overline{a_1(-n)}c_L(n,n), 
    \end{split}
\end{equation}
for all $\mathbf{a}_1=\{a_1(k)\}_{k\in \mathbb{Z}^{2d}}$ and $\mathbf{a}_2 = \{a_2(k)\}_{k\in \mathbb{Z}^{2d}}$. The space $\ell^1(\mathbb{Z}^{2d})$ with this structure and its original norm $\|\cdot\|_1$ gives it a Banach $*$-algebra structure, we denote the corresponding enveloping $C^*$-algebra (also known as $C^*$\emph-completion \cite[Definition 10.4]{1fell88}) by $A_L$, with the $C^*$-norm $\|\cdot\|_{A_L}$. There is yet another way to give a \emph{twisted} $*$-algebra structure to $\ell^1(\mathbb{Z}^{2d})$ which uses the matrix $L^{\circ}$ instead, given by:
\begin{equation}\label{twisted-structure-b}
    \begin{split}
        (b_1 *_{\overline{c}_{L^{\circ}}}b_2)(n) &:= \frac{1}{|\det L|}\sum_{k\in \mathbb{Z}^{2d}}b_1(k)b_2(n-k)\overline{c_{L^{\circ}}(k,n-k)};\\
b_1^{*_{\overline{c}_{L^{\circ}}}}(n) &:= \overline{b_1(-n)}\overline{c_{L^{\circ}}(n,n)}.
    \end{split}
\end{equation}
for $\mathbf{b}_1,\mathbf{b}_2\in \ell^1(\mathbb{Z}^{2d})$. This again turns $\ell^1(\mathbb{Z}^{2d})$ into a Banach $*$-algebra, and we denote the associated enveloping $C^*$-algebra by $B_L$ with corresponding $C^*$-norm $\|\cdot\|_{B_L}$. 
Both $A_L$ and $B_L$ can be identified with sequence spaces, where we have dense embeddings $\ell^1(\mathbb{Z}^{2d})\hookrightarrow C \hookrightarrow \ell^2(\mathbb{Z}^{2d})$ for $C$ either $A_L$ or $B_L$. The inner dense embedding is due to the fact that $C$ has a faithful $*$-representation (see Theorems \ref{thm: faithful-a} and \ref{thm: faithful-b}) which makes $\ell^1(\mathbb{Z}^{2d})$ with the twisted structure of either \eqref{twisted-structure-a} or \eqref{twisted-structure-b} a \textbf{reduced} $*$-algebra \cite[Definition 10.2]{1fell88}. On the other hand, the outer dense embedding is due to \cite[Proposition 3.3]{AuEn20}. 

We can actually see the $C^*$-algebras $A_L$ and $B_L$ as \textbf{twisted crossed-products} in the sense of Packer and Raeburn \cite{PaRa89}. For example, since the automorphism group (with respect to $C^*$-algebras) of $\mathbb{C}$, denoted by $\op{Aut}(\mathbb{C})$ is trivial, we have that any continuous homomorphism $\op{id}\colon \mathbb{Z}^{2d}\to \op{Aut}(\mathbb{C})$ must be $\op{id}(k,m)(\lambda)=\lambda$ for all $(k,m)\in \mathbb{Z}^{2d}$ and $\lambda\in \mathbb{C}$. Following the construction of Packer and Raeburn, we find that $A_L$ is the twisted crossed-product generated by the twisted $C^*$-dynamical system $(\mathbb{C}, \mathbb{Z}^{2d},\op{id}, c_L)$ \cite[Definition 2.1]{PaRa89}, where we interpret $\mathbb{Z}^{2d}$ as a locally compact group prescribed with a Haar measure coinciding with the usual counting measure. More importantly, formula \eqref{cov-rep} implies that the 
 time-frequency shifts induce a \textbf{covariant representation} \cite[Definition 2.3]{PaRa89} for the twisted $C^*$-dynamical system $(\mathbb{C},\mathbb{Z}^{2d},\op{id},c_L)$ via $\mathbb{Z}^{2d}\ni k\mapsto \pi(Lk)\in \mathcal{L}(L^2(\mathbb{R}^d))$. For each covariant representation on a $C^*$-dynamical system, there is a corresponding \textbf{integrated form} \cite[Definition 2.4 (b)]{PaRa89} which is an actual representation for the associated twisted crossed-product. We denote the integrated form of $k\mapsto \pi(Lk)$ by $\overline{\pi}_L\colon  A_L\to \mathcal{L}(L^2(\mathbb{R}^d))$, which is an $\ell^1$-norm decreasing representation, densely defined by
 \begin{align}\label{a-integrated-rep}
 \overline{\pi}_L(\mathbf{a})= \sum_{k\in \mathbb{Z}^{2d}}a(k)\pi(Lk), \qquad \forall \mathbf{a}\in \ell^1(\mathbb{Z}^{2d}).
 \end{align}
 An important property of this particular representation, shown by Rieffel \cite{Ri88}, is that it is faithful:
 \begin{theorem}\label{thm: faithful-a}
     The integrated form $\overline{\pi}_L\colon  A_L\to \mathcal{L}(L^2(\mathbb{R}^d))$ is a faithful $*$\hyp{}representation. In particular, it is an isometry: $\|\mathbf{a}\|_{A_L}=\|\overline{\pi}_L(\mathbf{a})\|$ for all $\mathbf{a}\in A_L$.
 \end{theorem}
 As usual, there are corresponding analogues for $B_L$. It can also be obtained as the twisted crossed-product of a twisted $C^*$-dynamical system $(\mathbb{C},\frac{1}{|\det L|}\mathbb{Z}^{2d},\op{id},\overline{c_{L^{\circ}}})$ where $\frac{1}{|\det L|}\mathbb{Z}^{2d}$ is the locally compact group $\mathbb{Z}^{2d}$ prescribed with the Haar measure coinciding with the counting measure scaled by $1/{|\det L|}$. The unitary map of interest for us is the one given by the adjoint of the time-frequency shifts $\mathbb{Z}^{2d}\ni k\mapsto \pi^*(L^{\circ}k)\in \mathcal{L}(\mathbb{R}^d)$. However, unlike the situation in $A_L$, Equations \eqref{cov-rep} and \eqref{adj-cocycle} show that $\pi^*$ is not\footnote{since $\pi^*(L^{\circ}k_1)\pi^*(L^{\circ}k_2)\neq \overline{c}(L^{\circ}k_1,L^{\circ}k_2)\pi^*(L^{\circ}k_1+L^{\circ}k_2)$} exactly a covariant representation for the twisted $C^*$-dynamical system. The integrated form of $\pi^*$, given by $\overline{\pi}^{*}_{L^{\circ}}$ densely defined via:
 \begin{align*}
     \overline{\pi}^*_{L^{\circ}}(\mathbf{b}):= \frac{1}{|\det L|} \sum_{k\in \mathbb{Z}^{2d}}b(k)\pi^*(L^{\circ}k), \qquad \forall \mathbf{b}\in \ell^1(\mathbb{Z}^{2d})
 \end{align*}
 is not a $*$-representation for $B_{L}$ but is instead a $\mathbf{*}$\textbf{-antirepresentation}. This means that it reverses the multiplicative structure on $B_{L}$, while still preserving the involutive structure. It is possible to choose a different $2$-cocycle instead of $\overline{c}$ such that $\overline{\pi}^*_{L^{\circ}}$ becomes a $*$-representation in its associated twisted group crossed-product. In our case however, the choice of cocycle $\overline{c}$ for $B_{L}$ is deliberate. Though not exactly explained in terms of $*$-antirepresentations, this was discussed briefly in \cite{Ri88,JaLu21}, the reason for this choice is roughly because $B_{L}$ acts to the right of the Heisenberg module (see Theorem \ref{heisenberg-module}), and we would then want a proper `representation' for this right module action. In any case, there is still an analogue of Theorem \ref{thm: faithful-a} for $B_L$.
 \begin{theorem}\label{thm: faithful-b}
     The integrated form $\overline{\pi}^*_{L^{\circ}}\colon  B_L\to \mathcal{L}(L^2(\mathbb{R}^d))$ is a faithful $*$\hyp{}antirepresentation. In particular, it is an isometry: $\|\mathbf{b}\|_{B_L}=\|\overline{\pi}^*_{L^{\circ}}(\mathbf{b})\|$ for all $\mathbf{b}\in B_L$.
 \end{theorem}
 
 \begin{remark}
     We note here that if we equip $\mathbb{Z}^{2d}$ with the counting measure scaled by $1/|\det L|$, then $\ell^1(\mathbb{Z}^{2d})$ has the following $\ell^1$-norm: $\|\mathbf{b}\|_1=\frac{1}{|\det L|} \sum_{k\in \mathbb{Z}^{2d}}|b(k)|$. So $\overline{\pi}^*_{L^{\circ}}$ being $\ell^1$-norm decreasing explicitly amounts to $\|\overline{\pi}^*_{L^{\circ}}(\mathbf{b})\|\leq \frac{1}{|\det L|}\sum_{k\in \mathbb{Z}^{2d}}|b(k)|$ for each $\mathbf{b}\in \ell^1(\mathbb{Z}^{2d})$.
 \end{remark}

The construction above descends to the weighted case, following Equation \eqref{cont-embed-ell1} and the fact that $\overline{\pi}_L$ and $\overline{\pi}_{L^{\circ}}^*$ are both $\ell^1$-norm decreasing representations. 
The following result can be found in \cite{AuLu21}.

\begin{proposition}
     For a fixed $L\in \op{GL}_{2d}(\mathbb{R})$, $\ell^1_{\nu}(\mathbb{Z}^{2d})$ is a Banach $*$-algebra when equipped with twisted convolution and involution as in \eqref{twisted-structure-a} (resp. \eqref{twisted-structure-b}). The $C^*$-enveloping algebra of $\ell^1_{\nu}(\mathbb{Z}^{2d})$ with such a structure coincides with $A_L$ (resp. $B_L$). 
\end{proposition}
 
We are now ready to describe the Heisenberg module associated with $L\in \op{GL}_{2d}(\mathbb{R}^d)$. We combine the construction in \cite{AuLu21} and the result of \cite{AuEn20}, obtained by \emph{localization}, to give a concrete description of the Heisenberg module as a function space that takes into account the weight $\nu$.

\begin{theorem}\label{heisenberg-module}
    Fix $L\in \op{GL}_{2d}(\mathbb{R})$, we denote the associated Heisenberg module by $\mathcal{E}_L(\mathbb{R}^d)$. It can be obtained by completing the weighted Feichtinger's algebra $M^1_{\nu}(\mathbb{R}^d)$ on $L^2(\mathbb{R}^d)$ with respect to the norm
    \begin{align*}
        \|f\|_{\mathcal{E}_L(\mathbb{R}^d)}:= \inf\left\{C^{1/2}:C\text{ is a Bessel bound for $\mathcal{G}(f,L)$} \right\}.
    \end{align*}
    We also have the following sequence of continuous dense embeddings: 
    \[M^1_{\nu}(\mathbb{R}^d)\hookrightarrow M^1(\mathbb{R}^d)\hookrightarrow \mathcal{E}_L(\mathbb{R}^d)\hookrightarrow L^2(\mathbb{R}^d).\]
    Furthermore, $\mathcal{E}_L(\mathbb{R}^d)$ is an $A_L$-$B_L$-imprimitivity bimodule with the following Hilbert left $A_L$-module structure: for $k\in \mathbb{Z}^{2d}$,  $\mathbf{a}\in A_L$ and $f,g\in \mathcal{E}_L(\mathbb{R}^d)$,
    \begin{multicols}{2}
        \begin{enumerate}
            \item $\lin{L}{f}{g}(k):=\rin{2}{f}{\pi(Lk)g}$;
            \item $\mathbf{a} f := \overline{\pi}_L(\mathbf{a})f$.
        \end{enumerate}
    \end{multicols}
    \noindent On the other hand, we find that $\mathcal{E}_L(\mathbb{R}^d)$ has the following Hilbert right $B_L$-module structure: for $\mathbf{b}\in B_L$,
    \begin{multicols}{2}
        \begin{enumerate}
            \item $\rin{L^{\circ}}{f}{g}(k):= \rin{2}{g}{\pi^*(L^{\circ}k)f}$;
            \item $f \mathbf{b} := \overline{\pi}^*_{L^{\circ}}(\mathbf{b})f$.
        \end{enumerate}
    \end{multicols}
\end{theorem}

\begin{remark}\label{resps-embed}
    The module structure of $\mathcal{E}_L(\mathbb{R}^d)$ respects the continuous dense embeddings 
    \[M^1_{\nu}(\mathbb{R}^d)\hookrightarrow M^1(\mathbb{R}^d)\hookrightarrow \mathcal{E}_L(\mathbb{R}^d).\] 
    For example, if $f,g\in M^1_{\nu}(\mathbb{R}^d)$ and $\mathbf{a}\in \ell^1_{\nu}(\mathbb{Z}^{2d})$, then $\lin{L}{f}{g}\in \ell^1_{\nu}(\mathbb{Z}^{2d})$ and $\mathbf{a}f\in M^1_{\nu}(\mathbb{R}^d)$ \cite[Proposition 3.14]{AuLu21}.
\end{remark} 

That $\mathcal{E}_L(\mathbb{R}^d)$ is a viable function space for time-frequency analysis comes from the fact that it also inherits the similar property of being Bessel bounded just like $M^1(\mathbb{R}^d)$, albeit it only holds for the particular lattice $L\mathbb{Z}^{2d}$  \cite[Theorem 3.10 and Theorem 3.15]{AuEn20}.

\begin{corollary}
    For each $g\in \mathcal{E}_L(\mathbb{R}^d)$, the family $\mathcal{G}(g,L)$ forms a Bessel family. 
    Consequently, the Gabor analysis, Gabor synthesis, and Gabor mixed-frame operators are all continuous when the atoms are taken from $\mathcal{E}_L(\mathbb{R}^d)$. 
    Furthermore, if $f,g\in \mathcal{E}_L(\mathbb{R}^d)$ then the Gabor mixed-frame operator satisfies the following restriction property:
    \begin{align}\label{mixed-frame-restrict}
        (S_{f,g,L})_{|\mathcal{E}_L(\mathbb{R}^d)}\colon  \mathcal{E}_L(\mathbb{R}^d)\to \mathcal{E}_L(\mathbb{R}^d)
    \end{align}
    and it is continuous with respect to the $\mathcal{E}_L(\mathbb{R}^d)$-norm.
\end{corollary}
Our previous discussions about imprimitivity bimodules also hold. For example, it is also true that for any $f\in \mathcal{E}_L(\mathbb{R}^d)$:
\begin{equation}\label{equiv-norm-for-heis}
    \begin{split}
        \|f\|_{\mathcal{E}_L(\mathbb{R}^d)} 
        &= \|\lin{L}{f}{f}\|_{A_L}^{1/2} =\| \overline{\pi}_{L}\left(\lin{L}{f}{f}\right)\|^{1/2} = \left\|\sum_{k\in \mathbb{Z}^{2d}}\rin{2}{f}{\pi(Lk)f}\pi(Lk) \right\|^{1/2} \\
        &=\|\rin{L^{\circ}}{f}{f}\|_{B_L}^{1/2} =\| \overline{\pi}^*_{L^{\circ}}\left(\rin{L^{\circ}}{f}{f}\right)\|^{1/2} = \left\|\frac{1}{|\det L|}\sum_{k\in \mathbb{Z}^{2d}}\rin{2}{f}{\pi^*(L^{\circ}k)f}\pi^*(L^{\circ}k) \right\|^{1/2}.
    \end{split}    
\end{equation}
More importantly, the associativity condition Equation \eqref{for-janssen} amounts to an extension of the so-called \emph{Fundamental Identity of Gabor Analysis} (FIGA) \cite{Ja94,FeLu06}. For any $f,g,h\in \mathcal{E}_L(\mathbb{R}^d)$, ${\lin{L}{f}{g}}h = f{\rin{L^{\circ}}{g}{h}}$ explicitly reads as:
\begin{align*}
    \sum_{k\in \mathbb{Z}^{2d}}\rin{2}{f}{\pi(Lk)g}\pi(Lk)h = \frac{1}{|\det L|}\sum_{k\in \mathbb{Z}^{2d}}\rin{2}{h}{\pi^*(L^{\circ}k)g}\pi^*(L^{\circ}k)f.
\end{align*}
Equivalently, the density of $\mathcal{E}_L(\mathbb{R}^d)$ in $L^2(\mathbb{R}^d)$ and Equation \eqref{adj-cocycle} implies that we have an extension of Janssen's representation \eqref{jan-rep} for the Heisenberg modules. Moreover, the function space characterization of $\mathcal{E}_L(\mathbb{R}^d)$ also shows that the Heisenberg modules obey the Wexler-Raz relations of Theorem \ref{wexler-raz}. We write these crucial observations below for later reference.

\begin{corollary}\label{gen-for-heis}
    For $f,g,h\in \mathcal{E}_L(\mathbb{R}^d)$, the following holds:
    \begin{enumerate}
        \item \emph{\textbf{(Janssen's Representation).}} The mixed Gabor frame operator $S_{g,h,L}$ is given by 
        \begin{align*}
        S_{g,h,L} = \overline{\pi}^*_{L^{\circ}}(\rin{L^{\circ}}{g}{h})&=\frac{1}{|\det L|}\sum_{k\in \mathbb{Z}^{2d}}\rin{2}{h}{\pi^*(L^{\circ}k)g}\pi^*(L^{\circ}k)\\
        &=\frac{1}{|\det L|}\sum_{k\in \mathbb{Z}^{2d}}\rin{2}{h}{\pi(L^{\circ}k)g}\pi(L^{\circ}k).
        \end{align*}
        \item \emph{\textbf{(Wexler-Raz Relation).}} The Gabor systems $\mathcal{G}(g,L)$ and $\mathcal{G}(h,L)$ are dual frames if and only if
        \begin{align*}
            \rin{L^{\circ}}{g}{h}=\rin{L^{\circ}}{h}{g}=|\det L| \cdot \delta_{0}.
        \end{align*}
    \end{enumerate}
\end{corollary}

\subsection{Banach Bundles}\label{subsec: bbunldles}
Bundles are ubiquitous in topology and geometry, and in this paper, we shall study bundles whose fibers are in general infinite-dimensional Banach spaces, with some emphasis on the case when they are $C^*$-algebras. The basics of Banach bundles can be found in the first volume of Fell's monograph \cite{1fell88}, and the classic book by Dixmier \cite{dix82} where they are called ``Continuous Fields of Banach Spaces''. For a more abstract approach on the case of $C^*$-bundles, one may want to consult Appendix C of Williams' book \cite{Will07}. 

Here we define some basic notions. A \textbf{bundle} is a triple $\mathscr{B}=(\mathcal{B},\rho,X)$ where $\rho\colon  \mathcal{B}\to X$ is an open continuous surjection from a Hausdorff space $\mathcal{B}$ to a Hausdorff space $X$. We call $X$ the \textbf{base space}, while $\mathcal{B}$ the \textbf{bundle space}, or \textbf{total space}, and the map $\rho$, the \textbf{bundle projection map}. For each $x\in X$, we call $\rho^{-1}(x):=\mathcal{B}_x$ the \textbf{fiber} over $x$. A \textbf{section} of the bundle $\mathscr{B}=(\mathcal{B},\rho,X)$ is any function $\gamma\colon X\to \mathcal{B}$ such that $(\rho\circ \gamma)(x)=x$ for all $x\in X$ (which implies that  $\gamma(x)\in \rho^{-1}(x) = \mathcal{B}_x$). 
We say that $\gamma$ \textbf{passes through} $s\in \mathcal{B}$ if $s\in \operatorname{Im}(\gamma)$. 
We say that the bundle $\mathscr{B}=(\mathcal{B},\rho,X)$ has \textbf{enough continuous sections} if for each $s\in \mathcal{B}$ there is a continuous section $\gamma\colon X\to \mathcal{B}$ that passes through $s$. The map $\rho\colon \mathcal{B}\to X$ is enough to determine a bundle, and we will usually refer to a bundle by the projection map itself.
\begin{definition}\label{bundledef}
	A \textbf{Banach bundle} $\mathscr{B}$ over $X$ is a bundle $\mathscr{B}=(\mathcal{B},\rho,X)$ over $X$, together with operations and norms making each fiber $\mathcal{B}_x$ $(x\in X)$ into a Banach space, and satisfying the following conditions:
	\begin{enumerate}
		\item The map $s\mapsto \|s\|_{\mathcal{B}_x}$ is continuous from $\mathcal{B}$ to $\mathbb{R}$;
		\item The operation $+$ is continuous as a function on $\{(s,t)\in \mathcal{B}\times \mathcal{B}: \rho(s)=\rho(t) \}$ to $\mathcal{B}$ for each $s,t\in X;$
		\item For each complex $\lambda$, the map $s\mapsto \lambda\cdot s$ is continuous from $\mathcal{B}$ to $\mathcal{B}$;
		\item If $x\in X$ and $\{s_i\}_{i}$ is any net of elements in $B$ such that $\|s_i\|_{\mathcal{B}_{x_i}}\to 0$ and $\rho(s_i)\to x$, then $s_i\to 0_x$, where $0_x$ is the additive identity on $\mathcal{B}_x$ and $s_i \in \mathcal{B}_{x_i}$ for each $i$.
	\end{enumerate}
	Furthermore, if each fiber $\mathcal{B}_x$ is a $C^*$-algebra, then with the added hypothesis that 
	\begin{enumerate}[resume]
        \item The multiplication $\cdot$ is continuous as a function on $\{(s,t)\in \mathcal{B}\times \mathcal{B}: \rho(s)=\rho(t)\}$ to $\mathcal{B};$
	    \item The involution $s\mapsto s^*$ is continuous from $\mathcal{B}$ to $\mathcal{B}$,
	\end{enumerate}
	then $\mathscr{B}$ is called a $C^*$\textbf{-bundle}. We shall drop the subscripts $\mathcal{B}_x$ when the context is clear that we are taking norms with respect to fibers.
\end{definition}
 The Banach space structure on the fibers gives our projection map some basic mapping properties, if $s\in \mathcal{B}$ such that $\rho(s)=x$, then it is in the fiber $s\in \mathcal{B}_x$, but in particular $\mathcal{B}_x$ is a linear space, and so $\lambda s+b \in \mathcal{B}_x$ for any $\lambda\in \mathbb{C}$ and $b\in \mathcal{B}_x$ (this is implicitly used to make sense of axiom $3$ above). This linear structure yields 
 \begin{align*}
     \rho(\lambda s+b)=x = \rho(s)=\rho(b).
 \end{align*} 
For the rest of the document, given a bundle $\rho\colon \mathcal{B} \to X$, we denote the complex linear space of continuous sections (where the operations are given pointwise) by $\Gamma(\rho)$. Furthermore, we denote by $\Gamma_0(\rho)$ the subspace of $\Gamma(\rho)$ consisting of \textbf{sections that vanish at infinity} (\ie $\gamma\in \Gamma_0(\rho)$ if and only if $x\mapsto \|\gamma(x)\|$ vanishes at infinity). Note that $\Gamma_0(\rho)$ is actually a Banach space with the supremum norm, since all such sections will be bounded. It is also a basic result that when $\rho$ is a $C^*$-bundle, then the associated section space $\Gamma_0(\rho)$ is a $C^*$-algebra with the necessary operations defined pointwise. 
To reiterate, continuous sections of $\rho$ are continuous maps $\gamma\colon X\to \mathcal{B}$ with the property that $\gamma(x)\in \mathcal{B}_x$.

We give a description of how we shall encounter bundles in time-frequency analysis. Suppose we have a Hausdorff space $X$ and we have a family of Banach spaces $\{\mathcal{B}_x\}_{x\in X}$, and we fix $\mathcal{B} := \bigsqcup_{x\in X}\mathcal{B}_x:=\bigcup_{x\in X}\{(s,x):s\in B_x\}$. 
Now we simply define $\rho\colon \mathcal{B}\to X$ to be the usual projection, \ie $\rho(s,x)=x$. 
In practice, we shall omit the the index $x$ in $(s,x)\in \mathcal{B}$ and work with the implicit understanding that each $s$ is an element in some unique $\mathcal{B}_x$. So, we shall write $\rho(s)=x \iff s\in \mathcal{B}_x$. This definition will give us fibers that we want, that is: $\rho^{-1}(x):=\mathcal{B}_x$. A section $\gamma\colon X\to \mathcal{B}$ with respect to $\rho$ satisfies $\gamma(x)\in \mathcal{B}_x$ for each $x\in X$, and can be succinctly denoted by the product notation: $\gamma \in \prod_{x\in X}\mathcal{B}_x$.  Our bundle construction problem now becomes: Finding a proper topology for $\mathcal{B}$ such that $\rho$ is a Banach bundle. There is a canonical answer to this problem if one also starts with initial sections possessing nice properties by Fell \cite[Theorem 13.18]{1fell88}.
\begin{theorem}\label{bbundleconstruction}
If $\rho\colon \mathcal{B}\to X$ is a surjective map such that $\rho^{-1}(x):=\mathcal{B}_x$ is a complex Banach space for each $x\in X$, and there exists a complex linear space of sections $\Gamma$ of $\rho$, with the added properties that
\begin{enumerate}
    \item For all $\gamma\in \Gamma$, the function $x\mapsto \|\gamma(x)\|$ is a continuous function from $X$ to $\mathbb{R}$;
	\item For each $x\in X$, $\Gamma^x := \{\gamma(x)\mid \gamma\in \Gamma \}$ is dense in $\mathcal{B}_x$.
\end{enumerate}
Then there exists a unique topology on $\mathcal{B}$ that makes $\rho\colon \mathcal{B}\to X$ a Banach bundle over $X$ such that the elements of $\Gamma$ are continuous sections of $\rho$, \ie $\Gamma\subseteq \Gamma(\rho)$. Explicitly, a basis for the topology on $\mathcal{B}$ is given by
	\begin{align}\label{basic-opens}
	    W(f,U,\varepsilon)=\{s\in \mathcal{B}: \rho(s)\in U, \|s-f(\rho(s))\|<\varepsilon \}
	\end{align}
	for $f\in \Gamma$, $U$ an open set in $X$ and $\varepsilon>0$.
\end{theorem}
\noindent Theorem \ref{bbundleconstruction}  above will be the most relevant result for the subsequent section, as all bundles that we shall construct will be obtained this way.
\begin{definition}\label{local-approx}
    Let $\rho\colon \mathcal{B}\to X$ be a Banach bundle constructed through Theorem \ref{bbundleconstruction} with initial section space $\Gamma\in \Gamma(\rho)$. We call $\Gamma$ the \textbf{local approximating sections} for $\rho$.
\end{definition}
The local approximating sections given by $\Gamma$ above play an important role, as they are used to verify that a section is indeed continuous if and only if they can be locally approximated by sections that belong in $\Gamma$. We shall show this in the following proposition, for the sake of completeness.
\begin{proposition}\label{approxbycontinuity}
    Let $\rho\colon X\to \mathcal{B}$ be a Banach bundle via Theorem \ref{bbundleconstruction} with local approximating sections $\Gamma$. 
    Then for any section $\gamma\colon X\to \mathcal{B}$, $\gamma\in \Gamma(\rho)$ if and only if for each $x_0\in X$ and $\varepsilon>0$, there exists a $\gamma'\in \Gamma$ such that for some open neighborhood $U_0\ni x_0$
    \begin{align*}
        \|\gamma(x) - \gamma'(x)\|<\varepsilon, \qquad \forall x\in U_0.
    \end{align*}
\end{proposition}
\begin{proof}
    To start with, we can apply Theorem \ref{bbundleconstruction} to the triple $\mathscr{B} = (\mathcal{B},\rho,X)$ with $\Gamma$. Because our  $\Gamma$ satisfies the hypothesis of Theorem \ref{bbundleconstruction}, we have an explicit basis for the topology on $\mathcal{B}$:
    \begin{align*}
        W(\gamma'',U,\varepsilon):= \{s\in \mathscr{B}: \rho(s)\in U, \|s-\gamma''(\rho(s))\|<\varepsilon  \}
    \end{align*}
    for $\gamma''\in \Gamma$, $U$ an open set in $X$ and $\varepsilon>0$. 
    
    If $\gamma \in \Gamma(\rho)$, \ie if $\gamma$ is continuous, then the conclusion is trivial. 
    Suppose now that the converse hypothesis holds for the section $\gamma\colon X\to  \mathcal{B}$. 
    We want to show that $\gamma$ is continuous at any point, we take $x_0\in X$ and take a basic open set $W(\gamma'',U_0,\frac{\varepsilon}{2})\ni \gamma(x_0)$. 
    By hypothesis, there exists $\gamma'\in \Gamma$ and an open set $V_0'\ni x_0$ such that we have for all $x\in V_0'$:
    \begin{align*}
        \|\gamma(x)-\gamma'(x)\|<\frac{\varepsilon}{2}.
    \end{align*}
    Now, we know that $\gamma'\in \Gamma\subseteq \Gamma(\mathscr{B})$, therefore it is immediate that there exists an open set $V_0''\ni x_0$ such that $\gamma'(V_0'')\subseteq W(\gamma'',U_0,\frac{\varepsilon}{2})$. Take the open set $x_0\in V_0 := V_0' \cap V_0''$, then we have for all $x\in V_0:$
    \begin{align*}
        \|\gamma(x) - \gamma''(x)\|\leq \|\gamma(x)-\gamma'(x)\|+\|\gamma'(x)-\gamma''(x)\| <\frac{\varepsilon}{2}+\frac{\varepsilon}{2}=\varepsilon.
    \end{align*}
    Therefore $\gamma(V_0)\subseteq W(\gamma'',U_0,\varepsilon)$, hence $\gamma$ is continuous. 
\end{proof}
\begin{remark}\label{approx-section-refinement}
    By the triangle inequality, we can refine Proposition \ref{approxbycontinuity}. We have that a section $\gamma$ is continuous if and only if it can it can be locally approximated by sections in $\Gamma(\rho)$ (instead of just $\Gamma$).
\end{remark}
The results given above show that Banach bundles constructed via Theorem \eqref{bbundleconstruction} are \emph{equivalent} to the continuous fields of Banach spaces defined by Dixmier in \cite[Definition 10.1.2]{dix82}. It then follows that we have the crucial result \cite[Proposition 10.1.10]{dix82}.
\begin{corollary}\label{enough-c-sections}
    Every Banach bundle constructed by Theorem \ref{bbundleconstruction} will always have enough continuous sections.
\end{corollary}
\noindent A more general result exists for base spaces that are locally compact. It was communicated by Fell in \cite[Appendix C]{1fell88}, that by Douady and Soglio-H\'erault \cite{SoDo}: every Banach bundle whose base space is locally compact Hausdorff automatically has enough continuous sections. In the sequel we shall construct our Banach bundles with base space $\op{GL}_{2d}(\mathbb{R})$, which is locally compact, so we are fully vindicated in invoking the existence of a continuous section that passes through any point in $\op{GL}_{2d}(\mathbb{R})$.

The following result from \cite[Proposition C.24]{Will07} or \cite[Corollary 14.7]{1fell88} regarding global approximations of sections in $\Gamma_0(\rho)$ holds in general even if the Banach bundle is not constructed from Theorem \ref{bbundleconstruction}.
\begin{proposition}\label{globalapprox}
    Let $\rho\colon  \mathcal{B}\to X$ be a Banach bundle such that there exists a linear subspace $\Gamma_0$ of $\Gamma_0(\rho)$ that satisfies:
    \begin{enumerate}
        \item For all $\varphi \in C_0(X)$, and $\gamma\in \Gamma_0$, we have $\varphi \cdot \gamma \in \Gamma_0$.
        \item For each $x\in X$, $\{\gamma(x):\gamma\in \Gamma_0 \}$ is dense in $\mathcal{B}_x$.
    \end{enumerate}
    Then $\Gamma_0$ is dense in $\Gamma_0(\rho)$.
\end{proposition}
\begin{remark}
    All of the preceding results hold in particular for the $C^*$-bundle case.
\end{remark}

The following results hold just for $C^*$-bundles. The next proposition can be deduced from \cite[Proposition 10.3.3]{dix82}.

\begin{proposition}\label{positive-sections}
    Let $\rho\colon \mathcal{A} \to X$ be a $C^*$-bundle. If $\gamma \in \Gamma_0(\rho)$ fulfills $\gamma(x)\geq 0$ for all $x\in X$, then $\gamma$ is a positive element in the $C^*$-algebra $\Gamma_0(\rho)$.
\end{proposition}

\begin{lemma}[{\cite[Lemma 10.4.2]{dix82}}]\label{ideal-section}
    Let $\rho\colon  \mathcal{A}\to  X$ be a $C^*$-bundle and $I$ a closed two-sided ideal of $\Gamma_0(\rho)$. Put $I_x:=\{\gamma(x)\in \mathcal{A}_x: \gamma\in I\}$. Then $I = \{\gamma\in \Gamma_0(\rho): \gamma(x)\in I_x \text{ for all 
 } x\in X\}$.
\end{lemma}

\begin{corollary}\label{cor-ideal-section}
    Let $\rho\colon  \mathcal{A} \to X$ be a $C^*$-bundle. If $I$ is a closed two-sided ideal of $\Gamma_0(\rho)$ such that $I_x=\mathcal A_x$ for each $x\in X$, then $I= \Gamma_0(\rho)$.
\end{corollary}
\begin{proof}
We know from Lemma \ref{ideal-section} that $I = \{\gamma\in \Gamma_0(\rho): \gamma(x)\in I_x\}$, but $I_x=\mathcal{A}_x$, so $I= \{\gamma\in \Gamma_0(\rho): \gamma(x)\in \mathcal{A}_x\}$ which is exactly $\Gamma_0(\rho)$.
\end{proof}

\subsection{Noncommutative Tori}\label{subsec:nctori}

We denote by $\op{Skew}_n(\mathbb{R})$ the set of skew-symmetric matrices in $M_n(\mathbb{R})$.

\begin{definition}
    Let $\Theta\in \op{Skew}_n(\mathbb{R})$, then the \textbf{noncommutative $n$-torus determined by} $\mathbf{\Theta}$ is the universal unital $C^*$-algebra generated by $n$-generators $U_1(\Theta),\ldots,U_{n}(\Theta)$ satisfying the relations:
    \begin{enumerate}
        \item $U_i^*(\Theta)=U_{i}^{-1}(\Theta)$ for all $i=1,\ldots,n;$
        \item $U_i(\Theta)U_j(\Theta) = e^{2\pi i\, \Theta_{ij}}U_j(\Theta)U_i(\Theta)$ for all $i,j=1,\ldots,n$.
    \end{enumerate}
    We denote by $\mathcal{A}_{\Theta}$ the noncommutative $n$-torus determined by $\Theta$. 
\end{definition}
\begin{remark}\label{skew}
   There is an equivalent way of determining a noncommutative $n$-torus by just looking at elements of $\mathbb{T}^{n(n-1)/2}\ni \lambda= (\lambda_{i,j})_{1\leq i<j\leq n}$, and from here we can define the relations on the generators $U_1(\lambda),\ldots,U_n(\lambda)$ of noncommutative $n$-torus determined by $\lambda\in \mathbb{T}^{n(n-1)/2}$ to be
    \begin{enumerate}
        \item $U_i(\lambda)^*=U_i^{-1}(\lambda)$ for all $i=1,\ldots,n$;
        \item $U_i(\lambda)U_j(\lambda) = \lambda_{i,j}U_j(\lambda)U_i(\lambda)$ for all $1\leq i<j\leq n$.
    \end{enumerate}
    Similar to the original definition, we denote by $\mathcal{A}_{\lambda}$ the noncommutative $n$-torus determined by $\lambda\in \mathbb{T}^{n(n-1)/2}$. 
    Note that one can go from $\Theta\in \op{Skew}_n(\mathbb{R})$ to $\lambda =\mathbb{T}^{n(n-1)/2}$ simply by the map $\Theta \mapsto \lambda = (e^{2\pi i\Theta_{1,2}}, e^{2\pi i\Theta_{1,3}},\ldots,e^{2\pi i\, \Theta_{n-1,1}})$, so that $\mathcal{A}_{\Theta}\cong \mathcal{A}_{\lambda}$. 
    Note however that while $\Theta\mapsto \lambda$ is surjective, we will never find a global inverse for this map since the complex exponential function itself is not globally invertible. 
\end{remark}

There is yet another way to see these higher dimensional noncommutative tori. For this, we introduce the so-called \textbf{twisted group} $\mathbf{C^*}$\textbf{-algebras} on $\mathbb{Z}^n$. Fix the following: Let $\zeta\colon \mathbb{Z}^{n}\times \mathbb{Z}^{n}\to  \mathbb{T}$ be a normalized $2$-cocycle on $\mathbb{Z}^n$ (\ie it satisfies formulae \eqref{cocycle 1} and \eqref{cocycle 2} on $\mathbb{Z}^{n}$), equip $\mathbb{Z}^{n}$ with the counting measure\footnote{The construction can be done for any Haar measure on $\mathbb{Z}^n$ of course.} and construct $\ell^1(\mathbb{Z}^{n})$. We equip $\ell^1(\mathbb{Z}^n)$ with a $\zeta$-twisted convolution and involution given exactly by replacing $c_L$ with $\zeta$ in Equations \eqref{twisted-structure-a}, that is:
\begin{align*}
    (a_1 *_{\zeta} a_2)(m) &:= \sum_{k\in \mathbb{Z}^n}a_1(k)a_2(m-k)\zeta(k,m-k),\\
    a_1^{*_{\zeta}}(m) &:= \overline{a_1(-m)} \overline{\zeta(m,-m)}
\end{align*}
for all $k,m\in \mathbb{Z}^n$, and $\mathbf{a}_1,\mathbf{a}_2\in \ell^1(\mathbb{Z}^n).$ The corresponding enveloping $C^*$-algebra is denoted $C^*(\mathbb{Z}^n,\zeta)$. We see that $A_L = C^*(\mathbb{Z}^{2d},c_L)$, and it can also be shown that $B_L = C^*\left(\frac{1}{|\det L|}\mathbb{Z}^{2d},\overline{c_{L^{\circ}}}\right)$ if we again denote by $\frac{1}{|\det L|}\mathbb{Z}^{2d}$ the space $\mathbb{Z}^{2d}$ equipped with the counting measure scaled by $\frac{1}{|\det L|}$.
\begin{remark}
    Note that the $C^*$-algebra $C^*(\mathbb{Z}^n, \zeta)$ is exactly the twisted crossed-product generated by the twisted $C^*$-dynamical system $(\mathbb{C},\mathbb{Z}^{n},\op{id},\zeta)$ of Packer and Raeburn \cite{PaRa89} (see the paragraph succeeding Proposition \ref{cont-embed-ell1}). It is useful to realize that their definition of twisted crossed-products are quite general that they include the twisted group $C^*$-algebras.
\end{remark}
\begin{remark}
    The $\zeta$-twisted $\ell^1(\mathbb{Z}^n)$ Banach $*$-algebra that we use to obtain $C^*(\mathbb{Z}^n,\zeta)$ is also a reduced $*$-algebra since its regular representation $\lambda: \ell^1(\mathbb{Z}^n)\to \mathcal{L}(\ell^2(\mathbb{Z}^n))$ $$\lambda(\mathbf{a}_1)\mathbf{a} = \mathbf{a_1}*_{\zeta}\mathbf{a}$$ for $\mathbf{a}_1\in \ell^1(\mathbb{Z}^n)$ and $\mathbf{a}\in \ell^2(\mathbb{Z}^n)$ is faithful (see \cite{GrLe16} and \cite{SaWi20}). Therefore even in the most general case, we have that the $\zeta$-twisted Banach $*$-algebra $\ell^1(\mathbb{Z}^n)$ is densely embedded in $C^*(\mathbb{Z}^n,\zeta).$ We find this to be a useful fact, since then we can always densely define $*$-homomorphisms on $C^*(\mathbb{Z}^n,\zeta)$ through $\ell^1(\mathbb{Z}^n)$. 
\end{remark}
\begin{remark}
    For $M\in \op{M}_n(\mathbb{R})$, we shall use the notation $\zeta_{M}\colon\mathbb{Z}^n\times \mathbb{Z}^n \to \mathbb{T}$ to denote the $2$-cocycle induced by $M$:
    \begin{align}\label{remark:induced cocycle}
        \zeta_M(k,m) = e^{2\pi i k^T M m}, \qquad \forall k,m\in \mathbb{Z}^n.
    \end{align}
\end{remark}

Before we proceed, we recognize, and we will make references to the work of Gjertsen in \cite{Gj23}, who carried out very explicit computations on the noncommutative tori and twisted $C^*$-algebras. Now, if we have a noncommutative $n$-torus $\mathcal{A}_{\Theta}$, then by definition of $\mathcal{A}_{\Theta}$ as a free $*$-algebra generated by $\{U_i(\Theta)\}_{i=1}^n$ and its inherent commutation relations, a typical element in $\mathcal{A}_{\Theta}$ is the norm-limit of the elements of the following form:
\begin{align}\label{typical-nctori}
    \sum_{k\in \mathbb{Z}^n}a(k)U_{1}^{k_1}(\Theta)\cdot \ldots \cdot U_n^{k_n}(\Theta)
\end{align}
with $k=(k_1,...,k_n)\in \mathbb{Z}^n$, and $\mathbf{a}=\{a(k)\}_{k\in \mathbb{Z}^{n}}\in c_{00}(\mathbb{Z}^n).$ If we use the notation $U^k(\Theta)=U^{k_1}_1\cdot \ldots \cdot U_{n}^{k_n}(\Theta)$, then by a careful book-keeping of the appearing phase-factors, as was done in \cite[Lemma 4.2.3]{Gj23}, one finds that
\begin{align}\label{gives-cocycle}
U^k(\Theta)U^m(\Theta) = \op{exp}\left(2\pi i \sum_{i=1}^{n-1}\sum_{j=i+1}^{n}m_i\Theta_{j,i}k_j \right)U^{k+m}.
\end{align}
Notice that 
\begin{align*}
    \sum_{i=1}^{n-1}\sum_{j=i+1}^{n}m_i\Theta_{j,i}k_j&=-\sum_{i=1}^{n-1}\sum_{j=i+1}^{n}m_i\Theta_{i,j}k_j \\
    &=-m^T \Theta^{\op{up}}k\\
    &=k^T(-\Theta^{\op{up}})^Tm=k^T\Theta^{\op{low}}m
\end{align*}
where $\Theta^{\op{up}}$ and $\Theta^{\op{low}}$ are the matrices coinciding with the strictly upper and strictly lower triangular parts of $\Theta$ respectively. We obtain $\zeta_{\Theta^{\op{low}}}\colon \mathbb{Z}^n\times \mathbb{Z}^n\to \mathbb{T}$, the $2$-cocycle induced by $\Theta^{\op{low}}$ using the notation of \eqref{remark:induced cocycle}, and see an emerging structure on $\mathcal{A}_{\Theta}$. Consider two typical elements $\mathbf{a}(\Theta):=\sum_{k\in \mathbb{Z}^n}a(k)U^k(\Theta)$ and $\mathbf{b}(\Theta):=\sum_{m\in \mathbb{Z}^n}b(m)U^m(\Theta)$, then we find that
\begin{align*}
    \mathbf{a}(\Theta)\cdot\mathbf{b}(\Theta) &= \left(\sum_{k\in \mathbb{Z}^n}a(k)U^k(\Theta)\right)\cdot \left(\sum_{k\in \mathbb{Z}^n}b(m)U^m(\Theta)\right)\\
    &=\sum_{k,m\in \mathbb{Z}^n}a(k)b(m)U^k(\Theta)U^m(\Theta) \\
    &=\sum_{k,m\in \mathbb{Z}^n}a(k)b(m)\zeta_{\Theta^{\op{low}}}(k,m) U^{k+m} \\
    &= \sum_{k,m\in \mathbb{Z}^n}a(k)b(m-k)\zeta_{\Theta^{\op{low}}}(k,m-k)U^m.
\end{align*}
The last expression should remind us of twisted convolutions, in fact if we make then the identification $\delta_k\mapsto U^k$ for all $k\in \mathbb{Z}^n$ then we find that the multiplication above is exactly the multiplication on the twisted $C^*$-algebra $C^*(\mathbb{Z}^n,\zeta_{\Theta^{\op{low}}}).$ The explicit isomorphism is provided by $\pi_{\Theta}\colon C^*(\mathbb{Z}^n,\zeta_{\Theta^{\op{low}}})\to \mathcal{A}_{\Theta}$, densely defined via:
\begin{align}\label{canonical-identification}
    \pi_{\Theta}(\mathbf{a}) = \sum_{k\in\mathbb{Z}^n}a(k)U^k(\Theta), \qquad\forall \mathbf{a}\in\ell^1(\mathbb{Z}^n)\hookrightarrow C^*(\mathbb{Z}^n,\zeta_{\Theta^{\op{low}}}).
\end{align}
We can easily verify this by checking that the map is a $*$-homomorphism on the generators. Let $i,j=1,...,n$ then we compute:
\begin{align*}
    \pi_{\Theta}(\delta_{e_i}*\delta_{e_j}) = \pi_{\Theta}(\zeta_{\Theta^{\op{low}}}(e_i,e_j)\delta_{e_i+e_j}) = \zeta_{\Theta^{\op{low}}}(e_i,e_j) \pi_{\Theta}(\delta_{e_i+e_j}).
\end{align*}
Suppose $i\leq j$, then we find that $\zeta_{\Theta^{\op{low}}}(e_i,e_j)=e^{2\pi i 0}=1$, and so $\pi_{\Theta}(\delta_{e_i}*\delta_{e_j})=1\cdot \pi_{\Theta}(\delta_{e_i+e_j})=U_iU_j = \pi_{\Theta}(\delta_{e_i})\pi_{\Theta}(\delta_{e_j}).$ On the other hand, if $i\geq j$, we find that $\zeta_{\Theta^{\op{low}}}(e_i,e_j)=e^{2\pi i \Theta_{i,j}}$, and so $\pi_{\Theta}(\delta_{e_i}*\delta_{e_j})=e^{2\pi i \Theta_{i,j}}U_jU_i=U_iU_j = \pi_{\Theta}(\delta_i)\pi_{\Theta}(\delta_j).$ Furthermore, we also have:
\begin{align*}
    \pi_{\Theta}(\delta_{e_i}^*)=\overline{\zeta_{\Theta}(e_i,-e_i)} \pi_{\Theta}(\delta_{-e_i}) =e^{-2\pi i \Theta_{i,i}^{\op{low}}}U_i^{-1} = U_i^* = \pi_{\Theta}(\delta_{e_i})^*.
\end{align*}
We summarize what we have in a proposition:
\begin{proposition}\label{pi-theta}
    The map $\pi_{\Theta}$ of Equation \eqref{canonical-identification} extends to an isomorphism $\pi_{\Theta}\colon C^*(\mathbb{Z}^n,\zeta_{\Theta^{\op{low}}})\to \mathcal{A}_{\Theta}.$
\end{proposition}
On the other hand, if we consider an arbitrary $2$-cocycle $\zeta\colon\mathbb{Z}^n\times \mathbb{Z}^n \to \mathbb{T}$, it is well-known that $C^*(\mathbb{Z}^n,\zeta)$ is isomorphic to some noncommutative torus. In fact it follows by an elementary computation that:
\begin{align*}
    \delta_{e_i}*_{\zeta}\delta_{e_j}=\zeta(e_i,e_j)\overline{\zeta(e_j,e_i)}\delta_{e_j}*_{\zeta}\delta_{e_i}, \qquad \forall i,j=1,...,n.
\end{align*}
We can always find $\Theta\in \op{Skew}_n(\mathbb{R})$ such that $e^{2\pi i\Theta_{i,j}}=\zeta(e_i,e_j)\overline{\zeta(e_j,e_i)}$. Merely by looking at the generators $\{\delta_{e_i}\}_{i=1}^n$, we find that $C^*(\mathbb{Z}^n,\zeta)\cong \mathcal{A}_{\Theta}$. However for our purposes, we want to find an explicit isomorphism that respects the typical presentation \eqref{typical-nctori} of an element in a noncommutative torus. By also carefully keeping track of the appearing $2$-cocycles, it can be shown, as was done in \cite[Lemma 4.2.4]{Gj23} that for any $2$-cocycle $\zeta\colon\mathbb{Z}^n\times \mathbb{Z}^n\to \mathbb{T}$, there exists a unique $P_{\zeta}\colon\mathbb{Z}^n\to \mathbb{T}$ such that for all $k=(k_1,\ldots,k_n)\in \mathbb{Z}^n$
\begin{align*}
    \delta_{k}= P_{\zeta}(k) \delta_{e_1}^{k_1} *_{\zeta}\ldots *_{\zeta} \delta_{e_n}^{k_n}
\end{align*}
where $\mathbf{a}^{k_i}$ denotes $k_i$-exponentiation of $\mathbf{a}$ in $C^*(\mathbb{Z}^n,\zeta).$ The map $P_{\zeta}$ is explicitly given by:
\begin{align}\label{collected-cocycles}
    P_{\zeta}(k) = \prod_{i=1}^{n-1} \overline{\zeta\left(k_ie_i,\sum_{j=i+1}^{n}k_je_j \right) } \times \prod_{i=1}^{n}\prod_{j=1}^{k_i-1}\overline{\zeta(e_i,(k_i-j)e_i)}
\end{align}
and it can be seen, that as expected, $P_{\zeta}(e_i)=1$ and $P_{\zeta}(e_i+e_j)=\overline{\zeta}(e_i,e_j)$ whenever $i<j$ for all $i,j=1,\ldots,n$. Each element in $ C^*(\mathbb{Z}^n,\zeta)$ can be densely approximated by elements of the form:
\begin{align}\label{typical-a}
    \mathbf{a} = \sum_{k\in \mathbb{Z}^n}a(k)P_{\zeta}(k)\delta_{e_1}^{k_1}*_{\zeta}\ldots *_{\zeta}\delta_{e_n}^{k_n}.
\end{align}
The map
\begin{align*}
    \mathbf{a}=\sum_{k\in \mathbb{Z}^n}a(k)P_{\zeta}(k)\delta_{e_1}^{k_1}*_{\zeta}\ldots *_{\zeta} \delta_{e_n}^{k_n}\mapsto \sum_{k\in \mathbb{Z}^n}a(k)P_{\zeta}(k) U^k(\Theta), \qquad \mathbf{a}\in \ell^1(\mathbb{Z}^n)\hookrightarrow C^*(\mathbb{Z}^n,\zeta)
\end{align*}
extends to a $C^*$-isomorphism, mapping $C^*(\mathbb{Z}^n,\zeta)$ to $\mathcal{A}_{\Theta}.$ Note that this is consistent with Proposition \ref{pi-theta}, given that in the case $\zeta = \zeta_{\Theta^{\op{low}}}$, we have $P_{\zeta_{\Theta^{\op{low}}}}(k)=1$ for all $k\in \mathbb{Z}^n.$ We do not find much practical use for the general isomorphism $C^*(\mathbb{Z}^n,\zeta)\cong \mathcal{A}_{\Theta}$ for an arbitrary $\zeta$. It is enough for us that we always have the general presentation \eqref{typical-a} for twisted $C^*$-algebras, and that we know Proposition \ref{pi-theta}. 
\begin{definition}
    Two (continuous) $2$-cocycles $\zeta_1,\zeta_2: \mathbb{Z}^{n}\times\mathbb{Z}^n\to \mathbb{T}$ are said to be \textbf{cohomologous}\footnote{This definition actually extends to any locally compact Abelian groups}, denoted by $\zeta_1\cong \zeta_2$, if and only if there exists a (continuous) map $\rho: \mathbb{Z}^n\to \mathbb{T}$ called a $\mathbf{1}$\textbf{-cochain} such that
    \begin{align}\label{cohomologous}
        \zeta_1(k,m) = \zeta_2(k,m)\frac{\rho(k+m)}{\rho(k)\rho(m)}, \qquad\forall (k,m)\in\mathbb{Z}^n\times \mathbb{Z}^n.
    \end{align}
\end{definition}
It is not hard to see that cohomology on $2$-cocycles is an equivalence relation. Since we consider normalized $2$-cocycles, then they must satisfy \eqref{cocycle 2} on $\mathbb{Z}^n$, it follows that $1$-cochains satisfy $\rho(0)=1.$ Now, for two normalized $2$-cocycles $\zeta_1,\zeta_2 \colon \mathbb{Z}^n\times \mathbb{Z}^n \to \mathbb{T}$, it turns out that $\zeta_1\cong \zeta_2$ implies $C^*(\mathbb{Z}^n,\zeta_1)\cong C^*(\mathbb{Z}^n,\zeta_2)$. This is known (\eg see the discussion on twisted $C^*$-algebras in \cite{ElLi08}), and it will be useful for us to write out the exact isomorphism that identifies the two.
\begin{proposition}\label{induced-isom-from-cochain}
    If $\zeta_1\cong \zeta_2$ by a $1$-cochain $\rho$ as in \eqref{cohomologous}, define for each $\mathbf{a}\in \ell^1(\mathbb{Z}^n)\hookrightarrow C^*(\mathbb{Z}^n,\zeta_1)$:
    \begin{align*}
    \widetilde{\rho}(\mathbf{a}) = \mathbf{a} \cdot \overline{\rho}.
    \end{align*}
    Then $\widetilde{\rho}$ extends to a $C^*$-isomorphism $\widetilde{\rho}:C^*(\mathbb{Z}^n,\zeta_1)\to C^*(\mathbb{Z}^n,\zeta_2).$
    Furthermore, for each $k=(k_1,\ldots,k_n)\in \mathbb{Z}^n$, fix:
    \begin{align*}
        \rho^k := \rho^{k_1}(e_1)\cdot \ldots \cdot \rho^{k_n}(e_n),
    \end{align*}
    then $P_{\zeta_1}(k)\rho(k)=P_{\zeta_2}(k)\rho^k$ for all $k\in \mathbb{Z}^n.$
\end{proposition}
\begin{proof}
    $\widetilde{\rho}$ is obviously linear, injective, and surjective as a map from $\ell^1(\mathbb{Z}^n)$ to itself. Next we show that it is a $*$-homomorphism on $\ell^1(\mathbb{Z}^n)$ respecting the relevant twisted structures, and it is enough to show this property on the generators. Let $i,j=1,\ldots n.$ The following computations are valid due to \eqref{cohomologous}:
    \begin{align*}
        \widetilde{\rho}(\delta_{e_i}*_{\zeta_1}\delta_{e_j}) &= \widetilde{\rho}(\zeta_1(e_i,e_j)\delta_{e_i+e_j}) = \zeta_1(e_i,e_j)\delta_{e_i+e_j} \overline{\rho}(e_i+e_j)\\
        &=\zeta_1(e_i,e_j)\overline{\zeta_2}(e_i,e_j)\overline{\rho}(e_i+e_j)(\delta_{e_i}*_{\zeta_2}\delta_{e_j})\\ &=(\overline{\rho}(e_i)\delta_{e_i}) *_{\zeta_2}(\overline{\rho}(e_j)\delta_{e_j})=\widetilde{\rho}(\delta_{e_i})*_{\zeta_{2}}\widetilde{\rho}(\delta_{e_j}).
    \end{align*}
    Furthermore,
    \begin{align*}
        \widetilde{\rho}(\delta_{e_i}^{*_{\zeta_1}})&= \widetilde{\rho}(\overline{\zeta_1}(e_i,-e_i)\delta_{-e_i}) = \overline{\zeta_1}(e_i,-e_i)\delta_{-e_i}\overline{\rho}(-e_i)\\
        &=\overline{\zeta_{2}}(e_i,-e_i)\rho(e_i)\delta_{-e_i} = \rho(e_i)\delta_{e_i}^{*_{\zeta_2}} \\
        &=(\overline{\rho}(e_i)\delta_{e_i})^{*_{\zeta_2}} = \widetilde{\rho}(\delta_{e_i})^{*_{\zeta_2}} .
    \end{align*}
    Therefore, $\widetilde{\rho}$ extends to an injective and surjective $*$-homomorphism $\widetilde{\rho}: C^*(\mathbb{Z}^n,\zeta_1)\to C^*(\mathbb{Z}^n, \zeta_2)$. Recall that injective $*$-homomorphisms on $C^*$-algebras are automatically isometric, therefore $\widetilde{\rho}$ is a $C^*$-isomorphism. Next, we write $\delta_{k} = P_{\zeta_1}(k)\delta_{e_1}^{k_1}*_{\zeta_1}\ldots *_{\zeta_1} \delta_{e_n}^{k_n},$ hence we have:
\begin{align*}
    \widetilde{\rho}(\delta_k) = \delta_k\overline{\rho}(k) = \widetilde{\rho}(P_{\zeta_1}(k)\delta_{e_1}^{k_1}*_{\zeta_1}\ldots *_{\zeta_1} \delta_{e_n}^{k_n}) =  P_{\zeta_1}(k)\overline{\rho}^k\delta_{e_1}^{k_1}*_{\zeta_2}\ldots *_{\zeta_2} \delta_{e_n}^{k_n} = P_{\zeta_1}(k)\overline{\rho}^k \overline{P_{\zeta_2}}(k)\delta_k,
\end{align*}
therefore $P_{\zeta_1}(k)\rho(k)=P_{\zeta_2}(k)\rho^k.$
\end{proof}
 \begin{remark}\label{cocycle difference}
    Regarding cohomology, the $2$-cocycle $c$ that we used to give $\ell^1(\mathbb{Z}^{2d})$ a twisted structure is different from the one suggested in some references, e.g \cite{FeKa04,GrLe04}. 
    This is because the time-frequency shifts considered in these papers were of the form $T_xM_{\omega}$, amounting to a cocycle $c'(z_1,z_2)=e^{2\pi i\, z_2^T Kz_1}$ instead of the one given by \eqref{heis-cos}. 
    The difference however is inconsequential since $c$ and $c'$ are cohomologous via the $1$-cochain $\rho(z)=e^{-2\pi i \omega\cdot x}$ for $z=(x,\omega)\in \mathbb{R}^{2d}$. Therefore they generate isomorphic $C^*$-algebras.
\end{remark}
\section{Results}\label{section:results}

\subsection{Deforming Noncommutative Tori}\label{subsec:general-nctori-deform}

Consider the noncommutative $n$-tori determined by skew-symmetric matrices $\Theta\in\operatorname{Skew}_n(\mathbb R)$ (or by $\lambda\in \mathbb{T}^{n(n-1)/2}$, see Remark \ref{skew}). 
We see that as we vary $\Theta$ (or $\lambda$) we find possibly different  $C^*$-algebras, hence it is a reasonable question to ask if families of noncommutative $n$-tori define a $C^*$-bundle over the space $\op{Skew}_n(\mathbb{R})$ (or $\mathbb{T}^{n(n-1)/2}$) as in Definition \ref{bundledef}. 
The answer is affirmative, but the precise formulation can be quite technical, there are related `deformation' results in this direction. The case $n=2$ is quite well studied in the past, note that in this case there really is just one parameter $\lambda\in \mathbb{T}$ fully determining the noncommutative $2$-torus $\mathcal{A}_{\lambda}$. It was noted by Elliot in \cite{El82} that the $C^*$-algebras $\{\mathcal{A}_{\lambda}\}_{\lambda\in \mathbb{T}}$ do form a $C^*$-bundle, but the associated $C^*$-algebra of sections (that vanish at infinity) was not really expounded on. It turns out, due to Anderson and Paschke \cite{AnPa89}, that if $H_3$ is the so-called discrete $3$-dimensional Heisenberg group, then the group $C^*$-algebra $C^*(H_3)$ is the section $C^*$-algebra for the $C^*$-bundle of noncommutative $2$-tori. A natural generalization of the Heisenberg group $H_3$ consistent with this deformation result is given by the the free nilpotent group of class $2$ and rank $n$, denoted by $G(n)$. In this case, we have $G(2)=H_3$, while in general $C^*(G(n))$ decomposes as a $C^*$-bundle of higher-dimensional noncommutative $n$-tori fibered over $\mathbb{T}^{n(n-1)/2}$ \cite{Omla15}.

There are, on the other hand, explicit results regarding $\frac{1}{2}$-Hölder continuity of the unitary generators of the noncommutative $n$-tori. 
Let us again go back to the $n=2$ case, the classic Haagerup and  Rørdam result \cite{HaRo95} is that for a fixed Hilbert $H$, there exist paths $u,v\colon [0,1]\to \mathcal{L}(H)$ of unitary operators such that for all $\theta\in [0,1]$, $u(\theta)v(\theta) = e^{2\pi i\, \theta}v(\theta)u(\theta)$ and there exists a $C>0$ where $\max\{\|u(\theta)-u(\theta')\|,\|v(\theta)-v(\theta')\|\}\leq C |\theta-\theta'|^{1/2}$. 
This $\frac{1}{2}$-Hölder continuity generalizes to $n>2$, as was shown by Gao in \cite{Ga18}. 
The deformation results on the unitary generators can be lifted to a deformation result of their noncommutative polynomials (see the short proof of Theorem \ref{malte} below), which would allow us to construct a $C^*$-bundle whose fibers coincide with the higher dimensional noncommutative tori. 
We shall also borrow techniques from the theory of dilations \cite{GeSh20,GPSS21}, which not only simplifies Gao's result but will also allow us to obtain more concrete deformation estimates.

Let us recall the isomorphism $\pi_{\Theta}:C^*(\mathbb{Z}^n,\zeta_{\Theta^{\op{low}}})\to \mathcal{A}_{\Theta}$ from Equation \eqref{canonical-identification}, we have the following estimate involving the weighted sequence space $\ell^1_{\nu}(\mathbb{Z}^{n})$ (treated as a dense subspace of $C^*(\mathbb{Z}^n,\zeta_{\Theta^{\op{low}}})$).
\begin{theorem}\label{malte}
    There is a constant $C>0$ such that for all $\mathbf{a}\in \ell^1_{\nu}(\mathbb{Z}^n)$ and $\Theta,\Theta'\in \op{Skew}_n(\mathbb{R}):$
    \begin{align*}
        \bigl| \|\pi_{\Theta}(\mathbf{a})\|-\|\pi_{\Theta'}(\mathbf{a})\| \bigr|\leq C \|\mathbf{a}\|_{\ell^1_{\nu}} \|\Theta-\Theta'\|^{1/2}.
    \end{align*}
\end{theorem}

\noindent Let us present two proofs, a short one with bad control on $C$ and a slightly longer one with $C=2\sqrt{\pi}<4$.
\begin{proof}[Short Proof]
     By \cite[Theorem 1.1]{Ga18}, there exist (automatically faithful) representations $U_i(\Theta)\mapsto V_i(\Theta)\in \mathcal{L}(H)$ on the same Hilbert space $H$ such that 
     \[
        \|V_i(\Theta)-V_i(\Theta')\|
        \leq \hat C \Bigl(\sum_{k=1}^n (\Theta_{ij}-\Theta'_{ij})^{\frac{1}{2}}\Bigr)
        \leq C \|\Theta-\Theta'\|^{\frac{1}{2}}
    \] 
    for some constants $\hat C,C\in\mathbb R_+$, which may depend on $n$, and all $\Theta,\Theta'\in \operatorname{Skew}_n(\mathbb{R})$. 
    Let $\mathbf{a}\in \ell^1_{\nu}(\mathbb{Z}^n)$, $\nu\geq1$. Note that $\|\pi_\Theta(\mathbf a)\|=\|\sum_{k\in\mathbb Z^n}a(k)V^k(\Theta)\|$. 
    By an elementary estimate (cf.\ \cite[Lemma 4.1]{GeSh20}), for each $k\in\mathbb Z^{n}$ we have
  \[
  \|V^k(\Theta)- V^k(\Theta')\|\leq |k| \max_i\|V_i(\Theta)-V_i(\Theta')\|\leq C |k| \|\Theta-\Theta'\|^{\frac{1}{2}} 
  \]
  and therefore
  \begin{multline*}
    \biggl|\Bigl\|\sum_{k\in\mathbb Z^d} a(k) V^k(\Theta)\Bigr\| - \Bigl \|\sum_{k\in\mathbb Z^d} a(k) V^k(\Theta')\Bigr\|\biggr|
    \leq\Bigl\|\sum_{k\in\mathbb Z^d} a(k) V^k(\Theta) -  \sum_{k\in\mathbb Z^d} a(k) V^k(\Theta')\Bigr\|
    \\
    \leq \sum_{k\in\mathbb Z^d} |a(k)|\| V^k(\Theta)- V^k(\Theta')\|\leq C \|\mathbf{a}\|_{\ell^1_\nu} \|\Theta-\Theta'\|^{\frac{1}{2}},
  \end{multline*}
  as claimed.
\end{proof}
\begin{proof}[Proof with $C=2\sqrt\pi$.]
    Fix $\Theta,\Theta'\in\operatorname{Skew}_n(\mathbb{R})$. By \cite[Theorem 5.4]{GPSS21}, there exist faithful representations $U_i(\Theta)\mapsto W_i(\Theta)\in \mathcal{L}(H)$ of $\mathcal A_\Theta$ on some Hilbert space $H$ and $U_i(\Theta')\mapsto W_i(\Theta')$ of $\mathcal A_{\Theta'}$ on a Hilbert space $H'$ containing $H$ as a subspace such that compression of $W_i(\Theta')$ to $H$ yields $P_H W_i(\Theta')|_H=cW_{\Theta}$ with $c=\exp (-\frac{\pi}{2}\|\Theta'-\Theta\|)$.\footnote{In comparison to \cite{GPSS21}, there is a discrepancy in the normalization of the matrices $\Theta$ parametrizing noncommutative tori by a factor of $2\pi$, which explains the different value of $c$.}  Denote by $R_i:=P_{H^\perp} W_i(\Theta')|_{H^\perp}\in B(H^{\perp})$ the corresponding compression of $W_i(\Theta')$ to $H^\perp\subset H'$, so that $W_i(\Theta)\oplus R_i\in B(H')$. Note that $\|R_i\|\leq \|W_i(\Theta')\|=1$, so each $R_i$ is a compression. From \cite[Lemma 4.3]{GeSh20} and the elementary inequality $1-e^{-x}\leq x$ for $x\in\mathbb R$, we conclude that
    \begin{align*}
        \|W_i(\Theta') - W_i(\Theta) \oplus R_i\| \leq 2\sqrt{1-c^2}\leq 2\sqrt{1-e^{-\pi\|\Theta-\Theta'\|}}\leq 2\sqrt{\pi}\|\Theta-\Theta'\|^{\frac{1}{2}}.
    \end{align*}
     Let $\mathbf{a}\in \ell^1_{\nu}(\mathbb{Z}^n)$, $\nu\geq1$. By faithfulness,  $\|\pi_\Theta(\mathbf a)\|=\|\sum_{k\in\mathbb Z^n}a(k)W^k(\Theta)\|$ and $\|\pi_{\Theta'}(\mathbf a)\|=\|\sum_{k\in\mathbb Z^n}a(k)W^k(\Theta')\|$. 
     As in the first proof, it follows that for each $k\in\mathbb Z^{n}$ we have
     \[
  \|W^k(\Theta')- W^k(\Theta)\oplus R^k\|\leq |k| \max_i\|W_i(\Theta')-W_i(\Theta)\oplus R_i\|\leq 2\sqrt{\pi}|k| \|\Theta-\Theta'\|^{\frac{1}{2}} 
\]
and therefore
    \[
    \Bigl\|\sum_{k\in\mathbb Z^n} a(k)W^k(\Theta') - \sum_{k\in\mathbb Z^n} a(k) (W^k(\Theta)\oplus R^k)\Bigr\|
    \leq  2\sqrt{\pi}\|\mathbf{a}\|_{\ell^1_\nu} \|\Theta-\Theta'\|^{\frac{1}{2}}
  \]
  This yields
    \begin{align*}
        \MoveEqLeft 
        \|\pi_{\Theta}(\mathbf{a})\| 
        = \Bigl\|\sum_{k\in\mathbb Z^n} a(k) W^k(\Theta)\Bigr\|\\
      &\leq \max\left(\Bigl\|\sum_{k\in\mathbb Z^n} a(k) W^k(\Theta)\Bigr\|,\Bigl\|\sum_{k\in\mathbb Z^n} a(k) R^k\Bigr\|\right)\\
      &= \Bigl\|\sum_{k\in\mathbb Z^n} a(k) (W^k(\Theta)\oplus R^k)\Bigr\|\\
      &\leq \Bigl\|\sum_{k\in\mathbb Z^n} a(k) W^k(\Theta')\Bigr\| + \Bigl\|\sum_{k\in\mathbb Z^n} a(k) (W^k(\Theta)\oplus R^k) - \sum_{k\in\mathbb Z^n} a(k) W^k(\Theta')\Bigr\|\\
      &\leq \|\pi_{\Theta'}(\mathbf{a})\| + 2\sqrt{\pi}\|\mathbf{a}\|_{\ell^1_\nu} \|\Theta-\Theta'\|^{\frac{1}{2}}.
    \end{align*}
    Since the roles of $\Theta$ and $\Theta'$ are interchangeable, we also have
    \[\|\pi_{\Theta'}(\mathbf{a})\|\leq \|\pi_{\Theta}(\mathbf{a})\| + 2\sqrt{\pi}\|\mathbf{a}\|_{\ell^1_\nu} \|\Theta-\Theta'\|^{\frac{1}{2}}\]
    and the proof is complete.
\end{proof}

This result is enough to give us a $C^*$-bundle over $\text{Skew}_n(\mathbb{R})$ whose fibers coincide with $\mathcal{A}_{\Theta}$ via Theorem \ref{bbundleconstruction} with a local approximation section that we can identify 
 with $\ell^1_{\nu}(\mathbb{Z}^{2d})$ (by seeing sequences inside it as constant sections since we have a dense embedding $\ell^1_{\nu}(\mathbb{Z}^{2d})\hookrightarrow \mathcal{A}_{\Theta}$ for each $\Theta$). 
 However, we have no need for such a general result in the current paper, and we would rather defer the relevant constructions to the sequel where the resulting bundles come from noncommutative tori generated by lattices in the time-frequency plane.

\begin{remark}\label{remark:spectrum}
    With the same methods, one can prove a useful estimate for the Hausdorff distance between the spectra $\sigma(\pi_\Theta(\mathbf a))$ of selfadjoint operators $\pi_\Theta(\mathbf a)$. Since we have limited use for it, we sketch the argument only roughly and leave the details to the ambitious reader. Suppose that $\Theta,\Theta_0\in\operatorname{Skew}_n(\mathbb R)$ and $\mathbf a,\mathbf a_0 \in \ell^1_\nu(\mathbb Z^n)$ such that $\pi_\Theta(\mathbf a)$ and $\pi_{\Theta_0}(\mathbf{a}_0)$ are selfadjoint (note that rarely happens for $\mathbf a=\mathbf a_0$). Then, in the notation of the second proof of Theorem \ref{malte},
    \begin{align*}
        \MoveEqLeft 
        \sigma(\pi_{\Theta}(\mathbf a))
        = \sigma\Bigl(\sum_{k\in\mathbb Z^n} a(k) W^k(\Theta)\Bigr)\\
        &\subset 
        \sigma\Bigl(\sum_{k\in\mathbb Z^n} a(k) (W^k(\Theta)\oplus R^k)\Bigr)\\
        &\subset \sigma(\pi_{\Theta_0}(\mathbf a_0)) +  [-1,1]\cdot\Bigl\|\sum_{k\in\mathbb Z^n} a(k) (W^k(\Theta)\oplus R^k) - \sum_{k\in\mathbb Z^n} a_0(k) W^k(\Theta_0)\Bigr\|\\
        &\subset \sigma(\pi_{\Theta_0}(\mathbf a_0)) + [-1,1]\cdot (2\sqrt{\pi}\|\mathbf{a}\|_{\ell^1_\nu} \|\Theta-\Theta_0\|^{\frac{1}{2}} + \|\mathbf a -\mathbf a_0\|_{\ell^1}).
    \end{align*}
\end{remark}

\subsection{The Noncommutative Tori Generated by Lattices}\label{subsec: nctori by lattices}
Our goal here is to generate a skew-symmetric matrix $\Theta_{L}\in \op{Skew}_{2d}(\mathbb{R})$ such that $\mathcal{A}_{\Theta_L}\cong C^*(\mathbb{Z}^{2d},c_L)=A_L$. To start, let $\{e_1,e_2,\ldots,e_{2d}\}$ be the standard basis for $\mathbb{R}^{2d}$, consider the family of unitary operators via time-frequency shifts $\{\pi(L e_1),\pi(Le_2),\ldots,\pi(Le_{2d})\}\subseteq \mathcal{L}(L^2(\mathbb{R}^{d}))$. The commutation-relations \eqref{symp-comm-rel} for this family say that
\begin{align*}
    \pi(Le_i)\pi(Le_j) =c_s(Le_i,Le_j)\pi(Le_j)\pi(Le_i),\qquad \forall i,j = 1,\ldots,2d.
\end{align*}
    where $c_s(Le_i,Le_j)=e^{-2\pi i\, (Le_i)^T J (Le_j)}$, and so we define the skew-symmetric matrix associated to $L$ via $(\Theta_L)_{i,j} = (Le_j)^T (-J) (Le_i) = -e_j^T(L^TJL)e_i$, \ie $\Theta_L = -L^T J L$. For the rest of the paper, we now make the identification $\mathcal{A}_{\Theta_L}=C^*(\{Le_i\}_{i=1}^{2d})\subseteq \mathcal{L}(L^2(\mathbb{R}^d)).$ As we have discussed in Section \ref{subsec:nctori}, by taking $\Theta=\Theta_L$, we know that $\pi_{\Theta_L}: C^*\left(\mathbb{Z}^{2d},\zeta_{\Theta_L^{\op{low}}}\right)\to \mathcal{A}_{\Theta_L}$ of Equation \eqref{canonical-identification} is a $C^*$-isomorphism, but for our purposes, we want to be able to lift the deformation results on $\pi_{\Theta_L}$ to the faithful $*$-representations, say $\overline{\pi}_L$, of Equation \eqref{a-integrated-rep}. To this end, we have the following result.
    \begin{lemma}\label{lemma: tf-cocycle-cohom-canoncocycle}
        For a fixed $L\in \op{GL}_{2d}(\mathbb{R})$, consider the skew-symmetric matrix $\Theta_L = -L^TJL \in \op{Skew}_{2d}(\mathbb{R}).$ The induced $2$-cocycle $\zeta_{\Theta_L^{\op{low}}}$ is cohomologous to the $2$-cocycle $c_L$ by the $1$-cochain $\rho_L:\mathbb{Z}^{2d}\to \mathbb{T}$ given by $\rho_L(k)=\exp(-\pi i k^T(L^TKL+\Theta^{\op{low}}_L)k)$ for all $k\in \mathbb{Z}^{2d}$. It satisfies:
        \begin{align*}
            c_L(k,m) = \zeta_{\Theta^{\op{low}}_L}(k,m)\frac{\rho_L(k+m)}{\rho_L(k)\rho_L(m)}, \qquad \forall k,m\in \mathbb{Z}^{2d}.
        \end{align*}
    \end{lemma}
    \begin{proof}
        It follows from a straightforward computation that $$\frac{\rho_L(k+m)}{\rho_L(k)\rho_L(m)}=\exp\left(-\pi i k^T(L^T(K+K^T)L + (\Theta_L+(\Theta_L^{\op{low}})^T))m\right)$$
    Therefore we have
    \begin{align*}
        \zeta_{\Theta^{\op{low}}_L}(k,m)\frac{\rho_L(k+m)}{\rho_L(k)\rho_L(m)}&=\exp(2\pi i k^T\Theta_L^{\op{low}}m)\exp\left(-\pi i k^T(L^T(K+K^T)L + (\Theta_L+(\Theta_L^{\op{low}})^T))m\right) \\
        &=\exp\left(\pi i k^T(2\Theta_L^{\op{low}}-L^T(K+K^T)L -\Theta_L^{\op{low}}-(\Theta_L^{\op{low}})^T )m\right)
    \end{align*}
    We compute that: $2\Theta_L^{\op{low}}-\Theta_L^{\op{low}}-(\Theta_L^{\op{low}})^T =\Theta_L$ since $\Theta_L$ is skew-symmetric. But then, $\Theta_L -L^T(K+K^T)L = -L^TJL-L^T(K+K^T)L = -L^T(J+K+K^T)L = -2L^TKL.$ Finally, we obtain
    \begin{align*}
        \zeta_{\Theta^{\op{low}}_L}(k,m)\frac{\rho_L(k+m)}{\rho_L(k)\rho_L(m)} = e^{\pi i k^T(-2L^TKL)m} = e^{-2\pi i k^T(L^TKL)m} = c_L(n,m).
    \end{align*}
    \end{proof}
    Let us now take note of the collected $2$-cocycles $P_{c_L}$ and $P_{\zeta_{\Theta_L^{\op{low}}}}$ as in \eqref{collected-cocycles} generated by $c_L$ and $\zeta_{\Theta_L^{\op{low}}}$. We obtain:
    \begin{align}\label{pee-el}
        P_{c_L}(k) = \prod_{i=1}^{2d-1}\overline{c}\left( Lk_ie_i,\sum_{j=1+1}^{2d}Lk_je_j\right)\times \prod_{i=1}^{2d}\overline{c}(Le_i,Le_i)^{n_i(n_i-1)/2}.
    \end{align}
    On the other hand, we have $P_{\zeta_{\Theta_L^{\op{low}}}}(k)=1$ for all $k\in\mathbb{Z}^{2d}$. Consider now the induced isomorphism $\widetilde{\rho}_L\colon A_L=C^*(\mathbb{Z}^{2d},c_L)\to C^*(\mathbb{Z}^{2d},\zeta_{\Theta_L^{\op{low}}})$ (see Proposition \ref{induced-isom-from-cochain}), we have two representations that we want to compare: $\overline{\pi}_L:A_L\to \mathcal{A}_{\Theta_L}$ and $\pi_{\Theta_L}\circ \widetilde{\rho_L}\colon A_L\to \mathcal{A}_{\Theta_L}$. They do not exactly give commutative maps, but they still behave in an expected way. We shall fix the following: $\rho_L^{(\cdot)}:\mathbb{Z}^{2d}\to \mathbb{T}$ is the map $k\mapsto \rho_L^{k}$.
    \begin{proposition}
        Define, for each $\mathbf{a}\in C^*(\mathbb{Z}^{2d},c_L)=A_L$ the element
        \begin{align*}
            \mathbf{a}^L :=\widetilde{\rho_L}(\mathbf{a}\cdot \rho_L^{(\cdot)})\in C^*(\mathbb{Z}^{2d},\zeta_{\Theta_L^{\op{low}}}).
        \end{align*}
        Explicitly, we have
        \begin{align}\label{a-twisted}
            \mathbf{a}^L = \sum_{k\in \mathbb{Z}^{2d}}a(k)P_{c_L}(k)(\delta_{e_1}^{k_1}*_{\zeta_{\Theta_L^{\op{\op{low}}}}}\ldots *_{\zeta_{\Theta_L^{\op{\op{low}}}}} \delta_{e_{2d}}^{k_{2d}})=\mathbf{a}\cdot P_{c_L} \in C^*(\mathbb{Z}^{2d},\zeta_{\Theta_L^{\op{low}}}).
        \end{align}
         Furthermore,
         \begin{align}\label{isomorphisms-with-cocycle}
             \overline{\pi}_L(\mathbf{a}) = \pi_{\Theta_L}(\mathbf{a}^L)
         \end{align}
    \end{proposition}
    \begin{proof}
        We have 
        \begin{align*}
            \mathbf{a}^L&=\widetilde{\rho_L}(\mathbf{a}\cdot \rho^{(\cdot)}_L) =\widetilde{\rho_L}\left(\sum_{k\in \mathbb{Z}^{2d}}a(k)P_{c_L}(k)\rho_L^k \delta_{e_1}^{k_1}*_{c_L}\ldots *_{c_L} \delta_{e_{2d}}^{k_{2d}} \right) \\
            &=\sum_{k\in \mathbb{Z}^{2d}}a(k)P_{c_L}(k)\rho^k_L \widetilde{\rho_L}(\delta_{e_1}^{k_1}*_{c_L}\ldots *_{c_L} \delta_{e_{2d}}^{k_{2d}})\\
            &=\sum_{k\in \mathbb{Z}^{2d}}a(k)P_{c_L}(k)\rho_L^k \overline{\rho_L}^k(\delta_{e_1}^{k_1}*_{\zeta_{\Theta_L^{\op{\op{low}}}}}\ldots *_{\zeta_{\Theta_L^{\op{\op{low}}}}} \delta_{e_{2d}}^{k_{2d}})\\
            &=\sum_{k\in \mathbb{Z}^{2d}}a(k)P_{c_L}(k)(\delta_{e_1}^{k_1}*_{\zeta_{\Theta_L^{\op{\op{low}}}}}\ldots *_{\zeta_{\Theta_L^{\op{\op{low}}}}} \delta_{e_{2d}}^{k_{2d}})\\
            &=\sum_{k\in \mathbb{Z}^{2d}}a(k)P{c_L}(k)\delta_k= \mathbf{a}\cdot P_{c_L}\in C^*(\mathbb{Z}^{2d},\zeta_{\Theta_L^{\op{low}}})
        \end{align*}
        Next we compute:
        \begin{align*}
            \pi_{\Theta_L}(\mathbf{a}^L) &= \sum_{k\in \mathbb{Z}^{2d}}a(k)P_{c_L}(k)\pi_{\Theta_L}\left( \delta_{e_1}^{k_1}*_{\zeta_{\Theta_L^{\op{low}}}}\ldots *_{\zeta_{\Theta_{L}^{\op{low}}}}\delta_{e_{2d}}^{k_{2d}} \right)\\
            &=\sum_{k\in \mathbb{Z}^{2d}}a(k)P_{c_L}(k)\pi(Le_1)^{k_1}\cdot \ldots \cdot \pi(Le_{2d})^{k_{2d}}\\
            &=\overline{\pi}_L\left(\sum_{k\in \mathbb{Z}^{2d}}a(k) P_{c_L}(k)\delta_{e_1}^{k_1}*_{c_L}\ldots *_{c_L}\delta_{e_{2d}}^{k_{2d}} \right)\\
            &=\overline{\pi}_L(\mathbf{a}).
        \end{align*}
        as required. 
    \end{proof}
\begin{lemma}\label{for-pl}
    For each fixed $n\in \mathbb{Z}^{2d}$, the map $L\mapsto P_{c_L}(n)$ is continuous.
\end{lemma}
\begin{proof}
    Looking at Equation \eqref{pee-el}, $P_{c_L}(n)$ is a finite product of the Heisenberg $2$-cocycles of the form $c(Lk,Lm)$ for some fixed $k,m\in \mathbb{Z}^{2d}$. The lemma follows if we can show that $L\mapsto c(Lk,Lm)$ is continuous for any fixed $k,m\in \mathbb{Z}^{2d}$. We have from Equation \eqref{heis-cos} that
    \begin{align*}
        c(Lk,Lm) = e^{-2\pi i\, (Lk)^T K (Lm)}.
    \end{align*}
    Since the complex exponential function is continuous, then it is sufficient to show that $L \mapsto (Ln)^T K (Lm)$ is continuous. Now it follows from a simple application of triangle-inequality that:
    \begin{align*}
        \|(Ln)^T K (Lm) - (L_0n)K (L_0m)\| &\leq \|n^T(L-L_0)KLm\| + \|n^T L_0 K(L-L_0)m\|\\
        &\leq \|n\|\|m\|(\|L\|+\|L_0\|)\|L-L_0\|. 
    \end{align*}
    Therefore $(Ln)^T K (Lm)\to (L_0n)^T K (L_0m)$ as $L\to L_0$ as required.
\end{proof}
 Now we are ready to apply Theorem \ref{malte} to the noncommutative $2d$-tori generated by lattices in $\mathbb{R}^{2d}$. To make sense of our computations below, recall that the weighted sequence space $\ell^1_{\nu}(\mathbb{Z}^{2d})$ is densely embedded inside all $C^*(\mathbb{Z}^{2d},\zeta)$ for arbitrary $2$-cocycle $\zeta.$
\begin{lemma}\label{a-l-tech-est}
    For each $\mathbf{a}\in \ell^1_{\nu}(\mathbb{Z}^{2d})$ and $L\in \op{GL}_{2d}(\mathbb{R})$: $\|\mathbf{a}\|_{\ell^1_{\nu}}=\|\mathbf{a}^L\|_{\ell^1_{\nu}}$. Moreover, the map $\op{GL}_{2d}(\mathbb{R})\ni L\mapsto \mathbf{a}^L\in \ell^1_{\nu}(\mathbb{Z}^{2d})$ is continuous.
\end{lemma}
\begin{proof}
    Since each $P_{c_L}(k)\in \mathbb{T}$, we have that
    \begin{align*}
        \|\mathbf{a}^L\|_{\ell^1_{\nu}} = \sum_{k\in \mathbb{Z}^{2d}} |a(k)\nu(k) P_{c_L}(k)| = \sum_{k\in \mathbb{Z}^{2d}}|a(k)|\nu(k) = \|\mathbf{a}\|_{\ell^1_{\nu}}.
    \end{align*}\
    Note that for all $L,L_0\in \op{GL}_{2d}(\mathbb{R})$ and $k\in \mathbb{Z}^{2d}$, $|P_{c_L}(k)-P_{c_{L_0}}(k)|\leq 2$, so it follows from the continuity result Lemma \ref{for-pl} and the Lebesgue-dominated convergence theorem\footnote{The usual dominated theorem works because $\op{GL}_{2d}(\mathbb{R})$ is second countable, therefore $\lim_{L\to L_0} \|\mathbf{a}^L-\mathbf{a}^{L_0}\|_{\ell^1_{\nu}}=0$ if and only if for all sequence $L_n\to L_0$, $\lim_{n\to \infty}\|\mathbf{a}^{L_n}-\mathbf{a}^{L_0}\|_{\ell^1_{\nu}}=0$.} that
    \begin{align*}
        \lim_{L\to L_0}\|\mathbf{a}^L-\mathbf{a}^{L_0}\|_{\ell^1_{\nu}} &= \lim_{L\to L_0}\sum_{k\in \mathbb{Z}^{2d}}|a(k)|\nu(k) |P_{c_L}(k)-P_{c_{L_0}}(k)|\\
        &=\sum_{k\in\mathbb{Z}^{2d}}|a(k)|\nu(k) \lim_{L\to L_0}|P_{c_L}(k)-P_{c_{L_0}}(k)| =0
    \end{align*}
    as required.
\end{proof}
\begin{remark}\label{a-l-tech-est2}
    Lemma \ref{a-l-tech-est} is also true if we consider the unweighted case. That is, $\op{GL}_{2d}(\mathbb{R})\ni L\mapsto \mathbf{a}^L\in \ell^1(\mathbb{Z}^{2d})$ is continuous for all $\mathbf{a}\in \ell^1(\mathbb{Z}^{2d}).$
\end{remark}
Now as a corollary of Theorem \ref{malte}, we have that:
\begin{corollary}\label{fund-c-continuity}
    There is a constant $C>0$ such that for all $\mathbf{a}\in \ell^1_{\nu}(\mathbb{Z}^{2d})$ and $L_0,L\in \op{GL}_{2d}(\mathbb{R})$:
    \begin{align*}
        \bigl|\|\overline{\pi}_L(\mathbf{a})\|- \|\overline{\pi}_{L_0}(\mathbf{a})\| \bigr|&\leq C \|\mathbf{a}\|_{\ell^1_{\nu}} \|L^TJL-L_0^TJL_{0}\|^{1/2}+ \|\mathbf{a}^L-\mathbf{a}^{L_0}\|_{\ell^1} \\
        &\leq C\|\mathbf{a}\|_{\ell^1_{\nu}} (\|L\|+\|L_0\|)^{1/2}\|L-L_0\|^{1/2} + \|\mathbf{a}^L-\mathbf{a}^{L_0}\|_{\ell^1}
    \end{align*}
    As a corollary, for each $\mathbf{a}\in \ell^1_{\nu}(\mathbb{Z}^{2d})$, the map $L \mapsto \|\overline{\pi}_{L}(\mathbf{a})\|$ is continuous. 
\end{corollary}
\begin{proof}
    Fix $\mathbf{a}\in \ell^1_{\nu}(\mathbb{Z}^{2d})$, and $L,L_0\in \op{GL}_{2d}(\mathbb{R})$. An application of triangle-inequality gives us, along with Equation \eqref{isomorphisms-with-cocycle}:
    \begin{align*}
        \bigl|\|\overline{\pi}_{L}(\mathbf{a})\|-\|\overline{\pi}_{L_0}(\mathbf{a})\| \bigr|
        &\leq \bigl|\|\pi_{\Theta_L}(\mathbf{a}^L)\|-\|\pi_{\Theta_{L_0}}(\mathbf{a}^L)\| \bigr|+\bigl| \|\pi_{\Theta_{L_0}}(\mathbf{a}^L)\|-\| \pi_{\Theta_{L_0}}(\mathbf{a}^{L_0})\| \bigr|\\
        &\leq \bigl|\|\pi_{\Theta_L}(\mathbf{a}^L)\|-\|\pi_{\Theta_{L_0}}(\mathbf{a}^L)\| \bigr| + \|\pi_{\Theta_{L_0}}(\mathbf{a}^L-\mathbf{a}^{L_0})\|.
    \end{align*}
    Time-frequency shifts are unitary, so it follows that $\|\pi_{\Theta_{L_0}}(\mathbf{a}^L - \mathbf{a}^{L_0})\|\leq \|\mathbf{a}^{L}-\mathbf{a}^{L_0}\|_{\ell^1}$. On the other hand, an application of Theorem \ref{malte} implies 
    \[\bigl|\|\pi_{\Theta_L}(\mathbf{a}^{L})\| -\| \pi_{\Theta_{L_0}}(\mathbf{a}^{L})\| \bigr| 
    \leq C \|\mathbf{a}^L\|_{\ell^1_{\nu}} \|\Theta_L-\Theta_{L_0}\|^{1/2}.\] 
    Because $\|\mathbf{a}^L\|_{\ell^1_{\nu}} = \|\mathbf{a}\|_{\ell^1_{\nu}}$, we finally obtain
    \begin{align}\label{take-limit-to-l0}
         \bigl|\|\overline{\pi}_L(\mathbf{a})\|-\|\overline{\pi}_{L_0}(\mathbf{a})\|\bigr|
         \leq C \|\mathbf{a}\|_{\ell^1_{\nu}} \|\Theta_L-\Theta_{L_0}\|^{1/2}+ \|\mathbf{a}^L-\mathbf{a}^{L_0}\|_{\ell^1}.
    \end{align}
    Similar to what we did in the proof of Lemma \ref{for-pl}, we can estimate $\|\Theta_L-\Theta_{L_0}\|=\|L^TJL-L^T_{0}JL_{0}\|\leq (\|L\|+\|L_0\|)\|L-L_0\|$. 
    If we combine this estimate with the continuity result of Lemma \ref{a-l-tech-est}, and Remark \ref{a-l-tech-est2}, then we find that as $L\to L_0$, we have $\bigl| \|\overline{\pi}_{L}(\mathbf{a})\| - \|\overline{\pi}_{L_0}(\mathbf{a})\| \bigr| \to 0$. 
    This proves continuity of $L\mapsto \|\overline{\pi}_L (\mathbf{a})\|.$
\end{proof}
We are now ready to define our $C^*$-bundle over $\op{GL}_{2d}(\mathbb{R})$ whose fibers coincide with noncommutative tori generated by lattices.
To start, define $\mathcal{C} = \bigsqcup_{L\in \op{GL}_{2d}(\mathbb{R})}A_L$, then define the obvious bundle projection map $\rho\colon  \mathcal{C}\to \op{GL}_{2d}(\mathbb{R})$, \ie $\rho(\mathbf{s})=L \iff \mathbf{s}\in A_L$.
Recall that a map $\mathfrak{a}\colon \op{GL}_{2d}(\mathbb{R})\to \mathcal{C}$ is a section of $\rho$ if and only if $\mathfrak{a}(L)\in A_L$ for all $L\in\op{GL}_{2d}(\mathbb{R})$, in which case we write $\mathfrak{a}\in \prod_{L\in \op{GL}_{2d}(\mathbb{R})}A_L$. 
An important observation is the fact that we have a dense inclusion $\ell^1_{\nu}(\mathbb{Z}^{2d})\hookrightarrow A_L$ for all $L\in \op{GL}_{2d}(\mathbb{R})$, this motivates us to identify $\ell^1_{\nu}(\mathbb{Z}^{2d})$ as a section space for $\rho$, since for each $\mathbf{a}\in \ell^1_{\nu}(\mathbb{Z}^{2d})$, we can define the constant section, which we shall also denote by $\mathbf{a}$ so that $\mathbf{a} \in \prod_{L\in \op{GL}_{2d}(\mathbb{R})}A_L$ via:
\begin{align}\label{a-constant-sections}
    \mathbf{a}(L) = \mathbf{a} \in \ell^1(\mathbb{Z}^{2d})\hookrightarrow A_L, \qquad \forall L\in \op{GL}_{2d}(\mathbb{R}).
\end{align}

\begin{theorem}\label{fundamental-c-bundle}
    The bundle $\rho\colon  \mathcal{C}\to \op{GL}_{2d},$ whose fibers are $\rho^{-1}(L)=A_L$ for all $L\in \op{GL}_{2d}(\mathbb{R})$, defines a $C^*$-bundle via Theorem \ref{bbundleconstruction} whose local approximating sections are identifiable with $\ell^1_\nu(\mathbb{Z}^{2d})$. 
    As a consequence, a section $\mathfrak{a}\in \prod_{L\in \op{GL}_{2d}(\mathbb{R})}A_L$ is continuous, \ie $\mathfrak{a}\in \Gamma(\rho)$, if and only if for all $\varepsilon>0$ and $L_0\in \op{GL}_{2d}(\mathbb{R})$, there exists $\mathbf{a}_0\in \ell^1_{\nu}(\mathbb{Z}^{2d})$ such that for some open neighborhood $U_0\ni L_0$:
    \begin{align*}
        \|\mathfrak{a}(L)-\mathbf{a}_0\|_{A_L}<\varepsilon, \qquad \forall L\in U_0.
    \end{align*}
\end{theorem}
\begin{proof}
    We only need to find local approximating sections for a bundle to make it a $C^*$-bundle (or a Banach bundle) using Theorem \ref{bbundleconstruction}. We identified $\ell^1_{\nu}(\mathbb{Z}^{2d})$ as the constant sections for $\rho$ defined via \eqref{a-constant-sections}. For $\ell^1_{\nu}(\mathbb{Z}^{2d})$ to be a space of local approximating sections, we need to show two things:
    \begin{enumerate}
        \item For each $\mathbf{a}\in \ell^1_{\nu}(\mathbb{Z}^{2d})$, $L\mapsto \|\mathbf{a}\|_{A_L}= \|\overline{\pi}_L(\mathbf{a})\|$ is continuous, and;
        \item For all $L\in \op{GL}_{2d}(\mathbb{R})$, $\ell^1_{\nu}(\mathbb{Z}^{2d})$ is dense in $A_L$.
    \end{enumerate}
    We already know that Item 2 is true. Item $1$ is covered by Corollary \ref{fund-c-continuity}. The last statement about the continuous sections follows from Proposition \ref{approxbycontinuity}, using $\ell^1_\nu(\mathbb{Z}^{2d})$ as the space of local approximating sections.
\end{proof}

\subsection{Deforming the Heisenberg Module}\label{algbundle}

The $C^*$-bundle $\rho\colon \mathcal{C}\to \op{GL}_{2d}(\mathbb{R})$ of Theorem \ref{fundamental-c-bundle} can be used to construct a corresponding Banach bundle of Heisenberg modules. 
In fact, the argument made to construct $\rho$ almost holds mutatis-mutandis for the Heisenberg modules using an analogous observation that for all $L\in \op{GL}_{2d}(\mathbb{R})$, we have a dense embedding $M^1_{\nu}(\mathbb{R}^d)\hookrightarrow \mathcal{E}_{L}(\mathbb{R}^d)$. 
To make our construction precise, let $\mathcal{E} = \bigsqcup_{L\in \op{GL}_{2d}(\mathbb{R})}\mathcal{E}_L(\mathbb{R}^d)$ and define the bundle $\kappa\colon  \mathcal{E}\to \op{GL}_{2d}(\mathbb{R})$ using the usual projection map. We see from our observation that $M^1_{\nu}(\mathbb{R}^d)$ can be identified as a space of constant sections for $\kappa$ by defining for each $f\in M^1_{\nu}(\mathbb{R}^d)$, the section $f\in \prod_{L\in \op{GL}_{2d}(\mathbb{R})}\mathcal{E}_{L}(\mathbb{R}^d)$ such that
\begin{align}\label{heis-constant-sections}
    f(L):= f\in \mathcal{E}_{L}(\mathbb{R}^d), \qquad \forall L\in \op{GL}_{2d}(\mathbb{R}).
\end{align}
\begin{theorem}\label{main-construction}
    The bundle $\kappa\colon  \mathcal{E}\to \op{GL}_{2d}(\mathbb{R})$ whose fibers are $\kappa^{-1}(L)=\mathcal{E}_{L}(\mathbb{R}^d)$ for all $L\in \op{GL}_{2d}(\mathbb{R}^d)$, defines a Banach bundle via Theorem \ref{bbundleconstruction} whose local approximating sections are identifiable with $M^1_{\nu}(\mathbb{R}^d)$. As a consequence, a section $\Upsilon\in \prod_{L\in \op{GL}_{2d}}\mathcal{E}_{L}(\mathbb{R}^d)$ is continuous, \ie $\Upsilon\in \Gamma(\kappa)$, if and only if for all $\varepsilon>0$ and $L_0\in \op{GL}_{2d}(\mathbb{R})$, there exists $f_0\in M^1_{\nu}(\mathbb{R}^d)$ such that for some open neighborhood $U_0\ni L_0$:
    \begin{align*}
        \|\Upsilon(L)-f_0\|_{\mathcal{E}_{L}(\mathbb{R}^d)}<\varepsilon, \qquad \forall L\in U_0.
    \end{align*}
\end{theorem}
\begin{proof}
    As in the proof of Theorem \ref{fundamental-c-bundle}, we only need to prove that the identification of $M^1_{\nu}(\mathbb{R}^d)$ as sections of $\kappa$ actually defines a space of local approximating section. We need to show that:
    \begin{enumerate}
        \item For each $f\in M^1_{\nu}(\mathbb{R}^d)$, the map $L\mapsto \|f\|_{\mathcal{E}_{L}(\mathbb{R}^d)}$ is continuous, and;
        \item For each $L\in \op{GL}_{2d}(\mathbb{R}^d)$, the space $M^1_{\nu}(\mathbb{R}^{d})$ is dense in $\mathcal{E}_{L}(\mathbb{R}^d)$.
    \end{enumerate}
    We already know that Item $2$ is true. Due to \eqref{equiv-norm-for-heis}, Item $1$ follows if we can show that the map $L\mapsto \|\overline{\pi}_{L}(\lin{L}{f}{f})\|$ is continuous. Using the estimates in Corollary \ref{fund-c-continuity}, we find that:
    \begin{equation}\label{ell^1-cont-of-ff}
    \begin{split}
         \MoveEqLeft
         \bigl|\|\overline{\pi}_L(\lin{L}{f}{f})\|- 
         \|\overline{\pi}_{L_0}(\lin{L_0}{f}{f})\| \bigr|\\
         &\leq \bigl|\|\overline{\pi}_L(\lin{L}{f}{f})\|-\|\overline{\pi}_{L_0}(\lin{L}{f}{f})\|\bigr| + \|\overline{\pi}_{L_0}(\lin{L}{f}{f}-\lin{L_0}{f}{f})\| \\
        &\begin{multlined}
            \leq C\cdot\|{\lin{L}{f}{f}}\|_{\ell^1_{\nu}} \cdot(\|L\|+\|L_0\|)^{1/2}\cdot\|L-L_0\|^{1/2}\\
            + \|{\lin{L}{f}{f}^{L}-\lin{L_0}{f}{f}^{L_0}}\|_{\ell^1}+
        \|\lin{L}{f}{f}-\lin{L_0}{f}{f}\|_{\ell^1}
        \end{multlined}
    \end{split}
    \end{equation}
    It can be deduced from the proof \cite[Lemma 3.5]{FeKa04} that, for any $h\in M^1_{\nu}(\mathbb{R}^d)$, $L\mapsto  {\lin{L}{h}{h}} = \sum_{k\in \mathbb{Z}^{2d}}\mathcal{V}_hh(Lk)$ is continuous with respect to both the $\ell^1$ and $\ell^1_{\nu}$ norm. It follows from a standard triangle-inequality argument that $L\mapsto {\lin{L}{f}{f}^L}$ is also continuous in $\ell^1$ norm. All these imply that, as $L\to L_0$, the right-most expression in \eqref{ell^1-cont-of-ff} approaches $0$. This proves continuity of the map $L\mapsto \|f\|_{\mathcal{E}_{L}(\mathbb{R}^d)}$. On the other hand, the proof for Item 2 is trivial since $M^1_{\nu}(\mathbb{R}^d)$ is dense in $\mathcal{E}_{L}(\mathbb{R}^d).$
    
    The statement about the continuous sections follows from Proposition \ref{approxbycontinuity}, using $M^1_{\nu}(\mathbb{R}^d)$ as the space of local approximating sections.
\end{proof}

So far we have exploited the results involving the weighted spaces $\ell^1_{\nu}(\mathbb{Z}^{2d})$ and $M^1_{\nu}(\mathbb{R}^d)$ for an initial construction of the $C^*$- and Banach bundles $\rho$ and $\kappa$ respectively. The following result shows that the uweighted spaces $\ell^1(\mathbb{Z}^{2d})$ and $M^1(\mathbb{R}^d)$ can also be identified as spaces of continuous (constant) sections for $\rho$ and $\kappa$ respectively. We first give a technical lemma that we can use throughout our computations.
\begin{lemma}\label{heis-feich-estimate}
    Let $f\in M^1(\mathbb{R}^d)$, then for any window $\phi\in M^1(\mathbb{R}^d)$, we have the following estimate for any $L\in \op{GL}_{2d}(\mathbb{R})$:
    \begin{align}\label{heis-cont-embed-in-feich}
        \|f\|_{\mathcal{E}_L(\mathbb{R}^d)}\leq \|\phi\|_2^{-1}\|f\|_{M^1,\phi}.
    \end{align}
\end{lemma}
\begin{proof}
    The result follows from a straightforward computation along with \cite[Proposition 4.10]{Ja18}:
\begin{align*}
        \|f\|_{\mathcal{E}_L(\mathbb{R}^d)}&=\|\overline{\pi}_L\left({\lin{L}{f}{f}}\right)\|^{1/2}\\
        &\leq \|{\lin{L}{f}{f}}\|_{\ell^1}^{1/2} \\
        &=\left(\sum_{k\in \mathbb{Z}^{2d}}|\mathcal{V}_{f}f|(Lk) \right)^{1/2}\\
        &\leq  \left(\int_{\mathbb{R}^{2d}}|\mathcal{V}_{f}f|(x,\omega) \d x \d \omega\right)^{1/2}\\
        &= \|f\|_{M^1, f}^{1/2} \\
        &\leq \left(\|f\|_{M^1,\phi}\cdot \|\phi\|_2^{-2} \|f\|_{M^1,\phi}\right)^{1/2}\\
        &\leq \|\phi\|_2^{-1}\|f\|_{M^1,\phi}
    \end{align*}
    as required.
\end{proof}
\begin{proposition}\label{for-refinement-of-sections}
    $\ell^1(\mathbb{Z}^{2d})\subseteq \Gamma(\rho)$ and $M^1(\mathbb{R}^{d})\subseteq \Gamma(\kappa).$
\end{proposition}
\begin{proof}
    We shall prove $M^1(\mathbb{R}^d)\subseteq \Gamma(\kappa)$, the proof for $\ell^1(\mathbb{Z}^{2d})\subseteq \Gamma(\rho)$ is similar. Suppose $f\in M^1(\mathbb{R}^d)$, and $\varepsilon>0.$ For this problem, we use a particular window $\phi\in M^1(\mathbb{R}^d)$ for our $M^1(\mathbb{R}^d)$-norm. It follows from Proposition \ref{prop:cont-embed-weight} that there exists $f_0\in M^1_{\nu}(\mathbb{R}^d)$ such that $\|f-f_{0}\|_{M^1,\phi}< \varepsilon\cdot\|\phi\|_2.$ Now, for any $L\in \op{GL}_{2d}(\mathbb{R})$, we use Estimate \eqref{heis-cont-embed-in-feich} to obtain $$\|f-f_0\|_{\mathcal{E}_L(\mathbb{R}^d)}\leq \|\phi\|_2^{-1}\|f-f_0\|_{M^1,\phi}<\varepsilon.$$ Therefore, if we consider $f\in M^1(\mathbb{R}^d)$ as a constant section in $\kappa$ with $f(L):=f \in M^1(\mathbb{R}^d)\hookrightarrow \mathcal{E}_L(\mathbb{R}^d)$, we see from Theorem \ref{main-construction} that $f\in \Gamma(\kappa).$
\end{proof}
\begin{corollary}\label{refinement-of-sections}
    $\mathfrak{a}\in \Gamma(\rho)$ if and only if for any $L_0\in \op{GL}_{2d}(\mathbb{R}^d)$ and $\varepsilon>0$, there exists an open set satisfying $U_0\ni L_0$ and $\mathbf{a}_0\in \ell^1(\mathbb{Z}^{2d})$ such that $\|\mathfrak{a}(L)-\mathbf{a}_0\|_{A_L}<\varepsilon$ for all $L\in U_0$.
    
    Similarly, $\Upsilon\in \Gamma(\kappa)$ if and only if for any $L_0\in \op{GL}_{2d}(\mathbb{R}^d)$ and $\varepsilon>0$, there exists an open set $U_0\ni L_0$ and $f_0\in M^1(\mathbb{R}^d)$ such that $\|\Upsilon(L)-f_0\|_{\mathcal{E}_L(\mathbb{R}^d)}<\varepsilon$ for all $L\in U_0.$
\end{corollary}
\begin{proof}
    The result follows from Remark \ref{approx-section-refinement} and Proposition \ref{for-refinement-of-sections}.
\end{proof}
Corollary \ref{refinement-of-sections} above shows what we should expect for a map that continuously varies along Heisenberg modules should look like, it should locally behave as if it is any other element in a fixed Heisenberg module -- we must be able to approximate it by some function in Feichtinger's algebra. To get a better grasp of what these continuous sections look like, we have the following Lemma regarding the space $C(\op{GL}_{2d}(\mathbb{R}^d), M^1(\mathbb{R}))$ of continuous functions from $\op{GL}_{2d}(\mathbb{R})$ to $M^1(\mathbb{R}^d)$. Note that any map $\Upsilon \in C(\op{GL}_{2d}(\mathbb{R}),M^1(\mathbb{R}))$ will define a section with respect to $\kappa$, since its image will always land in $M^1(\mathbb{R}^d)$, and as we have repeatedly emphasized, is always embedded in \emph{all} $\mathcal{E}_L(\mathbb{R}^d)$, for $L\in \op{GL}_{2d}(\mathbb{R})$. Therefore the norm $\|\Upsilon(L)-\Upsilon(L_0)\|_{\mathcal{E}_L(\mathbb{R}^d)}$ makes sense. In fact,  $\|\sum_{i=1}^n\Upsilon(M_i)\|_{\mathcal{E}_L(\mathbb{R}^d)}$ makes sense for all $M_1,\ldots,M_n,L\in \op{GL}_{2d}(\mathbb{R}^d).$
\begin{lemma}\label{cont-map-to-cont-sec}
    We have the following inclusion: $C(\op{GL}_{2d}(\mathbb{R}),M^1(\mathbb{R}^d))\subseteq \Gamma(\kappa).$
\end{lemma}
\begin{proof}
    Fix a window function $\phi\in M^1(\mathbb{R}^d)$. Let $\Upsilon \in C(\op{GL}_{2d}(\mathbb{R}),M^1(\mathbb{R}^d)),$ we must show that $\Upsilon$ satisfies the characterization in Corollary \ref{refinement-of-sections} to show that $\Upsilon \in \Gamma(\kappa)$. Let $\varepsilon>0$ and $L_0\in \op{GL}_{2d}(\mathbb{R})$, choose $f_0:= \Upsilon(L_0)\in M^1(\mathbb{R}^d)$. Since $\Upsilon\in C(\op{GL}_{2d}(\mathbb{R}, M^1(\mathbb{R}^d))$, we can choose an open set $U_0$ containing $L_0$ such that $\|\Upsilon(L)-\Upsilon(L_0)\|_{M^1,\phi}<\|\phi\|_2\cdot\varepsilon$ for all $L\in U_0$. If we take $L\in U_0$, then by Lemma \ref{heis-feich-estimate}:
    \begin{align*}
        \|\Upsilon(L)-f_0\|_{\mathcal{E}_L(\mathbb{R}^d)} &=\|\Upsilon(L)-\Upsilon(L_0)\|_{\mathcal{E}_L(\mathbb{R}^d)}\\
        &\leq \|\phi\|_2^{-1} \|\Upsilon(L)-f_0\|_{M^1,\phi} <\varepsilon.
    \end{align*}
    This shows $\Upsilon\in \Gamma(\kappa).$ 
\end{proof}
\noindent There is a corresponding statement for the space $\Gamma_0(\kappa)$ of continuous sections that vanish at infinity.
\begin{lemma}\label{cont-map-to-cont-sec-2}
    We have the following inclusion: $C_0(\op{GL}_{2d}(\mathbb{R}),M^1(\mathbb{R}^d))\subseteq \Gamma_0(\kappa).$
\end{lemma}
\begin{proof}
    Let $\Upsilon \in C_0(\op{GL}_{2d},M^1(\mathbb{R}^d))$. We already know that $\Upsilon\in \Gamma(\kappa)$, we only need to show that $L\mapsto \|\Upsilon(L)\|_{\mathcal{E}_L(\mathbb{R}^d)}$ vanishes at infinity. However, this follows from the same estimate as in Lemma \ref{heis-and-feich-est}. For a window $\phi\in M^1(\mathbb{R}^d)$:
    \begin{align} \label{heis-and-feich-est}
        \|\Upsilon(L)\|_{\mathcal{E}_{L}(\mathbb{R}^d)} \leq \|\phi\|_2^{-1}\|\Upsilon(L)\|_{M^1,\phi}.
    \end{align}
    So for any fixed $\varepsilon>0$, we can find by hypothesis, a compact subset $M_{\varepsilon}\subseteq \op{GL}_{2d}(\mathbb{R})$ such that $\|\Upsilon(L)\|_{M^1,\phi}<\|\phi\|_2\cdot \varepsilon$ whenever $L\in \op{GL}_{2d}(\mathbb{R})\setminus M_{\varepsilon}$. Therefore
    \begin{align*}
        \|\Upsilon(L)\|_{\mathcal{E}_{L}(\mathbb{R}^d)}\leq \|\phi\|_2^{-1} \|\Upsilon(L)\|_{M^1,\phi} <\varepsilon, \qquad \forall L\in \op{GL}_{2d}(\mathbb{R})\setminus M_{\varepsilon}. 
    \end{align*}
    So $L\mapsto \|\Upsilon(L)\|_{\mathcal{E}_L(\mathbb{R}^d)}$ vanishes at infinity, as required.
\end{proof}
In the case of the Heisenberg module on a fixed lattice, $M^1(\mathbb{R}^d)$ serves as space of approximating functions (with respect to the module-norm), and we see in turn that the space $C_0(\op{GL}_{2d}(\mathbb{R}),M^1(\mathbb{R}^d))$ takes this role for the Banach space $\Gamma_0(\kappa)$. That is, $C_0(\op{GL}_{2d}(\mathbb{R}),M^1(\mathbb{R}^d))$ acts as a \emph{global approximating space} for $\Gamma_0(\kappa)$.
\begin{corollary}\label{globalaprrox2}
$C_0(\op{GL}_{2d}(\mathbb{R}),M^1(\mathbb{R}^d))$ is a dense subspace of $\Gamma_0(\kappa).$
\end{corollary}
\begin{proof}    It is not hard to see that for any $\varphi \in C_0(\op{GL}_{2d}(\mathbb{R}))$ and $\Upsilon \in C_0(\op{GL}_{2d}(\mathbb{R}),M^1(\mathbb{R}^d))$, we have that the pointwise product satisfies $\varphi \cdot \Upsilon \in C_0(\op{GL}_{2d}(\mathbb{R}),M^1(\mathbb{R}^d))$. Furthermore, we shall prove that:
    \begin{align}\label{s_0-by-vanish-sections}
        \{\Upsilon(L)\in \mathcal{E}_L(\mathbb{R}^d): \Upsilon \in C_0(\op{GL}_{2d}(\mathbb{R}),M^1(\mathbb{R}^d))\} = M^1(\mathbb{R}^d).
    \end{align}
     Fix $f\in M^1(\mathbb{R}^d)$ and $L_0\in \op{GL}_{2d}(\mathbb{R})$. Let $M_{0}$ be a compact neighborhood of $L_0$ and $U_{0}$ be an open neighborhood of $L_0$ that contains $M_{0}$, so we have $L_0\in M_{0} \subseteq U_{0}$. By Urysohn's Lemma for locally compact spaces \cite[Lemma 1.41]{Will07}, there exists a continuous function $\varphi\colon  \op{GL}_{2d}(\mathbb{R})\to [0,1]$ such that $\varphi_{|M_{0}}\equiv 1$ and $\varphi \equiv 0$ outside $U_{0}$. Now define
     \begin{align*}
         \Upsilon(L) = \varphi(L) \cdot f, \qquad \forall L\in \op{GL}_{2d}(\mathbb{R}).
     \end{align*}
     We see that $\Upsilon\in C_0(\op{GL}_{2d}(\mathbb{R}),M^1(\mathbb{R}^d))$ while $\Upsilon(L_0)=f$. Since $f$ and $L_0$ are both arbitrary, this proves Equation \eqref{s_0-by-vanish-sections}.
     We see that $C_0(\op{GL}_{2d}(\mathbb{R}),M^1(\mathbb{R}^d))$ satisfies the hypotheses of Proposition \ref{globalapprox}, therefore it must be dense in $\Gamma_0(\kappa).$
    \end{proof}
\noindent There are analogues of Lemma \ref{cont-map-to-cont-sec} and \ref{cont-map-to-cont-sec-2} for $\rho$, and the proof is easier owing to the fact that the norm of $A_L$ is always dominated by the $\ell^1(\mathbb{Z}^{2d})$-norm for all $L\in \op{GL}_{2d}(\mathbb{R}).$
\begin{lemma}
    $C(\op{GL}_{2d}(\mathbb{R}),\ell^1(\mathbb{Z}^{2d}))\subseteq \Gamma(\rho)$ and $C_0(\op{GL}_{2d}(\mathbb{R}),\ell^1(\mathbb{Z}^{2d}))\subseteq \Gamma_0(\rho).$
\end{lemma}
\begin{notation}
    Now would be a good time to take a pause to set some notations in place. It would be helpful to review the module action and inner-product on each Heisenberg module given in Theorem \ref{heisenberg-module}. For any sections (not necessarily continuous) $\Upsilon_1,\Upsilon_2\in \prod_{L\in \op{GL}_{2d}(\mathbb{R})}\mathcal{E}_L(\mathbb{R}^d)$ of $\kappa$, we define a section $\lin{\rho}{\Upsilon_1}{\Upsilon_2}\in \prod_{L\in \op{GL}_{2d}(\mathbb{R})}A_L$ via:
    \begin{align}\label{left-inner-bundle-product}
        \lin{\rho}{\Upsilon_1}{\Upsilon_2}(L):= \lin{L}{\Upsilon_1(L)}{\Upsilon_2(L)}\in A_L, \qquad \forall L\in \op{GL}_{2d}(\mathbb{R}).
    \end{align}
    If we also have the section (again, not necessarily continuous) $\mathfrak{a}\in \prod_{L\in \op{GL}_{2d}(\mathbb{R})}A_L$, then we define $\mathfrak{a}\cdot \Upsilon_1 \in \prod_{L\in \op{GL}_{2d}(\mathbb{R})}\mathcal{E}_L(\mathbb{R}^d)$ via:
    \begin{align}\label{left-bundle-module-act}
        (\mathfrak{a}\cdot \Upsilon_1)(L):= \mathfrak{a}(L)\cdot \Upsilon_1(L)\in \mathcal{E}_L(\mathbb{R}^d), \qquad \forall L\in \op{GL}_{2d}(\mathbb{R}),
    \end{align}
Furthermore, we shall fix a notation for the norm on the Banach space $\Gamma_0(\kappa)$. If $\Upsilon\in \Gamma_0(\kappa),$ we write:
\begin{align}\label{kappa-norm}
    \|\Upsilon\|_{\kappa}:= \sup_{L\in \op{GL}_{2d}(\mathbb{R})}\|\Upsilon(L)\|_{\mathcal{E}_L(\mathbb{R}^d)}.
\end{align}
Similarly, we denote by the norm on the $C^*$-algebra $\Gamma_0(\rho)$ via:
\begin{align}\label{rho-norm}
    \|\mathfrak{a}\|_{\rho}:= \sup_{L\in \op{GL}_{2d}(\mathbb{R})}\|\mathfrak{a}(L)\|_{A_L},
\end{align}
for each $\mathfrak{a}\in \Gamma_0(\rho).$
\end{notation}

\begin{proposition}\label{section-operations}
    If we again identify $\ell^1(\mathbb{Z}^{2d})$ and $M^1(\mathbb{R}^d)$ as the space of local approximating sections of $\rho$ and $\kappa$ respectively, then for all $\mathbf{a}\in \ell^1(\mathbb{Z}^{2d})$, and $f_1,f_2\in M^1(\mathbb{R}^d):$
    \begin{multicols}{2}
    \begin{enumerate}
        \item $\lin{\rho}{f_1}{f_2}\in \Gamma(\rho)$;
        \item $\mathbf{a}\cdot f_1\in \Gamma(\kappa).$
    \end{enumerate}
    \end{multicols}
    \noindent Consequently, if $\mathfrak{a}\in \Gamma(\rho)$, $\Upsilon_1,\Upsilon_2\in \Gamma(\kappa)$, then:
    \begin{multicols}{2}
        \begin{enumerate}[start=3]
        \item $\lin{\rho}{\Upsilon_1}{\Upsilon_2}\in \Gamma(\rho);$
        \item $\mathfrak{a}\cdot \Upsilon_1 \in \Gamma(\kappa)$.
    \end{enumerate}
    \end{multicols}
    \noindent Finally, if $\mathfrak{a}\in \Gamma_0(\rho)$, $\Upsilon_1,\Upsilon_2\in \Gamma_0(\kappa)$, then:
    \begin{multicols}{2}
        \begin{enumerate}[start=5]
        \item $\lin{\rho}{\Upsilon_1}{\Upsilon_2}\in \Gamma_0(\rho);$
        \item $\mathfrak{a}\cdot \Upsilon_1 \in \Gamma_0(\kappa).$
    \end{enumerate}
    \end{multicols}
\end{proposition}
\begin{proof}
    We have already seen the sufficient techniques to prove Items $1$ and $2$. For Item $1$, we need to show that $\lin{\rho}{f}{g}$ satisfies the continuous section characterization from Corollary \ref{refinement-of-sections}. Let $\varepsilon>0$ and $L_0\in \op{GL}_{2d}(\mathbb{R})$, choose $\mathbf{a}_0 = \lin{\rho}{f}{g}(L_0) := \lin{L_0}{f}{g}\in \ell^1(\mathbb{Z}^{2d})$. We know from \cite[Lemma 3.5]{FeKa04} that the map $L\mapsto \{\mathcal{V}_gf(Lk)\}_{k\in \mathbb{Z}^{2d}}$ is continuous from $\op{GL}_{2d}(\mathbb{R})$ to $\ell^1(\mathbb{Z}^{2d})$. So choose the open set $U_0\ni L_0$ of $\op{GL}_{2d}(\mathbb{R})$ such that
    \begin{align*}
        \|\lin{L}{f}{g}-\lin{L_0}{f}{g}\|_{\ell^1}\leq \sum_{k\in \mathbb{Z}^{2d}}|\mathcal{V}_gf(Lk)- \mathcal{V}_gf(L_0k)|<\varepsilon, \qquad \forall L\in U_0.
    \end{align*}
    It follows that for all $L\in U_0$
    \begin{align*}
        \|{\lin{\rho}{f}{g}(L)- \mathbf{a}_0}\|_{A_L} &=\|{\lin{L}{f}{g}- \lin{L_0}{f}{g} }\|_{A_L} \\
        &=\|\overline{\pi}_L(\lin{L}{f}{g}-\lin{L_0}{f}{g})\| \\
        & \leq \|{\lin{L}{f}{g}- \lin{L_0}{f}{g}}\|_{\ell^1}\\
        &<\varepsilon.
    \end{align*}
    Therefore $\lin{\rho}{f}{g}\in \Gamma(\rho)$ as required.
    
    For Item 2, it is sufficient that we show $\mathbf{a}\cdot f \in C(\op{GL}_{2d}(\mathbb{R}),M^1(\mathbb{R}^d))$, following Lemma \ref{cont-map-to-cont-sec}. Let $L,L_0\in \op{GL}_{2d}(\mathbb{R})$, we have:
    \begin{align*}
        \|(\mathbf{a}\cdot f)(L)- (\mathbf{a}\cdot f)(L_0)\|_{M^1}&:= \left\|\sum_{k\in \mathbb{Z}^{2d}}a(k)\left(\pi(Lk)-\pi(L_0k)\right)f \right\|_{M^1} \\
        &\leq \sum_{k\in \mathbb{Z}^{2d}}|a(k)|\|(\pi(Lk)-\pi(L_0k))f\|_{M^1}.
    \end{align*}
    It follows from \cite[Proposition 2.7 (ii)]{FeKa04} that as $L\to L_0$, $\|(\mathbf{a}\cdot f)(L)- (\mathbf{a}\cdot f)(L_0)\|_{M^1}\to 0$. Therefore $\mathbf{a}\cdot f \in C(\op{GL}_{2d}(\mathbb{R}),M^1(\mathbb{R}^d))$ as required.
    
    We next prove Item $3$. To do this, fix a window function $\phi\in M^1(\mathbb{R}^d)$, $\varepsilon>0$, and $L_0\in \op{GL}_{2d}(\mathbb{R})$. By local compactness of $\op{GL}_{2d}(\mathbb{R})$, we can fix a compact neighborhood $K\ni L_0$. 
    Since $\Upsilon_2\colon  \op{GL}_{2d}(\mathbb{R})\to \mathcal{E}$ is a continuous section, we conclude that $L\mapsto \|\Upsilon_2(L)\|_{\mathcal{E}_L(\mathbb{R}^d)}$ continuous, therefore $\sup_{L\in K}\|\Upsilon_2\|_{\mathcal{E}_L(\mathbb{R}^d)}<\infty$.  Now, by the continuous section criteria, there exist $f_1,f_2\in M^1(\mathbb{R}^d)$ such that for some open sets $U_1,U_2\ni L_0$
    \begin{align*}
        \|\Upsilon_1(L)-f_1\|_{\mathcal{E}_L(\mathbb{R}^d)}&< \frac{\varepsilon}{2\cdot \sup_{L\in K}\|\Upsilon_2\|_{\mathcal{E}_L(\mathbb{R}^d)}}, \qquad \forall L\in U_1 \\
        \|\Upsilon_2(L)-f_2\|_{\mathcal{E}_L(\mathbb{R}^d)}&<\frac{\varepsilon}{2\|\phi\|_2^{-1} \|f_1\|_{M^1,\phi}}, \qquad \forall L\in U_2.
    \end{align*}
    Using Proposition \ref{cauchy-sch} Item 1, we obtain the following estimate for $L\in \op{int}(U_1\cap U_2 \cap K)$:
    \begin{align*}
        \MoveEqLeft[3]
        \|{\lin{\rho}{\Upsilon_1}{\Upsilon_2}(L)-\lin{\rho}{f_1}{f_2}(L)}\|_{A_L} = \|\lin{L}{\Upsilon_1(L)}{\Upsilon_2(L)}- \lin{L}{f_1}{f_2}\|_{A_L}\\
        &\leq\|\lin{L}{\Upsilon_1(L)-f_1}{\Upsilon_2(L)}\|_{A_L} + \|\lin{L}{f_1}{\Upsilon_2(L)-f_2}\|_{A_L} \\
        &\leq \|\Upsilon_1(L)-f_1\|_{\mathcal{E}_L(\mathbb{R}^d)}\cdot\sup_{L\in K}\|\Upsilon_2\|_{\mathcal{E}_L(\mathbb{R}^d)} + \|\phi\|_2^{-1}\|f_1\|_{M^1,\phi}\|\Upsilon_2(L)-f_2\|_{\mathcal{E}_{L}(\mathbb{R}^d)}\\
        &<\varepsilon.
    \end{align*}
    Since $\lin{\rho}{f_1}{f_2}\in \Gamma(\rho)$ by Item $1$, it follows from Remark \ref{approx-section-refinement} that $\lin{\rho}{\Upsilon_1}{\Upsilon_2}\in \Gamma(\rho)$.

    The proof for Item $4$ is basically a rehash of Item $3$, where we use the estimate in Proposition \ref{cauchy-sch} Item $2$ instead.

    The proofs for Items $5$ and $6$ are also similar, we shall show one of them. To start, note that because of Proposition \ref{cauchy-sch} Item 1, we have for all $L\in \op{GL}_{2d}(\mathbb{R})$:
    \begin{align*}
        \|\lin{\rho}{\Upsilon_1}{\Upsilon_2}(L) \|_{A_L}:&= \|\lin{L}{\Upsilon_1(L)}{\Upsilon_2(L)} \|_{A_L}\\
        &\leq \|\Upsilon_1(L)\|_{\mathcal{E}_{L}(\mathbb{R}^d)} \|\Upsilon_2(L)\|_{\mathcal{E}_{L}(\mathbb{R}^d)}.
    \end{align*}
    So if both $\Upsilon_1$ and $\Upsilon_2$ vanish at infinity, then so does $\lin{\rho}{\Upsilon_1}{\Upsilon_2}$. Since we already know from Item $3$ that $\lin{\rho}{\Upsilon_1}{\Upsilon_2}\in \Gamma(\rho),$ then it follows from what we have just shown that $\lin{\rho}{\Upsilon_1}{\Upsilon_2}\in \Gamma_0(\rho)$. This shows Item 5.
\end{proof}
We will need a technical Lemma, which is reminiscent of the basic construction of the Heisenberg modules. Before proceeding, we acknowledge the fact that $M^1(\mathbb{R}^{2d})\subseteq C_0(\mathbb{R}^{2d})$ \cite[Lemma 4.19]{Ja18}, and so the evaluation map is well-defined in $M^1(\mathbb{R}^d)$. In particular, we can consider restrictions in $M^1(\mathbb{R}^{2d})$.
\begin{lemma}\label{heis-const-remin}
For a fixed $L\in \op{GL}_{2d}(\mathbb{R})$, the linear space $\op{span}\left\{\lin{L}{f_1}{f_2}: f_1,f_2\in M^1(\mathbb{R}^d)\right\}$ is dense in $\ell^1(\mathbb{Z}^{2d})$ with respect to the $\ell^1$-norm. 
\end{lemma}
\begin{proof}
    For each $L\in \op{GL}_{2d}(\mathbb{R})$, the restriction operator $R_{L}\colon M^1(\mathbb{R}^{2d})\to \ell^1(\mathbb{Z}^{2d})$, which is defined via $(R_{L}F)(k)=F(Lk)$, is a surjective countinuous operator \cite[Theorem 5.7]{Ja18}. Therefore, given any $\mathbf{a}\in \ell^1(\mathbb{Z}^{2d})$, there exists an $F\in M^1(\mathbb{R}^{2d})$ such that $\mathbf{a} = R_{L}F$. It follows from Corollary \ref{for-fullness} that there exists $\{f_i\}_{i\in \mathbb{N}}, \{g_i\}_{i\in \mathbb{N}}\subseteq M^1(\mathbb{R}^{d})$ such that $F = \sum_{i\in \mathbb{N}}\mathcal{V}_{g_i}f_i.$ It follows that $\mathbf{a} =(R_LF) =\sum_{i\in \mathbb{N}}R_L(\mathcal{V}_{g_i}f_i) = \sum_{i\in \mathbb{N}}\lin{L}{f_i}{g_i}$ in $\ell^1$-norm.
\end{proof}

\begin{corollary}\label{bundle-left-hilbert-c-mod}
    The Banach space $\Gamma_0(\kappa)$ defines a full Hilbert left $\Gamma_0(\rho)$-module. 
\end{corollary}

\begin{proof}
    We define $\Gamma_0(\kappa)$ as a Hilbert left $\Gamma_0(\kappa)$-module exactly by the module inner-product and action given by Proposition \ref{section-operations} Items $5$ and $6$. 
    The pointwise nature of the definition immediately shows that the Axioms $1$, $2$, $3$ and $5$ of Definition \ref{imprimitivity-bimodule} are satisfied with the defined operations. 
    Now, if $\Upsilon \in \Gamma_0(\kappa)$, then we know that $\lin{\rho}{\Upsilon}{\Upsilon}(L):= \lin{L}{\Upsilon(L)}{\Upsilon(L)}\geq 0$ in $A_L$ for all $L\in \op{GL}_{2d}(\mathbb{R})$. 
    It follows from Proposition \ref{positive-sections} that $\lin{\rho}{\Upsilon}{\Upsilon}\geq 0$ in the $C^*$-algebra $\Gamma_0(\rho)$. 
    Now let us look at the induced module norm on $\Gamma_0(\kappa)$, for $\Upsilon \in \Gamma_0(\kappa)$, if we denote by $\|\cdot\|_{\op{ind}}$ said induced norm, then using Equations \ref{kappa-norm} and $\ref{rho-norm}$:
    \begin{align*}
        \|\Upsilon\|_{\op{ind}}& :=\|\lin{\rho}{\Upsilon}{\Upsilon}\|_{\rho}^{1/2}  \\
        &=\left( \sup_{L\in \op{GL}_{2d}(\mathbb{R})}\|\lin{\rho}{\Upsilon}{\Upsilon}(L)\|_{A_L}\right)^{1/2}\\
        & = \left( \sup_{L\in \op{GL}_{2d}(\mathbb{R})}\|\lin{L}{\Upsilon(L)}{\Upsilon(L)}\|_{A_L} \right)^{1/2}\\
        &=\sup_{L\in \op{GL}_{2d}(\mathbb{R})}\|\lin{L}{\Upsilon(L)}{\Upsilon(L)}\|_{A_L}^{1/2}\\
        &=\sup_{L\in \op{GL}_{2d}(\mathbb{R})}\|\Upsilon(L)\|_{\mathcal{E}_L(\mathbb{R}^d)}\\
        &=: \|\Upsilon\|_{\kappa}.
    \end{align*}
We see that the induced module norm coincides\footnote{And therefore we shall never make use of the notation $\|\cdot\|_{\op{ind}}$ ever again.} with the norm for $\Gamma_0(\kappa)$ and since $\|\cdot\|_{\kappa}$ is complete, then the induced module norm is also complete. 
This shows that $\Gamma_0(\kappa)$ is a Hilbert left $\Gamma_0(\rho)$-module. 
All that remains is to show fullness. 
Let us first define \[I:= \overline{\op{span}}\left\{\lin{\rho}{\Upsilon_1}{\Upsilon_2}\in \Gamma_0(\rho): \Upsilon_1,\Upsilon_2\in \Gamma_0(\kappa)\right\}.\] 
We know that $I$ is an ideal of $\Gamma_0(\rho)$ by basic Hilbert $C^*$-module theory. 
Our goal is to show that $I=\Gamma_0(\rho)$. Define for each $L\in \op{GL}_{2d}(\mathbb{R})$, the ideal $I_L\subseteq A_L$ via:
\begin{align*}
    I_L &:=\{\mathfrak{a}(L)\in A_L: \mathfrak{a}\in I\} =\overline{\op{span}}\left\{\lin{\rho}{\Upsilon_1}{\Upsilon_2}(L)\in A_L: \Upsilon_1,\Upsilon_2\in \Gamma_0(\kappa)\right\}.
\end{align*}
By Corollary \ref{cor-ideal-section}, we are done if we can show that $I_L = A_L$, because it will follow that $I = \{\mathfrak{a} \in \Gamma_0(\rho): \mathfrak{a}(L)\in I_L\}=\{\mathfrak{a} \in \Gamma_0(\rho): \mathfrak{a}(L)\in A_L\}= \Gamma_0(\rho)$. 
However, it follows from Lemma \ref{heis-const-remin} and the density of $\ell^1(\mathbb{Z}^{2d})$ in $A_L$ that under the $A_L$-norm:
\begin{align*}
    A_L=\overline{\op{span}}\left\{\lin{L}{f_1}{f_2}:f_1,f_2\in M^1(\mathbb{R}^d)\right\}.
\end{align*}
We can proceed similar to the proof of Corollary \ref{globalaprrox2} to show that \[\left\{\lin{L}{f_1}{f_2}:f_1,f_2\in M^1(\mathbb{R}^d)\right\} \subseteq \left\{\lin{L}{\Upsilon_1(L)}{\Upsilon_2(L)}: \Upsilon_1,\Upsilon_2\in \Gamma_0(\kappa)\right\}.\] Therefore
\begin{align*}
    A_L = \overline{\op{span}}\{\lin{L}{\Upsilon_1(L)}{\Upsilon_2(L)}: \Upsilon_1,\Upsilon_2\in \Gamma_0(\kappa)\}:= I_L,
\end{align*}
as required. Hence we have shown that $\Gamma_0(\kappa)$ is a full Hilbert left $\Gamma_0(\rho)$-module.
\end{proof}

Now that we have shown that there is a full Hilbert \textit{left} $\Gamma_0(\rho)$-module structure on the continuous sections $\Gamma_0(\kappa)$ that vanish at infinity, it is natural to ask if $\Gamma_0(\kappa)$ is actually an imprimitivity bimodule implementing a Morita equivalence between $\Gamma_0(\kappa)$ and another $C^*$-algebra of sections. Knowing that at each $L\in \op{GL}_{2d}(\mathbb{R})$, the corresponding fiber of $\kappa$, the Heisenberg module $\kappa^{-1}(L)=\mathcal{E}_L(\mathbb{R}^d)$ is a Morita equivalence between $A_L$ and $B_L$ motivates us to construct $\mathcal{C}^{\circ} = \bigsqcup_{L\in \op{GL}_{2d}} B_L$ with the corresponding bundle projection map $\rho^{\circ}\colon  \mathcal{C}^{\circ}\to \op{GL}_{2d}(\mathbb{R})$. We now have a bundle which has the fiber $(\rho^{\circ})^{-1}(L)=B_L$ for each $L\in \op{GL}_{2d}(\mathbb{R})$. Our goal now is to prove that $\rho^{\circ}$ is also a $C^*$-bundle. There is no need to make this harder than it should be, we already have the tools that we can use from the construction of $\rho$ available to us, and it now all boils down to restating our problem in a familiar way. Let $L\in \op{GL}_{2d}(\mathbb{R})$, with the corresponding matrix $L^{\circ}=J(L^{-1})^TJ^T$ that generates the adjoint lattice, then using the cocycle-adjoint relations \eqref{adj-cocycle}, for each $\mathbf{b}\in \ell^1(\mathbb{Z}^{2d})$
\begin{align*}
    \overline{\pi}^*_{L^{\circ}}(\mathbf{b}) &= \frac{1}{|\det{L}|}\sum_{k\in \mathbb{Z}^{2d}}b(k) \pi^*(L^{\circ}k) \\
    &= \frac{1}{|\det{L}|}\sum_{k\in \mathbb{Z}^{2d}}b(k)c(L^{\circ}k,L^{\circ}k) \pi(-L^{\circ}k)\\
    &= \frac{1}{|\det{L}|}\overline{\pi}_{-L^{\circ}}\left(\prescript{L^{\circ}}{}{\mathbf{b}}\right),
\end{align*}
where $\prescript{L^{\circ}}{}{\mathbf{b}}\in \ell^1(\mathbb{Z}^{2d})$ is the sequence whose terms are: $\left(\prescript{L^{\circ}}{}{b}\right)(k)=c(L^{\circ}k,L^{\circ}k)b(k)$ for each $k\in \mathbb{Z}^{2d}.$ The computation above now opens us for an analogue of Corollary \ref{fund-c-continuity} for $\rho^{\circ}$. In fact, all of the results that we have derived for $\rho$, also have analogues for $\rho^{\circ}$, and we shall not repeat the associated proofs for these results. Most of them boil down to the fact that the maps $L\mapsto L^{\circ}$ and $L\to \op{det}L$ are continuous, with the occassional use of the triangle inequality. If one is truly interested in repeating the proofs for $\rho^{\circ},$ then it is also instructive to know that for each $L\in \op{GL}_{2d}(\mathbb{R})$, $\mathbf{b}\in \ell^1(\mathbb{Z}^{2d})$ and $f_1,f_2\in M^1(\mathbb{R}^d)$ that:
\begin{align*}
    \|\mathbf{b}\|_{B_L}&\leq \frac{1}{|\det{L}|} \sum_{k\in \mathbb{Z}^{2d}}|b(k)|; \\
\rin{L^{\circ}}{f_1}{f_2}(k)&= \frac{1}{|\det{L}|}\sum_{k\in \mathbb{Z}^{2d}}\overline{\mathcal{V}_{f_2}f_1(L^{\circ}k)}.
\end{align*} 
Using the same observation that we did for the bundle $\rho$, we have a dense inclusion: $\ell^1(\mathbb{Z}^{2d})\hookrightarrow B_L$ for all $L\in \op{GL}_{2d}(\mathbb{R})$, and so we can identify $\ell^1(\mathbb{Z}^{2d})$ as a space of sections for $\rho^{\circ}$ so that each $\mathbf{b}\in \ell^1(\mathbb{Z}^{2d})$ can be seen as $\mathbf{b}\in \prod_{L\in \op{GL}_{2d}}B_L$ via
\begin{align*}
    \mathbf{b}(L):= \mathbf{b}\in B_L.
\end{align*}
As a corollary, we have, using exactly the same kind of argument as in Theorem \ref{main-construction} and Proposition \ref{refinement-of-sections} that $\rho^{\circ}$ can be promoted to a $C^*$-bundle with the following continuous sections:
\begin{corollary}
    The bundle $\rho^{\circ}\colon  \mathcal{C}^{\circ}\to \op{GL}_{2d},$ whose fibers are $(\rho^{\circ})^{-1}(L)=B_L$ for all $L\in \op{GL}_{2d}(\mathbb{R})$, defines a $C^*$-bundle such that every section $\mathfrak{b}\in \prod_{L\in \op{GL}_{2d}(\mathbb{R})}B_L$ is continuous, \ie $\mathfrak{b}\in \Gamma(\rho^{\circ})$ if and only if for all $\varepsilon>0$ and $L_0\in \op{GL}_{2d}(\mathbb{R})$, there exists $\mathbf{b}_0\in \ell^1(\mathbb{Z}^{2d})$ such that for some open neighborhood $U_0\ni L_0$:
    \begin{align*}
        \|\mathfrak{b}(L)-\mathbf{b}_0\|_{B_L}<\varepsilon, \qquad \forall L\in U_0.
    \end{align*}
\end{corollary}
\begin{proposition}\label{cont-map-cont-sec-circ}
     $C(\op{GL}_{2d}(\mathbb{R}), \ell^1(\mathbb{Z}^{2d}))\subseteq \Gamma(\rho^{\circ})$ and $C_0(\op{GL}_{2d}(\mathbb{R}),\ell^1(\mathbb{Z}^{2d}))\subseteq \Gamma_0(\rho^{\circ}).$
\end{proposition}
\begin{notation}
    We now introduce the following notation for the $\rho^{\circ}$. For any sections (not necessarily continuous) $\Upsilon_1,\Upsilon_2\in \prod_{L\in \op{GL}_{2d}(\mathbb{R})}\mathcal{E}_L(\mathbb{R}^d)$ of $\kappa$, we define a section $\rin{\rho^{\circ}}{\Upsilon_1}{\Upsilon_2}\in \prod_{L\in \op{GL}_{2d}(\mathbb{R})}B_L$:
    \begin{align}\label{right-inner-bundle-product}
        \rin{\rho^{\circ}}{\Upsilon_1}{\Upsilon_2}(L):= \rin{L^{\circ}}{\Upsilon_1(L)}{\Upsilon_2(L)}\in B_L, \qquad \forall L\in \op{GL}_{2d}(\mathbb{R}).
    \end{align}
    If we also have the section $\mathfrak{b}\in \prod_{L\in \op{GL}_{2d}(\mathbb{R})}B_L$, then we define $\Upsilon_1 \cdot \mathfrak{b} \in \prod_{L\in \op{GL}_{2d}(\mathbb{R})}\mathcal{E}_L(\mathbb{R}^d)$ via:
    \begin{align}\label{right-bundle-module-act}
        (\Upsilon_1 \cdot \mathfrak{b})(L):= \Upsilon_1(L)\cdot \mathfrak{b}(L) \in \mathcal{E}_L(\mathbb{R}^d), \qquad \forall L\in \op{GL}_{2d}(\mathbb{R}),
    \end{align}
We denote by the norm on the $C^*$-algebra $\Gamma_0(\rho^{\circ})$ via:
\begin{align}\label{rho-adj-norm}
    \|\mathfrak{b}\|_{\rho^{\circ}}:= \sup_{L\in \op{GL}_{2d}(\mathbb{R})}\|\mathfrak{b}(L)\|_{B_L},
\end{align}
for each $\mathfrak{b}\in \Gamma_0(\rho^{\circ}).$
\end{notation}
\begin{proposition}\label{right-section-operations}
    If we again identify $\ell^1(\mathbb{Z}^{2d})$ and $M^1(\mathbb{R}^d)$ as the space of local approximating sections of $\rho^{\circ}$ and $\kappa$ respectively, then for all $\mathbf{\mathbf{b}}\in \ell^1(\mathbb{Z}^{2d})$, and $f_1,f_2\in M^1(\mathbb{R}^d):$
    \begin{multicols}{2}
    \begin{enumerate}
        \item $\rin{\rho^{\circ}}{f_1}{f_2}\in \Gamma(\rho^{\circ})$;
        \item $f_1 \cdot \mathbf{b} \in \Gamma(\kappa).$
    \end{enumerate}
    \end{multicols}
    \noindent Consequently, if $\mathfrak{b}\in \Gamma(\rho)$, $\Upsilon_1,\Upsilon_2\in \Gamma(\kappa)$, then:
    \begin{multicols}{2}
        \begin{enumerate}
        \setcounter{enumi}{2}
        \item $\rin{\rho^{\circ}}{\Upsilon_1}{\Upsilon_2}\in \Gamma(\rho^{\circ});$
        \item $\Upsilon_1 \cdot \mathfrak{b} \in \Gamma(\kappa).$
    \end{enumerate}
    \end{multicols}
    \noindent Finally, if $\mathfrak{b}\in \Gamma_0(\rho^{\circ})$, $\Upsilon_1,\Upsilon_2\in \Gamma_0(\kappa)$, then:
    \begin{multicols}{2}
        \begin{enumerate}
        \setcounter{enumi}{4}
        \item $\rin{\rho^{\circ}}{\Upsilon_1}{\Upsilon_2}\in \Gamma_0(\rho^{\circ});$
        \item $\Upsilon_1\cdot \mathfrak{b} \in \Gamma_0(\kappa).$
    \end{enumerate}
    \end{multicols}
\end{proposition}
\begin{corollary}\label{corollary: imp-bim-from-bundles}
    $\Gamma_0(\kappa)$ is a full Hilbert right $\Gamma_0(\rho^{\circ})$-module with respect to the $\Gamma_0(\rho^{\circ})$-valued inner-product and right module action defined by \eqref{right-inner-bundle-product} and \eqref{right-bundle-module-act}. As a consequence $\Gamma_0(\kappa)$ is a $\Gamma_0(\rho)$-$\Gamma_0(\rho^{\circ})$ imprimitivity bimodule. 
\end{corollary}
\begin{proof}
    That $\Gamma_0(\kappa)$ is a full Hilbert right $\Gamma_0(\rho^{\circ})$-module follows from the same argument as in Corollary \ref{bundle-left-hilbert-c-mod} using Proposition \ref{right-section-operations}. It is almost immediate that $\Gamma_0(\kappa)$ is a $\Gamma_0(\rho)$-$\Gamma_0(\rho^{\circ})$ imprimitivity bimodule, but we shall explicitly write what Items 2 and 3 of Definition \ref{imprimitivity-bimodule} mean in our setting. Let $\Upsilon_1,\Upsilon_2,\Upsilon_3 \in \Gamma_0(\kappa)$, $\mathfrak{a}\in \Gamma_0(\rho)$, and $\mathfrak{b}\in \Gamma_0(\rho^{\circ})$, we have from the pointwise definition of the section operations and the fact that $\mathcal{E}_L(\mathbb{R}^d)$ is an $A_L-B_L$-imprimitivity bimodule for each $L\in \op{GL}_{2d}(\mathbb{R})$, that:
    \begin{enumerate}
        \item $\rin{\rho^{\circ}}{\mathfrak{a} \Upsilon_1}{\Upsilon_2} = \rin{\rho^{\circ}}{\Upsilon_1}{\mathfrak{a}^*\Upsilon_2}$,

        \item $\lin{\rho}{\Upsilon_1 \mathfrak{b}}{\Upsilon_2}=\lin{\rho}{\Upsilon_1}{\Upsilon_2\ (\mathfrak{b})^*  }$,
    \item $\lin{\rho}{\Upsilon_1}{\Upsilon_2}\ \Upsilon_3 = \Upsilon_1 \ \rin{\rho^{\circ}}{\Upsilon_2}{\Upsilon_3}.$ \qedhere
    \end{enumerate}
    \noindent
\end{proof}

\subsection{Applications to Gabor Analysis}\label{subsec: gabanalysis app}
Before we proceed, we note that there is a fundamental obstruction regarding the continuity of the map $\op{GL}_{2d}(\mathbb{R})\to \mathcal{L}(L^2(\mathbb{R}))$ given by the integrated representations $L\mapsto \overline{\pi}_L(\mathbf{b})$ (resp. $L\mapsto \overline{\pi}^*_{L^{\circ}}(\mathbf{b})$) for a fixed $\mathbf{b}\in \ell^1(\mathbb{Z}^{2d})$. As discussed in \cite{FeKo98}, these maps are in general not globally continuous, and it remains so even if we consider our continuous sections. However, the local continuity of such maps on points of interest remain true in our setting.
\begin{lemma}\label{local-cont-of-varying-rep}
    Let $\mathfrak{b}\in \Gamma(\rho^{\circ})$. If $\lambda\in \mathbb{C}\setminus \{0\}$ and  $L_0\in \op{GL}_{2d}(\mathbb{R})$ are such that $\mathfrak{b}(L_0)=\lambda \delta_0\in \ell^1(\mathbb{Z}^{2d})$, then $L\mapsto \overline{\pi}_{L^{\circ}}^*(\mathfrak{b}(L))$ is a continuous map at the point $L_0.$
\end{lemma}
\begin{proof}
    We shall show the proof for continuity of $L\mapsto \overline{\pi}_{L^{\circ}}(\mathfrak{b}(L))$ at the point $L_0$. Fix $\varepsilon>0$, choose an open set $U_1\ni L_0$ and $\mathbf{b}\in \ell^1(\mathbb{Z}^{2d})$ such that
    \begin{align*}
        \|\mathfrak{b}(L)-\mathbf{b}\|_{B_L}<\frac{\varepsilon}{4}, \qquad \forall L\in U_1.
    \end{align*}
    With these fixed, we take note of the fact that $\mathbf{b}-\mathfrak{b}(L_0)\in \ell^1(\mathbb{Z}^{2d})$ and so the map 
    \[L\mapsto \|\overline{\pi}^*_{L^{\circ}}(\mathbf{b}-\mathfrak{b}(L_0))\|\]
    is continuous by Proposition \ref{cont-map-cont-sec-circ}. In particular, there exists an open set $U_2\ni L_0$ such that:
    \begin{align*}
        |\|\overline{\pi}^*_{L^{\circ}}(\mathbf{b}-\mathfrak{b}(L_0))\| - \|\overline{\pi}^*_{L^{\circ}_0}(\mathbf{b}-\mathfrak{b}(L_0))\||<\frac{\varepsilon}{4}, \qquad \forall L\in U_2.
    \end{align*}
    Now, take note of the fact that because $\mathfrak{b}(L_0) = \lambda \delta_0$, we have 
    \[\overline{\pi}^*_{L^{\circ}}(\mathfrak{b}(L_0))-\overline{\pi}^*_{L^{\circ}_0}(\mathfrak{b}(L_0))=\lambda\left(\frac{1}{|\det{L}|}-\frac{1}{|\det{L_0}|}\right).\]
    Hence we finally take the open set $U_3\ni L_0$ such that
    \begin{align*}
        \left|\frac{1}{|\det{L}|}-\frac{1}{|\det{L_0}|}\right| < \frac{\varepsilon}{4|\lambda|}, \qquad \forall L\in U_3.
    \end{align*}
    Since
    \begin{align*}
        \|\overline{\pi}^*_{L^{\circ}}(\mathfrak{b}(L))-\overline{\pi}^*_{L_0^{\circ}}(\mathfrak{b}(L_0))\|\leq & \|\overline{\pi}^*_{L^{\circ}}(\mathfrak{b}(L))- \overline{\pi}^*_{L^{\circ}}(\mathbf{b})\| + \|\overline{\pi}_{L^{\circ}}^*(\mathbf{b}-\mathfrak{b}(L_0))\|\\
        &+ \|\overline{\pi}^*_{L^{\circ}}(\mathfrak{b}(L_0))-\overline{\pi}^*_{L^{\circ}_0}(\mathfrak{b}(L_0))\|,
    \end{align*}
    we can estimate the terms individually, so that if $L\in U_1\cap U_2\cap U_3$, then
    \begin{align*}
        \|\overline{\pi}^*_{L^{\circ}}(\mathfrak{b}(L))- \overline{\pi}^*_{L^{\circ}}(\mathbf{b})\| = \|\mathfrak{b}(L)-\mathbf{b}\|_{B_L}<\frac{\varepsilon}{4};
    \end{align*}
    and
    \begin{align*}
        \|\overline{\pi}_{L^{\circ}}^*(\mathbf{b}-\mathfrak{b}(L_0))\| &\leq |\|\overline{\pi}^*_{L^{\circ}}(\mathbf{b}-\mathfrak{b}(L_0))\| - \|\overline{\pi}^*_{L^{\circ}_0}(\mathbf{b}-\mathfrak{b}(L_0))\|| + \|\overline{\pi}_{L_0^{\circ}}^*(\mathbf{b}-\mathfrak{b}(L_0))\|\\
        &\leq |\|\overline{\pi}^*_{L^{\circ}}(\mathbf{b}-\mathfrak{b}(L_0))\| - \|\overline{\pi}^*_{L^{\circ}_0}(\mathbf{b}-\mathfrak{b}(L_0))\||+ \|\mathfrak{b}(L_0)-\mathbf{b}\|_{B_{L_0}}\\
        &<\frac{\varepsilon}{4} + \frac{\varepsilon}{4};
    \end{align*}
    and
    \begin{align*}
        \|\overline{\pi}^*_{L^{\circ}}(\mathfrak{b}(L_0))-\overline{\pi}^*_{L^{\circ}_0}(\mathfrak{b}(L_0))\| = |\lambda|\left|\frac{1}{|\det{L}|}-\frac{1}{|\det{L_0}|}\right|< \frac{\varepsilon}{4}.
    \end{align*}
    Therefore, $L\in U_1\cap U_2\cap U_3$ implies:
    \begin{align*}
        \|\overline{\pi}^*_{L^{\circ}}(\mathfrak{b}(L))-\overline{\pi}^*_{L_0^{\circ}}(\mathfrak{b}(L_0))\|<\varepsilon
    \end{align*}
    as required.
\end{proof}
\begin{remark}
    An analogous result holds for $L\mapsto \overline{\pi}_L(\mathfrak{a}(L))$ with $\mathfrak{a}\in \Gamma(\rho).$
\end{remark}
It is now important that we recall that for each $L\in \op{GL}_{2d}(\mathbb{R}^d)$, $\mathcal{E}_L(\mathbb{R}^d)\hookrightarrow L^2(\mathbb{R}^d)$ via Theorem \ref{heisenberg-module}. This means that for any section $\Upsilon_1,\Upsilon_2 \in \prod_{L\in \op{GL}_{2d}(\mathbb{R}^d)}\mathcal{E}_L(\mathbb{R}^d)$, we have for each $L\in \op{GL}_{2d}(\mathbb{R}^d),$ that
\begin{align*}
    S_{\Upsilon_1(L),\Upsilon_2(L),L}\in \mathcal{L}(L^2(\mathbb{R}^d)).
\end{align*}
So every section in $\kappa$ defines a $\op{GL}_{2d}(\mathbb{R}^d)$-indexed family of mixed Gabor frame operators on $L^2(\mathbb{R}^d)$. Now, in the case where $\Upsilon_1,\Upsilon_2$ are continuous sections that form dual frames on some $L_0$, we have a crucial result.
\begin{lemma}\label{cont-frame-sec-op}
    Suppose $\Upsilon_1, \Upsilon_2\in \Gamma(\kappa)$, and $L_0\in \op{GL}_{2d}(\mathbb{R}^d)$ such that $\Upsilon_1(L_0)$ and $\Upsilon_2(L_0)$ form dual Gabor frames with respect to the lattice $L_0\mathbb{Z}^{2d}$, then the map $\op{GL}_{2d}(\mathbb{R})\to \mathcal{L}(L^2(\mathbb{R}^d))$ via $L\mapsto S_{\Upsilon_1(L),\Upsilon_2(L),L}$ is continuous at the point $L_0.$
\end{lemma}
\begin{proof}
    It follows from the Janssen's representation for Heisenberg modules, Corollary \ref{gen-for-heis} Item 2, that
    for all $L\in \op{GL}_{2d}(\mathbb{R}^d)$:
    \begin{align}\label{gen-jan}
        S_{\Upsilon_1(L),\Upsilon_2(L),L} = \overline{\pi}_{L^{\circ}}^*(\rin{L^{\circ}}{\Upsilon_2(L)}{\Upsilon_1(L)}) = \overline{\pi}_{L^{\circ}}^*(\rin{\rho^{\circ}}{\Upsilon_2}{\Upsilon_1}(L)).
    \end{align}
    Since $\Upsilon_1(L_0)$ and $\Upsilon_2(L_0)$ are dual Gabor frames with respect to $L_0\mathbb{Z}^{2d}$, then it follows from the Wexler-Rax relations for Heisenberg modules, Corollary \ref{gen-for-heis} Item 1, that:
    \begin{align*}
        \rin{\rho^{\circ}}{\Upsilon_2}{\Upsilon_1}(L_0)=|\det{L_0}|\cdot \delta_0.
    \end{align*}
    Since $\rin{\rho^{\circ}}{\Upsilon_2}{\Upsilon_1}\in \Gamma(\rho^{\circ})$ by Proposition \ref{right-section-operations} Item 3, then it follows from Lemma \ref{local-cont-of-varying-rep} that $L\mapsto \overline{\pi}_{L^{\circ}}^*(\rin{\rho^{\circ}}{\Upsilon_2}{\Upsilon_1}(L))$ is continuous at $L_0$. The result now follows from \eqref{gen-jan}
\end{proof}
\begin{theorem}\label{stability-heis}
    Let $L_0\in \op{GL}_{2d}(\mathbb{R}^d)$, and suppose $f_0 \in \mathcal{E}_{L_0}(\mathbb{R}^{d})$ such that $\mathcal{G}(f_0,L_0)$ generates a Gabor frame, then for every continuous section $\Upsilon\in \Gamma(\kappa)$ that passes through $f_0$, there exists an open set $U_0\ni L_0$ such that $\mathcal{G}(\Upsilon(L),L)$ generates a Gabor frame for all $L\in U_0.$
\end{theorem}
\begin{proof}
    It follows from \eqref{mixed-frame-restrict} that the canonical dual frame $h_0=S^{-1}_{f_0,L_0}f_0$ is in $\mathcal{E}_{L_0}(\mathbb{R}^d)$. Let $\Upsilon,\Omega\in \Gamma(\kappa)$ such that $\Upsilon(L_0)=f_0$ and $\Omega(L_0) = h_0$. We see that $\Upsilon$ and $\Omega$ form dual Gabor frames at the point $L_0$, and therefore by Lemma \ref{cont-frame-sec-op}, $L\mapsto S_{\Upsilon(L),\Omega(L),L}$ is continuous at $L_0$. This means there exists an open set $U_0\ni L_0$ such that
    \begin{align*}
        \|S_{\Upsilon(L),\Omega(L),L}- \op{Id}_{L(\mathbb{R}^d)}\|<1, \qquad \forall L\in U_0.
    \end{align*}
    Therefore $S_{\Upsilon(L),\Omega(L),L}$ is invertible for all $L\in U_0$. Since $(S_{\Upsilon(L),\Omega(L),L})_{|\mathcal{E}_{L}(\mathbb{R}^d)}\colon  \mathcal{E}_L(\mathbb{R}^d)\to \mathcal{E}_L(\mathbb{R}^d)$, we define the section:
    \begin{align}\label{dual-section}
        \Phi(L) := S_{\Upsilon(L),\Omega(L),L}^{-1}\Omega(L) \in \mathcal{E}_L(\mathbb{R}^d)\hookrightarrow L^2(\mathbb{R}^d), \qquad \forall L\in U_0.
    \end{align}
    Therefore if $L\in U_0,$ then for all $f\in L^2(\mathbb{R}^d):$
    \begin{align*}
        S_{\Upsilon(L),\Omega(L),L} S_{\Upsilon(L),\Phi(L),L}f &=S_{\Upsilon(L),\Omega(L),L} \sum_{k\in \mathbb{Z}^{2d}}\lin{}{f}{\pi(Lk)\Upsilon(L)}\pi(Lk)\Phi(L)f \\
        &=\sum_{k\in \mathbb{Z}^{2d}}\lin{}{f}{\pi(Lk)\Upsilon(L)}\pi(Lk)S_{\Upsilon(L),\Omega(L),L}\Phi(L)f \\
        &=\sum_{k\in \mathbb{Z}^{2d}}\lin{}{f}{\pi(Lk)\Upsilon(L)}\pi(Lk)\Omega(L)f\\
        &= S_{\Upsilon(L),\Omega(L),L}f,
    \end{align*}
    hence $S_{\Upsilon(L),\Phi(L),L}=\op{Id}_{L^2(\mathbb{R}^d)}$ for all $L\in U_0,$ which implies $\Phi(L)$ is a dual Gabor atom for $\Upsilon(L)$ in $L\in U_0$ and that $\mathcal{G}(\Upsilon(L),L)$ generates a Gabor frame in $L^2(\mathbb{R}^d)$ for all $L\in U_0.$
\end{proof}
The following proposition shows that the `local section' (it is only defined inside the open set $U_0$) implementing a local dual Gabor atom for $\Upsilon$ above, denoted by $\Phi$ in Equation \eqref{dual-section}, is continuous at the point $L_0$. Its precise meaning is given below.
\begin{proposition}\label{proposition: local-section-generator}
     The section $\Phi\in\prod_{L\in U_0} \mathcal{E}_L(\mathbb R^d)$ in \eqref{dual-section} satisfies: for each $\varepsilon>0$ there exists an open subset $W_0\ni L_0$ of $U_0$, and a $g_0\in M^1(\mathbb{R}^d)$ such that
    \begin{align*}
        \|\Phi(L)-g_0\|_{\mathcal{E}_L(\mathbb{R}^d)}<\varepsilon, \qquad \forall L\in W_0.
    \end{align*}
\end{proposition}
\begin{proof}
    Fix $\varepsilon>0$ and consider a compact neighborhood $K_0\ni L_0$ inside $U_0$. Since we have $\|\op{Id}_{L^2(\mathbb{R}^d)}-S_{\Upsilon(L),\Omega(L),L}\|<1$  for all $L\in U_0$, it follows from the Neumann and geometric series that $$\|S^{-1}_{\Upsilon(L),\Omega(L),L}\|\leq \frac{1}{1-\|\op{Id}_{L^2(\mathbb{R}^d)}-S_{\Upsilon(L),\Omega(L),L}\|}$$ for all $L\in U_0.$ It follows that $\|S^{-1}_{\Upsilon(L),\Omega(L),L}\|\leq C$ for some $C>0$ in $K_0$. 
    Because $\Omega\in \Gamma(\kappa)$, there exists $W_1\ni L_0$ an open subset contained in $K_0$ such that for $L\in W_1$:
    \begin{align*}
        \|\Omega(L)-g_0\|_{\mathcal{E}_L(\mathbb{R}^d)}<\frac{\varepsilon}{2C}.
    \end{align*}
    We know that $L\mapsto S_{\Upsilon(L),\Omega(L),L}$ is continuous at the point $L_0$, and since operator inversion is continuous in $\mathcal{L}(L^2(\mathbb{R}^d))$, we find that $U_0\ni L\mapsto S^{-1}_{\Upsilon(L),\Omega(L),L}\in \mathcal{L}(L^2(\mathbb{R}^d))$ is also continuous. Choose $W_2\ni L_0$ contained in $K_0$ such that for $L\in W_2:$
    \begin{align*}
        \|S^{-1}_{\Upsilon(L),\Omega(L),L}-S^{-1}_{\Upsilon(L_0),\Omega(L_0),L_0}\| &=\|S^{-1}_{\Upsilon(L),\Omega(L),L}-\op{Id}_{L^2(\mathbb{R}^d)}\| < \frac{\varepsilon}{2 \|g_0\|_2}.
    \end{align*}
    Finally we can estimate, for $L\in W_0:= W_1\cap W_2\subseteq K_0$:
    \begin{align*}
        \|\Phi(L)-g_0\|_{\mathcal{E}_L(\mathbb{R}^d)}&= \|S^{-1}_{\Upsilon(L),\Omega(L),L}\Omega(L)-g_0\| \\
        &\leq \|S^{-1}_{\Upsilon(L),\Omega(L),L}(\Omega(L)-g_0)\| + \|S^{-1}_{\Upsilon(L),\Omega(L),L}g_0-g_0 \| \\
        &\leq C\|\Omega(L)-g_0\|_{\mathcal{E}_{L}(\mathbb{R}^d)} +\|S^{-1}_{\Upsilon(L),\Omega(L),L}-\op{Id}_{L^2(\mathbb{R}^d)}\|\|g_0\|_2 \\
        &< \varepsilon.\qedhere
    \end{align*}
\end{proof}
As a Corollary of Gabor-stability of Heisenberg modules, we derive that these modules exhibit a version of the Balian-Low Theorem when considered as a function space.
\begin{corollary}[Balian-Low for Heisenberg Modules]\label{corollary: balian-low-1}
    Let $L_0\in \op{GL}_{2d}(\mathbb{R})$ such that $|\det{L_0}|=1$. If $f_0\in \mathcal{E}_{L_0}(\mathbb{R}^d)$, then $\mathcal{G}(f_0, L_0)$ cannot be a Gabor frame for $L^2(\mathbb{R}^d).$
\end{corollary}
\begin{proof}
    Suppose for contradiction that with $f_0 \in \mathcal{E}_{L_0}(\mathbb{R}^d)$ and $|\det{L_0}|=1$, that $\mathcal{G}(f_0,L_0)$ generates a Gabor frame for $L^2(\mathbb{R}^d)$. Let $\Upsilon \in \Gamma(\kappa)$ be a continuous section that passes through $f_0$ at $L_0$. Then it follows from Theorem \ref{stability-heis} that there exists an $\varepsilon>0$ such that  $\mathcal{G}(\Upsilon(L),L)$ generates a Gabor frame for all $L\in \op{GL}_{2d}(\mathbb{R})$ where $\|L-L_0\|<\varepsilon$. Now take $\delta>0$ such that $0<\delta \|L_0\|<\varepsilon,$ then $L = (1+\delta)L_0$ satisfies $\|L-L_0\| = \delta \|L_0\|<\varepsilon,$ so $\mathcal{G}(\Upsilon(L),L)$ generates a Gabor frame for $L^2(\mathbb{R}^d)$. However, $|\det{L}|=|\det{((1+\delta)L_0)}|=(1+\delta)^{2d}>1$, which contradicts the basic density result given by Theorem \ref{basic-density}.
  \end{proof}
\subsection{Modulus of continuity of local approximating sections}\label{subsec: modulus}
In order to establish a Banach bundle structure, it was enough that the the norm varies continuously along the local approximating sections, which we obtained for the sections associated with $f\in M^1_\nu$ with $\nu(k) = 1 + |k|$. However, there are actually much better estimates for sections associated with $f\in M^1_{\nu_s}$ with $\nu_s(k)= (1+|k|)^s$ for $s\geq2$, as we will show in this section. 

Recall that a function $f\colon X\to Y$ between normed spaces $X,Y$ is called $\alpha$-Hölder continuous ($\alpha$-Hölder for short) for $0< \alpha \leq 1$ if there is a constant $C>0$ such that $\|f(x_1)-f(x_2)\|\leq C \|x_1-x_2\|^\alpha$ for all $x_1,x_2\in X$. More generally, $f$ is called locally $\alpha$-Hölder if for every $x\in X$ there exists a neighborhood $x\in M\subset X$ such that $f|_M$ is $\alpha$-Hölder. In the special case of $\alpha=1$, we say that $f$ is (locally) Lipschitz instead.

Let us summarize some well-known facts on Hölder continuity.
\begin{observation}\leavevmode
    \begin{itemize}
        \item if $0<\alpha<\beta\leq 1$, then every locally $\beta$-Hölder map is locally $\alpha$-Hölder;
        \item if $f\colon X\to Y$ is locally $\alpha$-Hölder and $g\colon Y\to Z$ is locally $\beta$-Hölder, then $g\circ f\colon X\to Z$ is locally $(\alpha\cdot\beta)$-Hölder;
        \item if $f\colon S\to B$ with $S\subset \mathbb R^n$ and $B$ a Banach space is differentiable, then $f$ is locally Lipschitz.
    \end{itemize}
\end{observation}

\begin{lemma}
    For each fixed $n\in \mathbb{Z}^{2d}$, the map $L\mapsto P_{c_L}(n)$ fulfills the estimate
    \[|P_{c_L}(n) - P_{c_{L_0}}(n)|\leq (2d+1)\pi \|K\| \nu_2(n) (\|L\|+\|L_0\|)\cdot \|L-L_0\|.\]
    In particular, it is locally Lipschitz.
\end{lemma}
\begin{proof}
    Looking at Equation \ref{pee-el}, we find that
    \begin{align*}
        P_{c_L}(n)
        =\exp\left(2\pi i\, \left(\sum_{j<k} n_j n_k (Le_j)^TK(Le_k) + \sum_k \frac{n_k(n_k-1)}{2}(Le_k)^TK (Le_k)\right)\right)
    \end{align*}
    Combining the estimates $|e^{ia}-e^{ib}|\leq |a-b|$, $|\frac{m(m-1)}{2}|\leq |m|^2$ for all $m\in\mathbb{Z}$, and $\|(Le_j)^TK(Le_k)-(L_0e_j)^TK(L_0e_k)\|\leq \|K\| (\|L\|+\|L_0\|)\|L-L_0\|$, we obtain
    \begin{align*}
        |P_{c_L}(n)-P_{c_{L_0}}(n)|
        &\leq \pi \|K\| (\|L\|+\|L_0\|)\|L-L_0\| \left(\sum_{j\neq k} |n_j| |n_k| +  2\sum_k n_k^2\right).
    \end{align*}
    We estimate the last term further
    \begin{align*}
        \sum_{j\neq k} |n_j| |n_k| +  2\sum_k n_k^2 
        &= 
        \begin{pmatrix}
            |n_1|&\ldots & |n_{2d}|
        \end{pmatrix}
        \begin{pmatrix}
            2&1&\ldots &1\\
            1&\ddots&\ddots&\vdots\\
            \vdots &\ddots&\ddots&1\\
            1&\ldots&1&2
        \end{pmatrix}
        \begin{pmatrix}
            |n_1|\\\vdots \\ |n_{2d}|
        \end{pmatrix}\\
        &\leq \|n\|^2 \left\|\begin{pmatrix}
            2&1&\ldots &1\\
            1&\ddots&\ddots&\vdots\\
            \vdots &\ddots&\ddots&1\\
            1&\ldots&1&2
        \end{pmatrix}\right\| \leq \nu_2(n) \cdot (2d+1);
    \end{align*}
    indeed, the norm of the symmetric matrix coincides with its largest Eigenvalue $2d+1$ and $\|n\|^2=\sum n_k^2\leq \nu_2(n)$.
\end{proof}

\begin{lemma}\label{a-l-Lipschitz}
    For each $\mathbf{a}\in \ell^1_{\nu_2}(\mathbb{Z}^{2d})$ and $L\in \op{GL}_{2d}(\mathbb{R})$: The map $\op{GL}_{2d}(\mathbb{R})\ni L\mapsto \mathbf{a}^L\in \ell^1(\mathbb{Z}^{2d})$ fulfills the inequality
    \[\|\mathbf{a}^L-\mathbf{a}^{L_0}\|_{\ell^1}\leq (2d+1)\pi \|K\| (\|L\|+\|L_0\|) \|\mathbf a\|_{\ell^1_{\nu_2}}\cdot \|L-L_0\|.\]
    In particular, it is locally Lipschitz.
\end{lemma}
\begin{proof}
    \begin{align*}
        \|\mathbf{a}^L-\mathbf{a}^{L_0}\|_{\ell^1} &= \sum_{k\in \mathbb{Z}^{2d}}|a(k)| |P_{c_L}(k)-P_{c_{L_0}}(k)|\\
        &\leq \sum_{k\in \mathbb{Z}^{2d}} |a(k)| (2d+1)\pi \|K\| \nu_2(k) (\|L\|+\|L_0\|)\cdot \|L-L_0\|
        \\
        &= (2d+1)\pi \|K\| (\|L\|+\|L_0\|) \|\mathbf a\|_{\ell^1_{\nu_2}}\cdot \|L-L_0\|\qedhere
    \end{align*}
\end{proof}

\begin{lemma}\label{lem:Lipschitz}
  For $f,g\in M^1_{\nu_s}(\mathbb R^d)$, $s\geq 2t\geq2$, the function $L\mapsto {\lin{L}{f}{g}}$ is continuous as a map to $\ell^1_{\nu_s}$ and $r$-times differentiable as a map to $\ell^1_{\nu_{s-2r}}$ for $1\leq r \leq \frac{s}{2}$. 
\end{lemma}
\begin{proof}
    In this proof we will write $M^1_s$ and $\ell^1_s$ instead of $M^1_{\nu_s}$ and $\ell^1_{\nu_s}$, respectively, to increase readability.
    
  The continuity statement is contained in \cite[Lemma 3.5]{FeKa04}.
  Consider the partial derivatives
  \begin{align*}
  \frac{\partial}{\partial x_\ell}\pi(x,\omega)g(t)&=-e^{2\pi i\, \omega t}(\partial_\ell g)(t-x)=-\pi(x,\omega)(\partial_\ell  g)(t),\\
    \frac{\partial}{\partial \omega_\ell }\pi(x,\omega)g(t)&=2\pi i\, t_\ell  e^{2\pi i\, \omega t}g(t-x)=\pi(x,\omega) 2\pi i(t_\ell +x_\ell )g(t).
  \end{align*}
  Put
  \[z_m:=
    \begin{cases}
      x_m&1\leq m\leq d,\\
      \omega_{m-d} & d+1\leq m\leq d;
    \end{cases}\qquad
    D_m(g):=
    \begin{cases}
      -\partial_m g& 1\leq m\leq d,\\
       2\pi i(t_{m-d}+x_{m-d})g & d+1\leq m\leq d.
    \end{cases}
  \]
  so that $z=(x,\omega)$ and $\frac{\partial}{\partial z_m}\pi(z)g=\pi(z)D_m(g)$. Note that $\partial_ng$ and $t_ng$, and therefore all $D_m(g)$, belong to $M^1_{{s-1}}$ (the 1-dimensional case is treated in \cite{DJLL21}, the multidimensional analogue can be proved along the same lines). By \cite[Lemma 2.2 (ii)]{FeKa04}, $z\mapsto \frac{\partial}{\partial z_m}\pi(z)g=\pi(z)D_m(g)$ is a continuous map to $M^1_{{s-1}}$. By the chain rule,
   \begin{multline*}
    \frac{\partial}{\partial L_{m,n}}{\lin{L}{f}{g}}(k)
    =\left.\sum_\ell \frac{\partial \langle f,\pi(z)g\rangle}{\partial z_\ell}\right|_{z=Lk}\frac{\partial (Lk)_\ell}{\partial L_{m,n}}
    =\sum_\ell {\lin{L}{ f}{D_\ell(g)}}(k) \delta_{m,\ell}\cdot k_n
    \\
    ={\lin{L}{ f}{D_m(g)}}(k)\cdot k_n.      
  \end{multline*}
  We know that $f$ and the $D_m(g)$ lie in the modulation space $M^{1}_{{s-1}}$, therefore
  \[L\mapsto {\lin{L}{f}{D_m(g)}}\]
  is a continuous map into $\ell^1_{{s-1}}$ and, consequently,
  \[L\mapsto \frac{\partial}{\partial L_{m,n}}{\lin{L}{f}{g}} = \left({\lin{L}{f}{D_m(g)}}(k)\cdot k_n\right)_{k\in\mathbb{Z}^{2d}}\]
  is a continuous map into $\ell^1_{{s-2}}$, i.e.\ $L\mapsto {\lin{L}{f}{g}}$ is differentiable as a map to $\ell^1_{s-2}$.

  The same arguments can be iterated to show that the function $L\mapsto {\lin{L}{f}{g}}$ is $r$-times differentiable as a map to $\ell^1_{s-2r}$ as long as $s\geq 2r$.  
\end{proof}

\begin{corollary}\label{cor:continuity_of_coefficient-map}
  Let $f,g\in M_{\nu_2}^1(\mathbb{R}^d)$. Then $L\mapsto \|{\lin{L}{f}{g}}\|_{\ell^1_\nu}$ is continuous and $L\mapsto \lin{L}{f}{g}$ is locally Lipschitz as a map $\operatorname{GL}_{2n}(\mathbb R)\to\ell^1(\mathbb{Z}^{2d})$. 
\end{corollary}

\begin{proof}
    Since $M^1_{\nu_2}(\mathbb{R}^d)\subset M^1_\nu(\mathbb{R}^d)$, $L\mapsto \lin{L}{f}{g}$ is continuous as a map to $\ell^1_{\nu}(\mathbb{Z}^{2d})$ by \cite[Lemma 3.5]{FeKa04} and thus the first statement follows. 
  From the Lemma \ref{lem:Lipschitz} applied to $s=2$, we conclude that $L\mapsto \lin{L}{f}{g}$ is differentiable as a map to $\ell^1(\mathbb{Z}^{2d})$ and, therefore, locally Lipschitz. 
\end{proof}

\begin{theorem}\label{thm:Heisenberg-norm-Hoelder}
  Let $f,g\in M^1_{\nu_2}(\mathbb{R}^d)$. Then the map
  $L\mapsto \|\overline{\pi}_L(\lin{L}{f}{g})\|$ is locally $\frac{1}{2}$-H{\"o}lder continuous.
  In particular, the map $L\mapsto \|f\|_{\mathcal{E}_L(\mathbb{R}^d)}=\|\overline{\pi}_L(\lin{L}{f}{f})\|^{1/2}$ is locally $\frac{1}{2}$-H{\"o}lder continuous.
\end{theorem}

\begin{proof}
    Put $\mathbf a:=\lin{L}{f}{g}$ and $\mathbf a_0:=\lin{L_0}{f}{g}$. Then, using triangle inequality, opposite triangle inequality, and Corollary \ref{fund-c-continuity}
  \begin{align*}
         \MoveEqLeft
         \bigl|\|\overline{\pi}_L(\mathbf a)\|- 
         \|\overline{\pi}_{L_0}(\mathbf a_0)\| \bigr|\\
         &\leq \bigl|\|\overline{\pi}_L(\mathbf a)\|-\|\overline{\pi}_{L_0}(\mathbf a)\|\bigr| + \bigl| \|\overline{\pi}_{L_0}(\mathbf a)\|-\|\overline{\pi}_{L_0}(\mathbf a_0)\|\bigr| \\
        &\leq \bigl|\|\overline{\pi}_L(\mathbf a)\|-\|\overline{\pi}_{L_0}(\mathbf a)\|\bigr| + \|\overline{\pi}_{L_0}(\mathbf a-\mathbf a_0)\| \\
        &\leq C\|\mathbf{a}\|_{\ell^1_{\nu}} (\|L\|+\|L_0\|)^{1/2}\|L-L_0\|^{1/2} + \|\mathbf{a}^L-\mathbf{a}^{L_0}\|_{\ell^1} + \|\mathbf{a}-\mathbf{a}_0\|_{\ell^1}.
    \end{align*}
    The first statement thus follows from Lemma \ref{a-l-Lipschitz} and Corollary \ref{cor:continuity_of_coefficient-map} and the fact that every locally Lipschitz map is locally $\frac{1}{2}$-Hölder.

    The second statement if trivial for $f=0$. Let $f\neq0$. For all $L\in\operatorname{GL}_{2d}(\mathbb R)$, faithfulness of the integrated representation and positive definiteness of the inner product imply that $\overline{\pi}_L(\lin{L}{f}{f})\neq 0$, hence $\|f\|_{\mathcal E_{L}(\mathbb R^d)}>0$. Since the square-root function is differentiable on $\mathbb R_{>0}$, we get the claimed local Hölder continuity.
\end{proof}
\begin{remark}
    By \cite[Proposition 3.5]{GrLe16}, $\ell^1_{\nu_2}(\mathbb{Z}^{2d})$ is an inverse-closed subalgebra of $A_L$, in particular $\sigma(\lin{L}{f}{f})=\sigma(\overline{\pi}_L(\lin{L}{f}{f}))=\sigma({\pi_{\Theta_L}}(\lin{L}{f}{f}^L))$ for all $f\in M^1_{\nu_2}(\mathbb{R}^d)$. With the generalization of Theorem \ref{malte} to spectra mentioned in Remark \ref{remark:spectrum} and arguments parallel to the ones we just performed, one can also show that the map $L\mapsto \sigma(\lin{L}{f}{f})$ is locally  $\frac{1}{2}$-H{\"o}lder continuous for every $f\in M^1_{\nu_2}(\mathbb{R}^d)$. Indeed, for $\mathbf a=\lin{L}{f}{f}$ and $\mathbf a_0=\lin{L_0}{f}{f}$,
    \begin{align*}
        \sigma(\mathbf a)\subset \sigma(\mathbf a_0) + [-1,1]\cdot (2\sqrt{\pi}\|\mathbf{a}\|_{\ell^1_\nu} \|\Theta_L-\Theta_{L_0}\|^{\frac{1}{2}} + \|\mathbf a^L -\mathbf a_0^{L_0}\|_{\ell^1})
    \end{align*}
    by Remark \ref{remark:spectrum}. Therefore, for the Hausdorff distance $d_H$ we find that
    \begin{align*}
        d_H(\sigma(\mathbf a),\sigma(\mathbf{a_0}))
        &\leq 2\sqrt{\pi}\|\mathbf{a}\|_{\ell^1_\nu} \|\Theta_L-\Theta_{L_0}\|^{\frac{1}{2}} + \|\mathbf a^L -\mathbf a_0^{L_0}\|_{\ell^1}
        \\
        &\leq 2\sqrt{\pi}\|\mathbf{a}\|_{\ell^1_\nu} (\|L\|+\|L_0\|)^{1/2}\|L-L_0\|^{1/2} + \|\mathbf{a}^L-\mathbf{a}^{L_0}\|_{\ell^1} + \|\mathbf{a}-\mathbf{a}_0\|_{\ell^1}
    \end{align*}
    and the claimed $\frac{1}{2}$-Hölder continuity follows exactly as in the proof of Theorem \ref{thm:Heisenberg-norm-Hoelder}. 
    
    Of course, everything also works for the right Hilbert module structure and, thus, we get that the map $L\mapsto \sigma(\rin{L}{f}{f})$ is $\frac{1}{2}$-Hölder. Finally, recall that according to the Janssen representation (Corollary \ref{gen-for-heis} Item 2), we have $S_{f,L}=\overline{\pi}^*_{L^\circ}(\rin{L^\circ}{f}{f})$. Since $L\mapsto L^\circ=-J(L^{-1})^TJ$ is locally Lipschitz, we may conclude that for $f\in M^1_{\nu_2}(\mathbb{R}^d)$ the map $L\mapsto \sigma(S_{f,L})=\sigma(\rin{L^\circ}{f}{f})$ is locally $\frac{1}{2}$-Hölder. 
\end{remark}

\subsection{Multi-window Gabor systems and generalized Fell's Condition}\label{subsec: continuous-trace}
We consider generalizations of the results in Section \ref{subsec: gabanalysis app} to multi-window Gabor systems and show that they motivate the existence of a generalized continuous-trace $C^*$-algebra in this setting.

To this end, we call $\mathcal{G}(g_1,...,g_n,L):=\bigcup_{i=1}^n\left\{\pi(Lk)g_i:k\in \mathbb{Z}^{2d}\right\}\subseteq L^2(\mathbb{R}^d)$ a \textbf{multi-window Gabor system} for $g_1,...,g_n\in L^2(\mathbb{R}^d)$ and $L\in \op{GL}_{2d}(\mathbb{R})$. The system $\mathcal{G}(g_1,...,g_n,L)$ is said to be a \textbf{multi-window Gabor frame} if it generates a frame, or equivalently, the associated frame operator $S_{\mathbf{g},L}=\sum_{i=1}^nS_{g_i,L}$ is invertible. Relating to the Heisenberg modules, the crucial fact from \cite{AuEn20} is that for all $L\in \op{GL}_{2d}(\mathbb{R})$, $\mathcal{E}_{L}(\mathbb{R}^d)$ is a finitely generated $A_L$-module whose generators are exactly multi-window Gabor frames for $L^2(\mathbb{R}^d).$  

Now, we assume $\mathbf{g}=\{g_1,...,g_n\}\subseteq \mathcal{E}_{L_0}(\mathbb{R}^d)$, and consider a family $\mathbf{\Upsilon}= \{\Upsilon_1,...,\Upsilon_n\}\subseteq \prod_{L\in \op{GL}_{2d}(\mathbb{R})}\mathcal{E}_L(\mathbb{R}^d)$ of sections, we say that $\mathbf{\Upsilon}$ \textbf{passes through} $\mathbf{g}$ \textbf{at} $L_0$, if $(\Upsilon_1(L_0),...,\Upsilon_n(L_0))= (g_1,...,g_n).$  A now straightforward generalization of Theorem \ref{stability-heis} holds for multi-window Gabor frames, whose proof remains largely unchanged.
\begin{theorem}\label{theorem: local-section-multi-window}
    If $\mathbf{g}=\{g_1,...,g_n\}\subseteq \mathcal{E}_{L_0}(\mathbb{R}^d)$, such that $\mathbf{g}$ generates a multi-window Gabor frame for $L^2(\mathbb{R}^d)$, then for any family of continuous sections $\mathbf{\Upsilon}= \{\Upsilon_1,...,\Upsilon_2\}\subseteq \Gamma(\kappa)$ that goes through $\mathbf{g}$ at $L_0$, there exists an open set $U_0\ni L_0$ such that $\mathcal{G}(\Upsilon_1(L),...,\Upsilon_n(L),L)$ generates a multi-window Gabor frame for all $L\in U_0.$
\end{theorem}
\noindent A generalized Balian-Low theorem for multi-window Gabor frames generated by Heisenberg modules also holds. The proof is similar to Corollary \ref{corollary: balian-low-1}, and would only need a refinement of Theorem \ref{basic-density} supplied by \cite[Proposition 4.33]{AuJaMaLu20}: If $\mathcal{G}(g_1,...,g_n,L_0)$ is a multi-window Gabor frame and $\{g_1,...,g_n\}\subseteq \mathcal{E}_{L_0}(\mathbb{R}^d)$, then $|\det L_0|\leq n.$
\begin{corollary}\label{corollary: section-multi-window}
    Let $L_0\in \op{GL}_{2d}(\mathbb{R}^d)$ such that $|\det L_0| =n$. If $g_1,...,g_n \in \mathcal{E}_{L_0}(\mathbb{R}^d)$, then $\mathcal{G}(g_1,...,g_n,L_0)$ cannot be a multi-window Gabor frame for $L^2(\mathbb{R}^d).$
\end{corollary}

We now turn to a result concerning the existence of local projections for the bundle of noncommutative tori $\rho: \mathcal{C}\to \op{GL}_{2d}(\mathbb{R})$. Recall that if $Z$ is an $A$-$B$-imprimitivity bimodule, then $A$ is isomorphic to compact adjointable operators in $Z$ with respect to $B$. That is, $A \cong \mathcal{K}_B(Z)$, and it is completely determined by the map $\vartheta_Z(z\otimes w)= \lin{A}{z}{w}$, where $(z\otimes w) (\xi) = z \rin{B}{w}{\xi}$ for $z,w,\xi\in Z$.
\begin{definition}\label{def: n-rank-proj}
    Let $A$ be a $C^*$-algebra and $Z$ an $A$-$B$-imprimitivity bimodule for some $C^*$-algebra $B$. We say that the element $a\in A$ is a \textbf{rank-}$n$ \textbf{projection in} $A$ \textbf{with respect to} $Z$ if $a$ is the image of a rank-$n$ projection under the isomorphism $\vartheta_Z:\mathcal{K}_B(Z)\to A$.
\end{definition}
\begin{example}
    Let $H$ be a Hilbert space and $A=\mathcal{K}(H)$ its algebra of compact operators. It is not hard to see that $H$ is an $A$-$\mathbb{C}$ imprimitivity bimodule with the obvious structure on $A$. It follows that an element $a\in A$ is a rank-$n$ projection with respect to $H$ if and only if $a$ is a rank-$n$ projection in $\mathcal{K}(H)$ in the usual sense. 
\end{example}
     \emph{Continuous-trace} $C^*$\emph{-algebras} decompose as $C^*$-bundles over their Hausdorff spectrum $T=\widehat{A}$ where the fibers are compact operators on a Hilbert space and satisfying the so-called Fell's condition (see \cite{Fe61}, or \cite[Definition 10.5.7]{dix82}). As illustrated by the previous example, Definition~\ref{def: n-rank-proj} aligns with the classical formulation of Fell's condition. Its utility, however, lies in extending the notion of rank to elements of \(C^*\)-algebras that are not of the form \(\mathcal{K}(H)\) by employing Hilbert $C^*$-modules. This suggests the possibility of a generalized notion of continuous-trace \(C^*\)-algebras.
\begin{remark}
    The following construction will be relevant for our next result, we invite the reader to consult \cite{AuJaMaLu20} for more details. Suppose $Z$ is an $A$-$B$ imprimitivity bimodule, and $n\in \mathbb{Z}_{>0}$, then the matrix space $M_{n,1}(Z)$ is an $M_n(A)$-$B$ imprimitivity bimodule. The actions of $M_n(A)$ and $B$ are given by the obvious matrix actions. On the other hand, the inner-products are given by the following. For $[w],[z]\in M_{n,1}(Z)$
    \begin{equation}\label{lifting-to-matrix}
        \begin{split}
            \lin{M_n(A)}{[z]}{[w]}_{i,j} =\lin{A}{z_{i}}{w_{j}} \\
            \rin{B}{[z]}{[w]} = \sum_{i=1}^n \rin{B}{z_i}{w_i}
        \end{split}
    \end{equation}
\end{remark}
\begin{lemma}\label{lemma: lift-bundle-matrix}
    Fix $n\in \mathbb{Z}_{>0}.$ If $\rho:\mathcal{A}\to X$ is a $C^*$-bundle, then the induced bundle given by the projection $\rho_n: M_n(\mathcal{A})\to X$ where $M_n(\mathcal{A}):= \bigsqcup_{x\in X}M_n(\mathcal{A}_x)$, is also a $C^*$-bundle.
\end{lemma}
\begin{proof}
   Let $\Gamma(p)$ be the set of continuous sections for $p:X\to \mathcal{A}$. We define the complex linear space of sections $M_n(\Gamma(p))=\{[f]:X \to M_{n}(\mathcal{A})\ |\ f_{ij} \in \Gamma(p), \ \text{for all $i,j$} \}$. We claim that $M_n(\Gamma(p))$ defines a local approximating section for $p_n: M_n(\mathcal{A})\to X.$ This easily follows from the fact that $\|[a]\|_{M_n(A_x)}\leq \sum_{ij}\|a_{ij}\|_{A_x}$ for all $[a]\in M_n(A_x)$ given any $x\in X.$
\end{proof}
\begin{theorem}\label{theorem: local-projections}
    For each $L_0\in \op{GL}_{2d}(\mathbb{R})$, there exists an open neighborhood $U_0\ni L_0$, an integer $n_0\in \mathbb{Z}_{>0}$, and a continuous local section $p\in \prod_{L\in U_0}M_{n_0}(A_{L_0})$ such that each $p(L)$ is a rank-1 projection in $M_{n_0}(A_{L})$ with respect to $M_{n_0,1}(\mathcal{E}_{L}(\mathbb{R}^d))$ for all $L\in U_0$.  
\end{theorem}
\begin{proof}
Depending on $L_0\in \op{GL}_{2d}(\mathbb{R}^d)$, there exists $\mathbf{g}=\{g_1,...,g_{n_0}\}\subseteq  \mathcal{E}_{L_0}(\mathbb{R}^d)$ that generates it as a finitely generated $A_{L_0}$-module. Let $\mathbf{\Upsilon}=\{\Upsilon_1,...,\Upsilon_{n_0}\}\subseteq \Gamma(\kappa)$ that passes through $\mathbf{g}$ at $L_0$. We know that $\mathbf{g}$ generates a multi-window Gabor frame, therefore by Theorem \ref{theorem: local-section-multi-window}, there exists an open set $U_0\ni L_0$ such that $S_{\mathbf{\Upsilon}(L),L}$ is positive and invertible in $\mathcal{L}(L^2(\mathbb{R}^d))$ for all $L\in U_0.$ For $i=1,...,n_0$, define the local sections $\Phi_i\in \prod_{L\in U_0}\mathcal{E}_L(\mathbb{R}^d)$ via:
\begin{align*}
\Phi_i(L)=S_{\mathbf{\Upsilon}(L),L}^{-1/2}\Upsilon_i(L) \in \mathcal{E}_L(\mathbb{R}^d).
\end{align*}
For all $f\in \mathcal{E}_L(\mathbb{R}^d)$, we find that
\begin{align*}
    f \sum_{i=1}^{n_0}\rin{L^{\circ}}{\Phi_i(L)}{\Phi_i(L)} &= \sum_{i=1}^{n_0}S^{-1/2}_{\mathbf{\Upsilon}(L),L} \sum_{k\in \mathbb{Z}^{2d}} \rin{2}{S^{-1/2}_{\mathbf{\Upsilon}(L),L}f}{\pi(Lk)\Upsilon_i(L)}\pi(Lk) \Upsilon_i(L) \\
    &= \sum_{i=1}^{n_0}S^{-1/2}_{\mathbf{\Upsilon}(L),L} S_{\Upsilon_i(L),L} S_{\mathbf{\Upsilon(L)},L}^{-1/2}f = S_{\mathbf{\Upsilon}(L),L}^{-1/2}S_{\mathbf{\Upsilon}(L),L} S^{-1/2}_{\mathbf{\Upsilon}(L),L}f=f.
\end{align*}
Since $B_L$ must act faithfully on the right of $\mathcal{E}_L(\mathbb{R}^d)$, we obtain that 
\begin{align}\label{shows-idempotence}
\sum_{i=1}^{n_0}\rin{L^{\circ}}{\Phi_i(L)}{\Phi_i(L)}=1_{B_L}.    
\end{align}

Lifting to the matrix module, define the local section $[\Phi]\in \prod_{L\in U_0}M_{n_0,1}(\mathcal{E}_L(\mathbb{R}^d))$ whose $i$-th entry evaluated at $L$ is exactly $\Phi_i(L)\in \mathcal{E}_L(\mathbb{R}^d).$ Next, define $p\in \prod_{L\in U_0}M_{n_0}(A_L)$ to be:
\begin{align*}
    p_{ij}(L) = \lin{L}{\Phi_i(L)}{\Phi_j(L)}.
\end{align*}
It follows that $p(L)$ is the image of the rank-one adjointable $[\Phi](L)\otimes [\Phi](L)$ under the isomorphism $\mathcal{K}_{B}(M_{n_0,1}(\mathcal{E}_{L}(\mathbb{R}^d))\to M_{n_0}(A_L)$. It is immediate by its form that $[\Phi](L)\otimes [\Phi](L)$ is self-adjoint in $\mathcal{K}_B(\mathcal{E}_L(\mathbb{R}^d))$. We still need to show that it is idempotent: let $[f]\in M_{n_0,1}(\mathcal{E}_L(\mathbb{R}^d))$, and for simplicity, let $([\Phi](L)\otimes [\Phi](L)) [f] = [f^{\Phi}](L)$. The following computation
\begin{align*}
    \left([\Phi](L)\otimes [\Phi](L)\right)^2 ([f]) &= \sum_{i,k=1}^{n_0} \begin{pmatrix}
        \Phi_1(L) \\
        \Phi_2(L) \\
        \vdots \\
        \Phi_{n_0}(L)
    \end{pmatrix} \rin{L^{\circ}}{\Phi_i(L)}{\Phi_i(L)} \rin{L^{\circ}}{\Phi_k(L)}{f_k} \\
    &=\sum_{k=1}^{n_0} \begin{pmatrix}
        \Phi_1(L) \\
        \Phi_2(L) \\
        \vdots \\
        \Phi_{n_0}(L)
    \end{pmatrix} \rin{L^{\circ}}{\sum_{i=1}^{n_0}\Phi_k(L)\rin{L^{\circ}}{\Phi_i(L)}{\Phi_i(L)}}{f_k}\\
    &= \sum_{i=1}^{n_0} \begin{pmatrix}
        \Phi_1(L) \\
        \Phi_2(L) \\
        \vdots \\
        \Phi_{n_0}(L)
    \end{pmatrix} \rin{L^{\circ}}{\Phi_k(L)}{f_k} = \left([\Phi](L)\otimes [\Phi](L)\right)([f])
\end{align*}
shows idempotence. We have shown that $p(L)$ is a rank-one projection on $M_{n_0}(A_L)$ with respect to $M_{n_0,1}(\mathcal{E}_L(\mathbb{R}^d))$ for all $L\in U_0.$

We are done once we show that $p\in \prod_{L\in \op{GL}_{2d}(\mathbb{R})}M_{n_0}(A_L)$ is a local continuous section. However, this is straightforward, as we have already carried out similar estimates. We only give a sketch. First, the proof of Lemma~\ref{lemma: lift-bundle-matrix} reduces the problem to verifying the local continuity of each section \( p_{ij} \in \prod_{L \in U_0} A_L \). Then, by Proposition~\ref{section-operations}, it suffices to verify the local continuity of each \( \Phi_i \in \prod_{L \in U_0} \mathcal{E}_L(\mathbb{R}^d) \). Finally, an argument analogous to that in the proof of Proposition~\ref{proposition: local-section-generator} establishes the local continuity of each \( \Phi_i \), completing the proof.
\end{proof}
\noindent This suggests that the $C^*$-algebra $\Gamma_0(\rho)$ of Section \ref{algbundle} gives a more generalized notion of Fell's condition that is applicable to $C^*$-bundles\footnote{Or even to the slightly general upper semicontinuous $C^*$-bundles.}. Our proposed generalized Fell's condition is nontrivial due to generalized Balian-Low theorem \ref{corollary: section-multi-window} since it is impossible to fix a uniform integer $n_0$ in Theorem \ref{theorem: local-projections} that would hold globally for all $L\in \op{GL}_{2d}(\mathbb{R}^d)$. 

Naturally one is led to ask whether $C^*$-bundles that satisfy a version of Theorem \ref{theorem: local-projections} admit a topological invariant reminiscent of the so-called Dixmier-Douady class for continuous-trace $C^*$-algebras. We can deduce the following from Corollary \ref{corollary: imp-bim-from-bundles}.
\begin{corollary}\label{corollary: possibly-trivial}
    Consider the bundles $\rho,\kappa,$ and $\rho^{\circ}$ of Section \ref{algbundle}, then $\Gamma_0(\rho)\cong \mathcal{K}_{\Gamma_0(\rho^{\circ})}(\Gamma_0(\kappa))$.
\end{corollary}
One may then conjecture that an appropriate Dixmier-Douady class for $\Gamma_0(\rho)$ would be trivial if one proceeds with the idea of replacing the role Hilbert spaces with Hilbert $C^*$-modules in the theory of continuous-trace $C^*$-algebras. This is due to the fact that if $A$ is a continuous-trace $C^*$-algebra with Hausdorff spectrum $X$, then $A$ has trivial Dixmier-Douady class if and only if there exists a Banach bundle of Hilbert spaces $p: \mathcal{H}\to X$ such that $A\cong \mathcal{K}_{C_0(X)}\Gamma_0(p)$ \cite{DiDo63} (note that $\Gamma_0(p)$ forms a full Hilbert right $C_0(X)$-module). In Corollary~\ref{corollary: possibly-trivial}, the \(C^*\)-algebra \(\Gamma_0(\rho^{\circ})\) plays the role of a \emph{locally unital} \(C^*\)-algebra. Namely, one that is unital when restricted to compact neighborhoods—modeled on \(C_0(X)\).

The preceding discussion provides evidence for the existence of a meaningful generalization of continuous-trace $C^*$-algebras (or Fell algebras \cite{AnKuSi11}) for $C^*$-bundles which can be obtained by replacing Hilbert spaces with Hilbert $C^*$-modules. Approaching the general problem, one may also realize by looking at the proof of Theorem \ref{theorem: local-projections}, that the theory of module frames \cite{FrLa99,FrLa02} will play a significant role when dealing with continuous-trace $C^*$-algebras on general $C^*$-bundles.

\bibliographystyle{myalphaurl-sortbyauthor}
\bibliography{master}
\linespread{1}
\Addresses
\end{document}